\documentclass[11pt]{amsart}
\usepackage[left=1in,right=1in,bottom=1in,top=1in]{geometry}
\usepackage{amsmath}
\usepackage{amsfonts}
\usepackage{amssymb}
\usepackage{amsthm}
\usepackage{young}
\usepackage[vcentermath]{youngtab}
\usepackage{tikz}
\usetikzlibrary{shapes.misc}
\usepackage{graphicx}
\usepackage{caption}

\tikzset{cross/.style={cross out, draw=black, minimum size=2.5*(#1-\pgflinewidth), inner sep=2pt, outer sep=0.5pt},
	cross/.default={1pt}}
\usepackage{xcolor}
\usetikzlibrary{calc,arrows}

\usepackage{algorithm}
\usepackage[noend]{algpseudocode}
\usepackage{amsmath, amssymb, amsthm}
\usepackage{blindtext}
\usepackage[pdftex,bookmarks=true]{hyperref}
\usepackage{cleveref}
\usepackage{comment}
\usepackage{enumerate}
\usepackage{fullpage}
\usepackage{mathtools}
\usepackage{subcaption}

\usepackage{enumitem}
\setlist[enumerate]{itemsep=.3mm}
\setlist[itemize]{itemsep=.3mm}

\newcommand{\Z}{\mathbb Z}
\newcommand{\E}{\mathbb E}
\renewcommand{\P}{\mathbb P}
\newcommand{\Pcp}{\mathbb{P}_{\textsc{cp}}}
\newcommand{\Pgw}{\mathbb{P}_{\textsc{gw}}}
\newcommand{\Bin}{\textbf{Bin}}
\newcommand{\ep}{\varepsilon}
\renewcommand{\deg}{\text{deg}}
\newcommand{\one}{\textnormal{\textbf{1}}}
\newcommand{\zero}{\textnormal{\textbf{0}}}

\newcommand{\cp}{\textsf{\small{\textbf{CP}}}}

\newcommand{\gw}{\textsf{\footnotesize{\textbf{GW}}}}
\newcommand{\GW}{\textsf{\footnotesize{\textbf{GW}}}}
\newcommand{\gwc}{\textsf{\footnotesize{\textbf{GWC}}}}
\newcommand{\egw}{\textsf{\footnotesize{\textbf{EGW}}}}
\newcommand{\subrg}{\textsc{rg}}
\newcommand{\subcp}{\textsc{cp}}
\newcommand{\subgw}{\textsc{gw}}

\newcommand{\A}{\mathcal A}
\newtheorem{theorem}{Theorem}[section]
\newtheorem{lemma}[theorem]{Lemma}

\newtheorem{proposition}[theorem]{Proposition}
\newtheorem{corollary}[theorem]{Corollary}
\newtheorem{claim}[theorem]{Claim}

\newtheorem{thm}{Theorem}[section]

\newtheorem{conj}[thm]{Conjecture}
\newtheorem{cor}[thm]{Corollary}

\numberwithin{equation}{section}
\theoremstyle{definition}
\newtheorem{definition}[theorem]{Definition}
\newtheorem{remark}[theorem]{Remark}
\newtheorem*{remark-non}{Remark}
\newcommand{\G}{\mathcal G}

\newcommand{\whp}{\textsf{whp}}
\newcommand{\T}{\mathcal T}

\newcommand{\boundellipse}[3]
{(#1) ellipse (#2 and #3)
}

\begin{document}
	\title{\textbf{Critical value asymptotics for the contact process on random graphs}}
		\author{Danny Nam}
	\address{Department of Mathematics, Princeton University, Princeton, New Jersey 08544, USA}
	\email{dhnam@princeton.edu}
	\thanks{DN is supported by a Samsung scholarship}

	\author{Oanh Nguyen}
	\email{onguyen@princeton.edu}
	
	\author{Allan Sly}
	\email{asly@princeton.edu}
	\thanks{AS is supported by NSF grant DMS-1352013, Simons Investigator grant and a MacArthur Fellowship}
 	\maketitle
	
	\begin{abstract} 
		Recent progress in the study of the contact process \cite{BNNASurvival} has verified that the extinction-survival threshold $\lambda_1$ on a Galton-Watson tree is strictly positive if and only if the offspring distribution $\xi$ has an exponential tail.
		In this paper, we derive the first-order asymptotics of $\lambda_1$ for the contact process on Galton-Watson trees and its corresponding analog for random graphs.  In particular, if $\xi$ is appropriately concentrated around its mean, we demonstrate that $\lambda_1(\xi) \sim 1/\mathbb{E} \xi$ as $\mathbb{E}\xi\rightarrow \infty$, which matches with the known asymptotics on the $d$-regular trees. The same result for the short-long survival threshold on the Erd\H{o}s-R\'enyi and other random graphs are shown as well.
		
\end{abstract}
\setcounter{tocdepth}{1}
\tableofcontents

\section{Introduction}
The contact process is a model of the spread of disease. For a graph $G = (V, E)$, the contact process on $G$ with infection rate $\lambda$ and recovery rate $1$ is a continuous-time Markov chain, in which a vertex $v$ is either infected ($X_t(v)=1$) or healthy ($X_t(v)=0$). The process evolves according to the following rules.
\begin{itemize}	
	\item Each infected vertex infects each of its neighbors independently at rate $\lambda$ and is healed at rate $1$.
	
	\item Infection and recovery events in the process happen independently.
\end{itemize}

The contact process was first introduced by a work of Harris \cite{harris74} in which he studied the process on the lattice $\Z^{d}$. Among other things, he studied the phase diagrams of the contact process which since then has attracted intensive research. For an infinite rooted graph $G$, there are three phases that are of particular interest:
\begin{itemize}
	\item (\textit{Extinction}) the infection becomes extinct in finite time almost surely;
	
	\item (\textit{Weak survival}) the infection survives forever with positive probability, but the root is infected only finitely many times almost surely;
	
	\item (\textit{Strong survival}) the infection survives forever and the root gets infected infinitely many times with positive probability.
\end{itemize}
Denote the extinction-survival threshold by
$$\lambda_1(G) = \inf\{\lambda: (X_t) \text{ survives with positive probability}\}$$
and the weak-strong survival threshold by
$$\lambda_2(G) = \inf\{\lambda: (X_t) \text{ survives strongly}\}.$$

For the lattice, when the origin is initially infected, it is well known that there is no weak survival phase, that is, $0<\lambda_1(\Z^{d}) = \lambda_2(\Z^{d})<\infty$ (see \cite{harris74}, Bezuidenhout-Grimmett \cite{bezuidenhout1990critical}, and also the books of Liggett \cite{liggett:sis}, \cite{liggett:ips} and the references therein). On the other hand, for the infinite $d$-regular tree $\mathbb T_{d}$ with $d\ge 3$, we have that the contact process with the root initially infected has two distinct phase transitions with $0<\lambda_1(\mathbb T_{d})<\lambda_2(\mathbb{T}_d)<\infty$, by a series of beautiful work by Pemantle \cite{p92} (for $d\ge 4$), Liggett \cite{l96} (for $d=3$), and Stacey \cite{s96} (for a shorter proof that works for all $d\ge 3$). Moreover, we have from \cite{p92}  that
\begin{equation}\label{phase:d:tree}
\frac{1}{d-1}\le \lambda_1(\mathbb T_{d})<\frac{1}{d-2} \quad\text{and}\quad \frac{1}{2\sqrt{d-1}}\le \lambda_2(\mathbb T_{d}). 
\end{equation}
In particular, the first-order asymptotics of $\lambda_1(\mathbb{T}_d)$ is $1/d$ as $d$ becomes large.

Much less is known about the contact process on random trees. First of all, for a Galton-Watson tree $\T$ with offspring distribution $\xi$, it is not difficult to see that $\lambda_1(\T)$ and $\lambda_2(\T)$ are constants which are the same for a.e. $\T\sim \gw(\xi)$ conditioned on $|\T|=\infty$, and hence the constants $\lambda_1^{\textsc{gw}}(\xi)$ and $\lambda_2^{\textsc{gw}}(\xi)$ are well-defined. Huang and Durrett \cite{hd18} proved that on  $\T\sim \GW(\xi)$ with the root initially infected, $\lambda_2^{\textsc{gw}}(\xi)=0$ if the offspring distribution $\xi$ is subexponential, i.e., $\E e^{c\xi}=\infty$ for all $c>0$. So in this case, there is only the strong survival phase.

By contrast, if the offspring distribution $\xi$  has an exponential tail, i.e., $\E e^{c\xi} <\infty$ for some $c>0$, Bhamidi and the authors \cite{BNNASurvival} showed that there is an extinction phase: $\lambda_1^{\textsc{gw}}(\xi)\geq \lambda_0(\xi)$  for some constant $\lambda_0(\xi)>0$. Our first main result derives the first-order asymptotics on $\lambda_1^{\textsc{gw}}(\xi)$ for $\xi$ \textit{concentrated} around its mean, which turns out to have the same form as (\ref{phase:d:tree}).

\begin{thm}  \label{thm:phasetransition:tree}
	Let $\{\xi_k\}$ be a sequence of nonnegative integer-valued random variables with $\E \xi_k=:d_k \to \infty$ as $k\rightarrow \infty$. Assume that there exists a collection of positive constants $\mathfrak{c}=\{c_\delta  \}_{\delta \in (0,1]} $ for which 
	\begin{equation}\label{eq:concentration condition}
	\begin{split}
	\P(\xi_k \geq (1+\delta)d_k  ) &\leq \exp(- c_\delta d_k) ~~~~\textnormal{for all }\delta\in(0,1);\\
	\P(\xi_k \geq (1+a)d_k ) &\leq \exp(-c_1ad_k)~~~~\textnormal{for all }a\geq 1.
	\end{split}
	\end{equation}
	for all large enough $k$. Consider the contact process on the Galton-Watson tree $\mathcal T_k\sim \gw (\xi_k)$ with the root initially infected. Then, the extinction-survival threshold $\lambda_1^{\textsc{gw}}(\xi_k)$ satisfies
	\begin{equation}\label{eq:thm1 concl}
	\lim_{k\to\infty}\lambda_1^{\textsc{gw}} (\xi_k) d_k = 1.
	\end{equation}
\end{thm}

\begin{remark}
	The concentration condition \eqref{eq:concentration condition} we impose resembles the form of large deviation estimates. Notice that the family of Poisson distributions $\{\textnormal{Pois}(d) \}_{d>0}$ satisfies (\ref{eq:concentration condition}). This fact allows us to deduce an analogous result for Erd\H{o}s-R\'enyi random graphs in Corollary \ref{cor:phasetransition:ER}.
	
\end{remark}

\begin{remark}
	The conclusion (\ref{eq:thm1 concl}) is not always true without assuming \eqref{eq:concentration condition}. In fact, for any given $d>0$, we can construct $\xi$ with $\E \xi =d$ and arbitrarily small $\lambda_1^{\textsc{gw}} (\xi)>0$ (compared to $1/d$), by truncating a heavy-tailed distribution at a large enough degree. Understanding the size of $\lambda_1^{\textsc{gw}}(\xi)$ in full generality seems to be a challenging problem.
\end{remark}

In \cite{BNNASurvival}, the proof of $\lambda_1^{\textsc{gw}}(\xi)>0$ introduced a new method recursive analysis on Galton-Watson trees that controlled the expected survival times, but the  quantitative lower bounds on $\lambda_1$ they deduced were far from being sharp. Our main contribution is to introduce an alternative tree recursion, and develop techniques to control the tail probabilities of the survival time over the level of Galton-Watson trees in addition to its expectation. A detailed outline is given in Section \ref{subsec:idea1}.

A naturally related object is the contact process on a random graph with a given degree distribution. Let $\mu$ be a degree distribution and  $G_n\sim \mathcal{G}(n,\mu)$ be a random graph with degree distribution $\mu$, assuming the giant component condition \eqref{eq:condition on mu} (for details, see Section \ref{subsubsec:rgprelim}). Consider the contact process on $G_n\sim \mathcal{G}(n,\mu)$ where all vertices are initially infected. In \cite{BNNASurvival}, it was shown that if  $\E_{D\sim \mu} e^{cD}<\infty$ for some constant $c>0$, then there exist constants $0<\underline{\lambda}(\mu) \le\overline{\lambda}(\mu) <\infty$ such that the survival time $T_{\lambda,n}$  of the process satisfies the following:
\begin{enumerate}
	\item [(1)] For all $\lambda <\underline{\lambda}$,  $T_{\lambda,n}\leq n^{1+ o(1)}$ \textsf{whp};
	
	\item [(2)] For all $\lambda> \overline{\lambda}$, $T_{\lambda,n}\geq e^{\Theta(n)}$ \textsf{whp}.
\end{enumerate}

On the other hand if $\mu$ has a subexponential tail, they proved that \textsf{whp} there is no short survival phase. Based on this result, we formally define the short- and long-survival thresholds $\lambda_c^-(\mu)$, $\lambda_c^+(\mu)$ as follows.
\begin{equation}\label{eq:def:rgthres}
\begin{split}
\lambda_c^-(\mu) &:= \lim_{\alpha\to \infty}  \sup \Big\{\lambda: \;\lim_{n\to \infty}\P (T_{\lambda,n} \leq n^\alpha) = 1 \Big\};\\
\lambda_c^+(\mu) &:= \lim_{\beta \to 0} 
\;\inf \left\{\lambda: \;\; \lim_{n\rightarrow \infty}\P(T_{\lambda,n} \geq e^{\beta n} )=1 \right\}.
\end{split}
\end{equation}

\noindent The second result of the paper verifies the first-order asymptotics for $\lambda_c^-(\mu)$ and $\lambda_c^+(\mu)$ which have the same form as Theorem \ref{thm:phasetransition:tree}.
\begin{thm} \label{thm:phasetransition:graph}
	Let $(\mu_k)$ be a sequence of degree distributions with  the size-biased distribution $\widetilde{ \mu}_k$ (see (\ref{eq:def:sizebiased}) for its precise definition). Suppose that the mean $d_k$ of $\widetilde{ \mu}_k$ tends to infinity as $k\rightarrow \infty$. Moreover, assume that $(\widetilde{ \mu}_k)$ satisfies  the concentration condition (\ref{eq:concentration condition}) for fixed positive constants  $\mathfrak{c}=\{c_\delta\}_{\delta\in(0,1]}$.  Then, the short- and long-survival thresholds $\lambda_{c}^-(\mu_k)$, $\lambda_c^+(\mu_k)$ of the contact process on $G_n\sim \mathcal{G}(n,\mu_k)$ satisfy
	\begin{equation}\label{key}
	\lim_{k\to\infty} \lambda_c^-(\mu_k) d_k=\lim_{k\to\infty} \lambda_c^+(\mu_k) d_k = 1. \nonumber
	\end{equation}
\end{thm}

A proof of $\lambda_c^-(\mu)>0$ for $\mu$ with an exponential tail was given in \cite{BNNASurvival}, which relied on estimating the probability of having an infection deep inside Galton-Watson trees which are local weak limits of the random graphs---see Section \ref{subsubsec:rgprelim} for details. However, we will see in Section \ref{subsec:idea2} that controlling such an event is insufficient for Theorem \ref{thm:phasetransition:graph} since $\lambda$ is not small enough. To overcome this issue, we first take the expectation of the event over the level of the contact process, and then study its tail probability over the level of Galton-Watson trees. This new approach turns out to provide a substantial improvement from \cite{BNNASurvival}. On the other hand, we will soon see the generalized result on $\lambda_c^+$ in Theorem \ref{thm:supercritical:general} below, and it requires a different approach in spreading infections and an improved structural analysis of random graphs than in \cite{BNNASurvival}, which we overview in Section \ref{subsec:idea3}.

We can deduce an analog of Theorem \ref{thm:phasetransition:graph} for the Erd\H{o}s-R\'enyi random graph, arguably one of the most well-known models of random graphs, based on the contiguity between $\mathcal{G}_{\textsc{er}}(n,d/n)$ and $\mathcal{G}(n,\textnormal{Pois}(d))$ (see Section \ref{subsubsec:rgprelim} for details).
\begin{cor} \label{cor:phasetransition:ER}
	Consider the contact process on the random graph $G_n\sim \G_{\textsc{er}}(n, d/n)$ with all vertices initially infected. Then, the short- and long- survival thresholds of the process, defined analogously as (\ref{eq:def:rgthres}), satisfy
	\begin{equation}\label{key}
	 \lim_{d\to\infty}\lambda_\textsc{er}^- (d) d=\lim_{d\to\infty}\lambda_\textsc{er}^+ (d) d = 1. \nonumber
	\end{equation}
\end{cor}

As mentioned in \cite{BNNASurvival}, we actually expect the transition from  polynomial- to  exponential-time survival is sharp and happens at the extinction-survival threshold of the corresponding Galton-Watson tree. Namely, 

\begin{conj}\label{conjecture} 
	Let $\mu$ be a degree distribution satisfying the giant component condition (\ref{eq:condition on mu}), and let $\widetilde{\mu}$ be its size-biased distribution (definition given in (\ref{eq:def:sizebiased})). Recalling (\ref{eq:def:rgthres}), we have
	\begin{equation} 
	\lambda_c^-( \mu) = \lambda_c^+ (\mu)= \lambda_1^{\textsc{gw}}( \widetilde{\mu}).\nonumber
	\end{equation}
\end{conj}

The special case of random regular graphs $(\mu\equiv d)$ was established by  Mourrat-Valesin \cite{mv16} and Lalley-Su \cite{ls17} who showed that  the short-long survival threshold for random $d$-regular graph occurs exactly at $\lambda_1(\mathbb T_{d})$.

The next result establishes one inequality of the conjecture, by showing that the intensity which gives a supercritical contact process on a Galton-Watson tree implies an exponential time survival on the corresponding random graph. 

\begin{thm} \label{thm:supercritical:general}
	Let $\mu$ be a degree distribution satisfying the giant component condition \eqref{eq:condition on mu}, and let $\widetilde{\mu}$ be its size-biased distribution with $d:=\E_{D\sim\widetilde{\mu}}D$. Recalling (\ref{eq:def:rgthres}), we have
	\begin{equation}\label{eq:thm:super:ER}	
	\lambda_c^+( \mu) \le \lambda_1^{\textsc{gw}}(\widetilde{ \mu})\le \frac{1}{d-1}. 
	\end{equation}
\end{thm}


In the case of $\lambda>\lambda_1^{\textsc{gw}}(\widetilde{ \mu})$, one may ask if the survival time of the contact process on $G_n\sim \mathcal{G}(n,\mu)$ is still exponentially long, even with a single initial infection. The following theorem gives an affirmative answer to this question.
\begin{thm}\label{thm:super:onevertex}
	For all $\lambda>\lambda_1^{\textsc{gw}}(\widetilde{ \mu})$, \textsf{whp} in $G_n\sim \mathcal{G}(n,\mu)$, the contact process on $G_n$ starting from a single infection at a site chosen uniformly at random survives until time $e^{\Theta(n)}$ with positive probability. Moreover, the same holds for the  Erd\H{o}s-R\'enyi random graph	$G_n\sim \G_{\textsc{er}}(n,d/n)$, when $\lambda>\lambda_1^{\textsc{gw}}(\textnormal{Pois}(d))$.
\end{thm}

\section{Preliminaries}\label{subsec:prelim}
In this section, we set up the notations and review preliminary  concepts on the contact process and random graphs.
\subsection{The contact process and its graphical representation}\label{subsubsec:cpprelim}

Let $G=(V,E)$ be a graph (finite or infinite) and $A\subset V$. The configuration space of infections is $\{0,1 \}^V$, and  $\one_A\in \{0,1\}^V$ is the configuration of which the vertices in $A$ are infected and the others are healthy.  We denote the contact process with initial infections at $A$ by
$$(X_t) \sim \cp^\lambda(G; \one_A), $$
where $\lambda$ is the intensity of infection and $\one_A\in \{0,1\}^V$ describes the initial condition. We sometimes write $\cp^\lambda(G)$ when the initial condition is irrelevant in the context. Also, $\zero, \one$ denote all-healthy, all-infected configurations, respectively, and we write $\one_v$ for $\one_{\{v\}}$. The transition rule of the continuous-time Markov chain $(X_t)\sim \cp^\lambda(G;\one_A)$ can be defined as follows:

\begin{itemize}
	\item $X_t$ becomes $X_t - \one_v$ with rate $1$ for each $v$ such that $X_t(v)=1$.
	
	\item $X_t$ becomes $X_t + \one_u$ with rate $\lambda N_t(u)$ for each $u$ with $X_t(u)=0$, where $N_t(u):=\sum_{v \sim u}X_t(v)$ denotes the number of infected neighbors of $u$ at time $t$.
\end{itemize}

The dynamics of the contact process can be interpreted by the \textit{graphical representation} which provides a convenient coupling of the process. We briefly discuss this notion following Chapter 3, Section 6 of \cite{liggett:ips}. Let $\{N_v(t)\}_{v\in V}$ (resp., $\{N_{\vec{uv}} (t) \}_{\vec{uv} \in \vec{E}}$) be the family of independent Poisson processes with rate 1 (resp., rate $\lambda$), where $\vec{E}= 2E:= \{\vec{uv}, \vec{vu}: (uv)\in E \}$ is the set of directed edges. We set $\{N_v(t)\}_{v\in V}$ to be independent of $\{N_{\vec{uv}}(t)\}_{\vec{uv}\in\vec{E}}$ as well. Note that all  event times of the Poisson processes are distinct almost surely. We generate the graphical representation of $\cp^\lambda(G;\one_A)$ as follows:

\begin{enumerate}
	\item [1.] Initially, we have the empty space-time domain $V\times \mathbb{R}_+$.
	
	\item [2.] For each $v\in V$, place symbol $\times$ at $(v,t)$, for each event time $t$ of $N_v$. The symbol $\times$ describes the time when $v$ gets recovered.
	
	\item [3.] For each $\vec{uv}\in \vec{E}$, place an arrow from $(u,t)$ to $(v,t)$, for each event time $t$ of $N_{\vec{uv}}$. The arrow indicates that the infection is passed from $u$ to $v$ at time $t$ if $X_t(u)=1$.    
\end{enumerate}

\begin{figure}
	\centering
	\includegraphics[width=58mm]{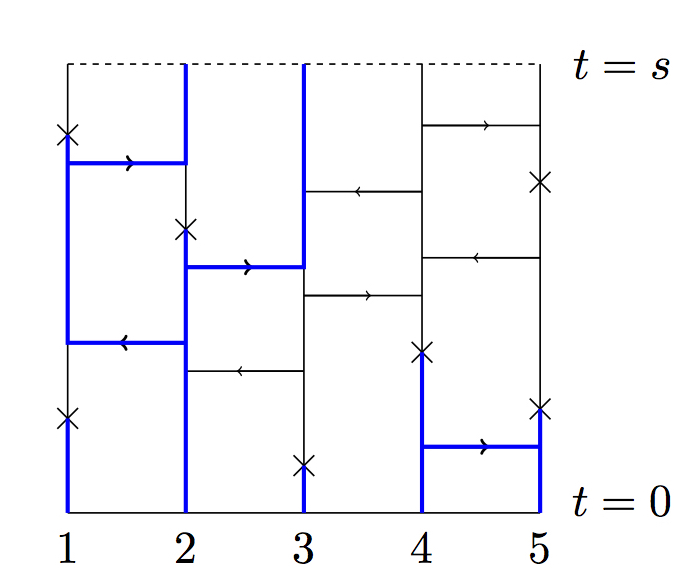}
	\caption{An instance of the graphical representation of the contact process on $V=\{1,\ldots,5\}$ with initial configuration $\one_V$. Note that $X_s = \one_{\{2,3\}}$. \label{fig1}}
\end{figure}

Therefore, as described in Figure \ref{fig1}, we can read off the diagram starting from the bottom horizontal line and obtain $(X_t)$. Construction of the graphical representation will play an important role in Sections \ref{subsec:unic prop small t2} and \ref{subsec:pf of thm2} when we introduce a \textit{decomposition} of the contact process.

\subsection{Random graphs and their limiting structure}\label{subsubsec:rgprelim}

Let $\mu$ be a probability distribution on $\mathbb{N}$. The random graph $G_n \sim \mathcal{G}(n,\mu)$ with degree distribution $\mu$ is defined as follows:

\begin{itemize}
	\item Let $d_i \sim $ i.i.d.$\;\mu$ for $i=1,\ldots, n$, conditioned on $\sum_{i=1}^n d_i \equiv 0$ mod $2$. The numbers $d_i$ refer to the number of \textit{half-edges} attached to vertex $i$.
	
	\item Generate the graph $G_n$ by pairing all half-edges uniformly at random. 
\end{itemize}
The resulting graph $G_n$ is also called the \textit{configuration model}. One may also be interested in the \textit{uniform model} $G^{\textsf{u}}_n \sim \mathcal{G}^{\textsf{u}}(n,\mu)$, which picks a uniformly random simple graph among all simple graphs with degree sequence $\{d_i\}_{i\in[n]} \sim $ i.i.d.$\;\mu$. It is well-known that if $\mu$ has a finite second moment, then the two laws $\mathcal{G}(n,\mu)$ and $\mathcal{G}^{\textsf{u}}(n,\mu)$ are \textit{contiguous}, in the sense that for any subset $A_n$ of graphs with $n$ vertices,
\begin{equation*}
\P_{G_n\sim \mathcal{G}(n,\mu) } \left(G_n \in A_n \right) \rightarrow 0
\quad
\textnormal{implies}
\quad
\P_{G^{\textsf{u}}_n\sim \mathcal{G}^{\textsf{u}}(n,\mu) } \left(G^{\textsf{u}}_n \in A_n \right) \rightarrow 0.
\end{equation*}
For details, we refer the reader to Chapter 7 of \cite{vanderhofstad17} or to \cite{j09}. We remark that when $\mu=\textnormal{Pois}(d)$, the random graph $G_n\sim \mathcal{G}(n,\mu)$ is contiguous to the Erd\H{o}s-R\'enyi random graph $G_n^{\textsc{er}}\sim \mathcal{G}_{\textsc{er}} (n,d/n)$ as shown in \cite{k06}, Theorem 1.1.

Furthermore, it is also well-known that the random graph $G_n \sim \mathcal{G}(n,\mu)$ is \textit{locally tree-like}, and the local neighborhoods converge \textit{locally weakly} to Galton-Watson trees. To explain this precisely, let us denote the law of Galton-Watson tree with offspring distribution $\mu$ by $\gw(\mu)$, and let $\gw(\mu)^l$ be the law of $\gw(\mu)$ truncated at depth $l$, that is, the vertices with distance $>l$ from the root are removed. Further, let $\widetilde \mu$ denote the \textit{size-biased} distribution of $\mu$, defined by
\begin{equation}\label{eq:def:sizebiased}
\widetilde{\mu}(k-1) := \frac{k\mu(k)}{\sum_{k'=1}^\infty k'\mu(k') }, \quad k=1,2,\ldots .
\end{equation}
Note that if $\mu=$Pois$(d)$, then $\widetilde{\mu}= \mu$. Lastly, define $\gw(\mu, \widetilde{\mu})^l$ to be the Galton-Watson process truncated at depth $l$, such that the root has offspring distribution $\mu$ while all other vertices have offspring distribution $\widetilde{\mu}$. Then the following lemma shows the convergence of local neighborhoods of $G_n$.

\begin{lemma}[\cite{DemboMontanari2010}, Section 2.1]\label{lem:lwc} Suppose that $\mu$ has a finite mean. Let $l>0$ and let $v$ denote the vertex in $G_n\sim \mathcal{G}(n,\mu)$ chosen uniformly at random. Then for any rooted tree $(T, x )$ of depth $l$, we have
	\begin{equation*}
	\lim_{n\rightarrow \infty}\P((N(v,l),v) \cong (T,x)  ) =
	\P_{(\mathcal{T},\rho )\sim \gw(\mu,\widetilde{\mu})^l} ((\mathcal{T},\rho) \cong (T,x)  ),
	\end{equation*} 
	where $N(v,l)$ is the $l$-neighborhood of $v$ in $G_n$ and $\cong$ denotes the  isomorphism of rooted graphs. We say that $G_n$ converges locally weakly to $\gw(\mu,\widetilde{\mu})$.
\end{lemma}

\noindent We remark that the same holds for a fixed vertex $v\in G_n$. Moreover, we also stress that the condition for $\gw(\mu, \widetilde{\mu})$ to be supercritical is equivalent to the condition for $\mathcal{G}(n,\mu)$ to have the unique giant component \textsf{whp} (see e.g., \cite{mr95}, or \cite{durrett:rgd}, Section 3 for details), which can be addressed as 
\begin{equation}\label{eq:condition on mu}
\E_{D\sim\mu} D(D-2) >0,
\end{equation}

\subsection{Notations}
For a tree $T$ and a depth $l$, we denote by $T_{l}$ the set of vertices of $T$ at depth $l$. We use $T^{l}$ and $T_{\le l}$ to denote the set of vertices of depth at most $l$. In particular, $\T^l\sim\gw(\xi)^l$ denotes the Galton-Watson tree generated up to depth $l$, while the infinite Galton-Watson tree is denoted by $\T\sim\gw(\xi)$.

Throughout the paper, we often work with the contact process defined on a (fixed) graph generated at random.  To distinguish between the two randomness of different nature, we introduce the following notations: 
\begin{itemize}
	\item $\Pcp$ and $\E_{\textsc{cp}}$ denote the probability and the expectation, respectively, with respect to the randomness from  contact processes
	
	\item $\Pgw$ and $\P_\textsc{rg}$ denote the probability   with respect to the randomness from the underlying graph, when the graph is a Galton-Watson tree and  a random graph $\mathcal{G}(n,\mu)$, respectively. We write $\E_\textsc{gw}$ and $\E_\textsc{rg}$ similarly for expectations.
	
	\item $\P$ and $\E$ denote the probability and expectation, respectively, with respect to the combined randomness over both the process and the graph. That is, for instance, $\E[\cdot] = \E_\textsc{gw} [\E_\textsc{cp} [\cdot] ]$, if the underlying graph is a Galton-Watson tree.
\end{itemize}

\section{Main concepts and ideas}\label{sec:proofoutline} 
Let us start by emphasizing that even though we borrow some ideas and notations from \cite{BNNASurvival}, this manuscript is self-contained and so the reader does not have to be familiar with \cite{BNNASurvival} to read this manuscript. In this section, we briefly introduce the primary notions and discuss the main ideas in the paper. We also address the organization of the rest of the article in Section \ref{subsec:organ}.
 
\subsection{The root-added process and the lower bound of Theorem \ref{thm:phasetransition:tree}}\label{subsec:idea1}

In \cite{BNNASurvival}, Bhamidi and the authors  studied the  \textit{root-added contact process}  to prove $\lambda_1>0$ on $\gw(\xi)$ with $\xi$ having an exponential tail.  This notion continues to play a huge role in the current work as well, and hence we begin with explaining its definition and the concept of \textit{excursion time}.

\begin{definition}[Root-added contact process, \cite{BNNASurvival}]\label{def:rcp}
	Let $T$ be a (finite or infinite) tree rooted at $\rho$. Let $T^+$ be the tree that has a parent vertex $\rho^+$ of $\rho$ which is connected only with $\rho$.  The \textit{root-added contact process} on $T$ is the continuous-time Markov chain on the state space $\{0,1\}^T$, defined as the contact process on $T^+$ with  $\rho^+$ set to be infected permanently (hence we exclude $\rho^+$ from the state space). That is, $\rho^+$ is infected initially, and it does not have a recovery clock attached to itself. Let $\cp^\lambda_{\rho^+}(T;x_0)$ denote the root-added contact process on $T$ with initial condition $x_0 \in \{0,1 \}^T $. Note that the root-added contact process no longer has an absorbing state.
\end{definition}

\begin{definition}[Survival and excursion times]\label{def:excursion and survival time}
	Let $T$ be a (finite or infinite) tree rooted at $\rho$. The \textit{survival time} and \textit{excursion time} on $T$, denoted by $\textbf{R}(T)$ and $\textbf{S}(T)$, respectively, are defined as follows:
	\begin{itemize}
		\item $\textbf{R}(T)$ is the first time when the contact process $\cp^\lambda (T;\one_\rho)$ is all-healthy (i.e., when the process terminates). We also denote the \textit{expected survival time} by $R(T) =\E_{\textsc{cp}} \textbf{R}(T)$.
		
		\item $\textbf{S}(T)$ is the first time when the root-added contact process $\cp^\lambda_{\rho^+} (T; \one_\rho)$ becomes all-healthy on $T$. We also denote the \textit{expected excursion time} by $S(T) = \E_{\textsc{cp}} \textbf{S}(T)$. 
	\end{itemize}
	Note that the quantities $R(T)$ and $S(T)$ are fixed numbers for each tree $(T,\rho)$ and satisfy $R(T)\leq S(T)$, which can be seen through the coupling via the graphical representation.
\end{definition}


The previous work \cite{BNNASurvival} established $\lambda_1 >0$ by a recursive inequality on the depth of the tree that showed $\E \textbf{S}(\mathcal{T}) <\infty$  for $\mathcal{T}\sim\gw(\xi)$ with $\xi $ having an exponential tail decay.  However, the argument had limitations since it could only deal with a small enough $\lambda$. In order to push its applicability to near-criticality, in Section \ref{sec:subcritical trees} we introduce another recursive inequality  based on fundamental properties of the contact process. Using the two different recursions, we can bound the tail probabilities of the \textit{expected} excursion time, namely, $\P_{\textsc{gw}}(S(\mathcal{T}) \geq t)$, and then the bound easily implies $\E_\textsc{gw} S(\mathcal{T}) <\infty$. This is a substantial improvement from \cite{BNNASurvival} where we could only control its expectation $\E_\textsc{gw} S(\mathcal{T})$ for small enough $\lambda$.

\subsection{Deep infections, unicyclic neighbors and the lower bound of Theorem \ref{thm:phasetransition:graph}}\label{subsec:idea2}

To establish Theorem \ref{thm:phasetransition:graph}, we attempt to generalize Theorem \ref{thm:phasetransition:tree} based on the fact that the local neighborhoods of $\mathcal{G}(n,\mu)$ look like Galton-Watson trees. There are two major obstacles on carrying out this idea.

\begin{enumerate}
	\item[1.] For a vertex $v$ in $G\sim\mathcal{G}(n,\mu)$, its local neighborhood $N(v,l)$ contains a lot of cycles if $l \geq c_\mu \log n$ for some constant $c_\mu$. 
	
	\item[2.] Even for small $l$, there are $o(n)$  vertices that contain a cycle in $N(v,l)$.
\end{enumerate}  

\subsubsection{Deep infections}\label{subsubsec:idea2-1}

To overcome the first issue, we  show that the probability of having a deep infection of depth   $\geq c_\mu \log n$ inside $\cp^\lambda(\mathcal{T},\one_\rho)$ is very small for a Galton-Watson tree $\mathcal{T}$. This leads to the consideration of  the \textit{total  infections at leaves} of (finite) trees defined as follows.

\begin{definition}[Total infections at leaves]\label{def:tit}
	Let $T$ be a finite tree rooted at $\rho$, set $l := \max\{\textnormal{dist}(\rho, v): v\in T \}$ be the depth of the tree and $$\mathcal{L}:=\{v\in T: \textnormal{dist}(\rho,v)=l \}  $$
	be the collection of depth-$l$ leaves of $T$. Suppose that $l\geq 1$ and consider the root-added contact process $(X_t)\sim \cp^\lambda_{\rho^+}(T;\one_\rho)$. For $v\in \mathcal{L}$, define the \textit{total infections at} $v$, by 
	\begin{equation*}
	\textbf{M}^l_v(T):=\,\textnormal{the number of infections at } v \textnormal{ in }(X_t) \textnormal{ during time } t\in[0,\textbf{S}(T)],
	\end{equation*}
	where  $\textbf{S}(T)$ is the excursion time of $(X_t)$. In other words, we count the number of times $t$ such that $X_t(v)=1$ and $X_{t-}(v)=0$ for $t\leq \textbf{S}(T)$. Then,
	we define  the \textit{total infections at depth-$l$ leaves} (\textit{during a single excursion}) by
	\begin{equation*}
	\textbf{M}^l(T) = \sum_{v \in \mathcal{L}} \textbf{M}^l_v(T).
	\end{equation*}
	For $l'>l$, we set $\textbf{M}^{l'}(T)\equiv 0$.
	
	We also denote the \textit{expected total  infections at depth-$l$ leaves} by $M^l(T) = \E_{\textsc{cp}} \textbf{M}^l(T)$. Also, as above, we write $M^{l'}(T)=0$ for $l'>l$. Moreover, if the tree depth is $0$  (that is, $T$ is a single vertex), we set $\textbf{M}^0(T) \equiv 1$.
\end{definition}

The previous work \cite{BNNASurvival} derived an exponential decay of $\textbf{M}^l(\mathcal{T}^l)$ in $l$ for $\mathcal{T}^l\sim \gw(\widetilde{\mu})^l$ to deal with the same issue, but as before it had to require $\lambda$ to be small enough. However, unfortunately, the decay of $\textbf{M}^l(\mathcal{T}^l)$ is insufficient for our purpose if $\lambda$ is close to $\lambda_1$, due to the reason we explain below. 

If $ \lambda = (1-\ep)d^{-1}$ with $d=\E_{D\sim\widetilde{\mu}} D$, then for an infected vertex $v$ in $\mathcal{T}^l \sim \gw(\widetilde{\mu})^l$, the expected number of offsprings of $v$ that get infected before $v$ becomes healthy is
$$\frac{\lambda d}{\lambda+1}  \approx \frac{(1-\ep)d}{d+1}, \quad \textnormal{hence}\quad  \E[\textbf{M}^l(\mathcal{T}^l)] \geq  \left(\frac{(1-\ep)d}{d+1}\right)^l,$$
and intuitively, the latter quantity essentially corresponds to $\P (\textbf{M}^l(\mathcal{T}^l)\geq 1)$. To apply a union bound over all vertices in $G_n\sim \mathcal{G}(n,\mu)$, we need this probability to be of order $o(n^{-1})$. That is, we roughly require
$$ l \gg \frac{\log n}{\log\{(1+\ep)(1+d^{-1}) \} }. $$
This is much larger than our budget $c_\mu \log n$ which is approximately $c \log_d n$. Thus, investigating the tail probability of $\textbf{M}^l(\mathcal{T}^l)$ is not enough for our purpose. However, in \cite{BNNASurvival}, this approach was sufficient since we could set $\lambda$ as small as we wanted.

In Section \ref{sec:numberofhits}, we instead focus on studying $\Pgw( M^l(\mathcal{T}^l) \geq t)$, which turns out to have a much better bound than the tail of $\textbf{M}^l(\mathcal{T}^l)$. Similarly as explained in Section \ref{subsec:idea1}, we derive two different recursive inequalities on $M^l(T)$ for a deterministic tree $T$, and prove its tail bound for the case of Galton-Watson trees.

\subsubsection{Unicyclic neighbors}\label{subsubsec:idea2-2}

Another major issue is to deal with the neighborhoods $N(v, l)$ in $G_n\sim \mathcal{G}(n,\mu)$ containing a cycle. We rely on idea as \cite{BNNASurvival}, by observing that if $\mu$ satisfies (\ref{eq:concentration condition}), then there exists $\gamma(\mathfrak{c})>0$ such that \textsf{whp, }$N(v, \gamma \log_d n)$ contains at most one cycle for all $v$ with $d$  as in the previous subsection (see Lemma \ref{lem:1cyc}). Therefore, we study $S(T')$ and $M^l(T')$ as above (precise definitions are given in Section \ref{subsec:unic}), for certain unicyclic graphs $T'$ which are closely related to Galton-Watson trees. To this end, we appropriately cover $T'$ by trees and deduce information on $S(T')$ and $M^l(T')$ from the results we obtained on trees. However, formalizing this idea requires a heavy technical work and it is presented in Appendix \ref{sec:unicyclic}.  Similar ideas are applied to studying the contact process on $G_n$. Roughly speaking, we decompose $G_n$ by its local neighborhoods $\{N(v,l)\}_{v\in G_n}$, and derive results on $\cp^\lambda(G_n)$ by using what we know on $N(v,l)$.

\subsection{The proof of Theorems \ref{thm:supercritical:general} and \ref{thm:super:onevertex}}\label{subsec:idea3}
The previous work \cite{BNNASurvival} settled that $\lambda_c^+(\mu)<\infty$, which was based on a challenging structural analysis on the configuration model. Roughly speaking, they showed the existence of an \textit{embedded expander}, a subset of large degree vertices in the random graph, on which it is easy to send infections from one vertex to another.
Upon establishing its existence, spreading infections on the embedded expander could then be done by a relatively straight-forward way, which was to infect a site at distance $l$ with probability $(c(1-e^{-\lambda}))^l$, since we could choose $\lambda$ to be large. One of the main difficulties in establishing the much improved bound $\lambda_c^+(\mu) \leq \lambda_1^{\textsc{gw}}(\widetilde{ \mu})$ is that we need to develop a more efficient way of sending infections from a vertex to another.

The key observation for such improvement is that if $\lambda>\lambda_1^{\textsc{gw}}(\widetilde{ \mu})$, the expected number of infections on $\mathcal{T}\sim \gw(\widetilde{ \mu})$ grows exponentially in time (Lemmas \ref{lm:good:tree} and \ref{lm:good:tree:1}). We use this property as our driving force of passing infections on the random graph, which is possible since the local neighborhoods look like Galton-Watson trees. This new method turns out to be substantially better than  the aforementioned approach. 

However, since we now need to reveal the neighborhoods to check if the infections spread well, the structural analysis on the random graph becomes even more involved than the previous proof in \cite{BNNASurvival}. We carry out by introducing an appropriate notion of \textit{good vertices}, which roughly refer to the sites that are capable of propagating enough infections around them, and showing that any set of $\delta n$ infected \textit{good vertices} causes $\ge 2\delta n$ \textit{good vertices} to be infected at a later time with high probability except for an exponentially small error.

\subsection{Organization of the article}\label{subsec:organ}

Sections \ref{sec:subcritical trees}--\ref{sec:shortsurvival randomgraphs} are  devoted to the derivation of the lower bounds of Theorems \ref{thm:phasetransition:tree} and \ref{thm:phasetransition:graph}. In Section \ref{sec:subcritical trees}, we introduce the basic form of the recursion argument on Galton-Watson trees and establish the lower bound of Theorem \ref{thm:phasetransition:tree}. In Section \ref{sec:numberofhits}, we extend the recursion criterion to the study of deep infections.  Section \ref{sec:shortsurvival randomgraphs} then concludes the proof of the lower bound of Theorem \ref{thm:phasetransition:graph}, while the technical works needed to study the unicyclic graphs are deferred to Appendix \ref{sec:unicyclic}.  Finally, we finish the proof of Theorems \ref{thm:phasetransition:tree} and \ref{thm:phasetransition:graph} by settling their upper bounds in Section \ref{sec:supercrit}.

\section{Survival and excursion  times on trees}\label{sec:subcritical trees}

In this section, we introduce primary recursive argument on the expected excursion time which are used throughout the paper. In Section \ref{subsec:deterministic recursion survival time}, we review some ideas developed in \cite{BNNASurvival}, and derive another recursive inequality on excursion times. In Sections \ref{subsec:tail estimate survival time} and  \ref{subsec:proof of recursive tail estim}, we prove a tail probability estimate and establish Theorem \ref{thm:phasetransition:tree} as its  application.

\subsection{Deterministic recursions on trees}\label{subsec:deterministic recursion survival time}

Let $T$ be a finite tree rooted at $\rho$ and
recall the definition of the root-added contact process $\cp^\lambda_{\rho^+}(T^+;x_0)$ (Definition \ref{def:rcp}). Let $R(T)$ and $S(T)$ be expected survival and recursion time as in Definition \ref{def:excursion and survival time}.

In $(T,\rho)$, let $D=\deg(\rho)$ and $v_1,\ldots, v_D$ be the children of $\rho$. Further, let $T_1,\ldots, T_D$ be the subtrees from each child of $\rho$, rooted at $v_1,\ldots, v_D$, respectively. In \cite{BNNASurvival}, we proved the following recursion on the excursion times.

\begin{proposition}[\cite{BNNASurvival}, Lemma 3.3]\label{prop:Srecursion atypical}
	Let $T$ and $T_1,\ldots, T_D$ be as above. Then, the expected excursion times $S(T)$ and $S(T_1),\ldots, S(T_D)$ satisfy
	\begin{equation}\label{eq:recursion eq atypical}
	S(T) \leq \prod_{i=1}^D (1+\lambda S(T_i)).
	\end{equation}
\end{proposition}

Even though the proof can be found in \cite{BNNASurvival}, we briefly explain it again, mainly because the ideas will be revisited in Proposition \ref{prop:Mrecursion atypcial}. For a detailed proof, we refer to \cite{BNNASurvival}.

\begin{proof}
	Consider $\cp^\lambda_{\rho} (T_i;\one_{v_i})$ (the subscript $\rho$ indicates that it serves as the added parent above $v_i$), the root-added contact process on each $T_i$, and their product
	\begin{equation}\label{eq:def:productchain}
	\cp^{\otimes}(T;\one_{v_i}) :=
	\left(\otimes_{\substack{j=1\\j\neq i}}^D \cp^\lambda_{\rho} (T_j;\zero ) \right)
	\otimes
	\cp^\lambda_{\rho} (T_i;\one_{v_i}),
	\end{equation}
	for each $i\in[D]$. Let $\textbf{S}_i^\otimes$ denote the  excursion time of this process, that is, the first return time to the all-healthy state $\otimes_{j=1}^D \zero_{T_j}$, and let $S_i^\otimes = \E_{\textsc{cp}} \textbf{S}_i^\otimes$. Further, define the average of $S_i^\otimes$ by
	\begin{equation}\label{eq:excursiontime of product}
	S^\otimes = \frac{1}{D} \sum_{i=1}^D S_i^\otimes.
	\end{equation}
	Then, we can control $S(T)$ by $S^\otimes$, based on the following modification of the process $\cp^\lambda_{\rho^+} (T;\one_\rho)$ introduced in \cite{BNNASurvival}, Lemma 3.3.
	\begin{itemize}
		\item Consider the process $(X^\sharp_t)$ on $T$ that follows the same transition rule as $\cp^\lambda_{\rho^+}(T;\one_\rho)$, except for the recoveries at root $\rho$.
		
		\item An independent rate-$1$ Poisson clock is associated with $\rho$, and the recovery at $\rho$ is only valid if $X^\sharp_t  = \one_\rho$ when the clock rings at time $t$.
	\end{itemize}
	In other words, $(X^\sharp_t)$ is generated by ignoring the recoveries of $(X_t) \sim \cp^\lambda_{\rho^+}(T;\one_\rho)$ at $\rho$ if there is another infected vertex at the time of recovery. Let us denote the expected excursion time of this process by $S^\sharp(T)$. Recalling the coupling argument using the graphical representation (Section \ref{subsubsec:cpprelim}), we know that $S^\sharp(T) \geq S(T)$. Moreover, an excursion of $(X^\sharp_t)$ can be described as follows.
	\begin{enumerate}
		\item [1.] Initially $X^\sharp_0 =\one_\rho$, and we terminate if $\rho$ gets healed before infecting any of its children. Otherwise, suppose that the first child to receive an infection from $\rho$ is $v_i$. 
		
		\item [2.] Since $\rho$ stays infected until everyone else is healthy, it is the same as running an excursion of $\cp^\otimes(T;\one_{v_i})$. When the excursion is finished, we go back to Step 1.
	\end{enumerate}
	
	The probability that we terminate at Step 1 is $(1+\lambda D)^{-1}$. So an excursion of $(X^\sharp_t)$ is a series of excursions of $\{\cp^\otimes(T;\one_{v_i}) \}_i$, until we stop when having a successful coin toss of probability $(1+\lambda D)^{-1}$ after each of the excursion.  Furthermore, note that the expected waiting time to see either a recovery at $\rho$ or an infection at a child is $(1+\lambda D)^{-1}$. Therefore, we obtain that
	\begin{equation}\label{eq:Sbound geometric trials}
	S(T) \leq S^\sharp(T) = \sum_{k=0}^\infty \left(\frac{1}{1+\lambda D}  \right)\left(\frac{\lambda D}{1+\lambda D}\right)^k \left[\frac{k+1}{1+\lambda D} + kS^\otimes  \right] = 1+\lambda D S^\otimes.
	\end{equation}
	
	The final step is to consider the stationary distributions of $\cp^\lambda_{\rho}(T_i)$ and their product. Let $\pi_i$ be the stationary distribution  of $\cp^\lambda(T_i)$. Then, $\pi_i(\zero)$ corresponds to the fraction of time that $\cp^\lambda(T_i)$ is at state $\zero$, and hence
	\begin{equation}\label{eq:stationary dist1}
	\pi_i(\zero) = \frac{\lambda^{-1}}{S(T_i) + \lambda^{-1}} = \frac{1}{1+\lambda S(T_i)}.
	\end{equation}
	Similarly, the stationary distribution $\pi^\otimes$ of $\cp^\otimes(T)$ satisfies 
	\begin{equation}\label{eq:stationary dist prop}
	\pi^\otimes(\zero) = \frac{(\lambda D)^{-1}}{(\lambda D)^{-1}+S^\otimes} = \frac{1}{1+\lambda D S^\otimes}.
	\end{equation}
	Since $\pi^\otimes= \otimes_{i=1}^D \pi_i$, this implies
	\begin{equation}\label{eq:Sotimes product}
	1+\lambda D S^\otimes = \prod_{i=1}^D (1+\lambda S(T_i)).
	\end{equation}
	Therefore, we plug  this into (\ref{eq:Sbound geometric trials}) and obtain the conclusion.
\end{proof}


Unfortunately, having (\ref{eq:recursion eq atypical}) is insufficient for our purpose. One can see this by taking expectation on each side of (\ref{eq:recursion eq atypical}) over $T\sim \gw(\xi)$. In order to yield a meaningful recursion, $\lambda$ should be small in terms of the exponential moment of $\xi$, which has nothing to do with $1/\E \xi$ in general (for details, see the proof of \cite{BNNASurvival}, Lemma 3.3). Therefore,  we develop another recursion which redeems  (\ref{eq:recursion eq atypical}). Our first step is to build up a recursion regarding $R(T)$, the expected survival time.

\begin{proposition}\label{prop:recursion of R}
	Let $T_1,\ldots,T_D$ be as Proposition \ref{prop:Srecursion atypical}, and assume that $\lambda^2 \sum_{i=1}^D R(T_i) <1$. Then, 
	\begin{equation}\label{eq:recursion of R}
	R(T) \leq \frac{1+\lambda \sum_{i=1}^D R(T_i)}{1-\lambda^2\sum_{i=1}^D R(T_i)}.
	\end{equation}
\end{proposition}

\begin{proof}
	Suppose that we run $\cp^\lambda(T;\one_\rho)$. In the beginning, which we call the \textit{first round}, the infection at the root stays there for a while, then it may infect some the children. If a children $v_i$ gets infected, then we can think of it as running a new contact process $\cp^\lambda(T_i; \one_{v_i})$. Here, we should also consider the effect of $v_i$ infecting $\rho$ again, and if this happens, the reinfected $\rho$ starts the \textit{second round} of the dynamics.
	
	The expected survival time of the first round is bounded by
	\begin{equation*}
	1+\lambda \sum_{i=1}^D R(T_i),
	\end{equation*}
	since the expected survival time of the root is $1$ and it sends infections $\lambda$ times in expectation to each children before dying out. Similarly, the expected number of infections sent from the children $\{v_i\}$ to $\rho$ in the first round is bounded by
	\begin{equation*}
	\lambda^2 \sum_{i=1}^D R(T_i).
	\end{equation*}
	Therefore, we obtain that
	\begin{equation*}
	R(T )\leq 
	1+\lambda \sum_{i=1}^D R(T_i) + \left\{\lambda^2 \sum_{i=1}^D R(T_i)\right\} R(T),
	\end{equation*}
	and the conclusion follows since we assumed $\lambda^2 \sum_{i=1}^D R(T_i) <1$.
\end{proof}	

We are now interested in the relation between $R(T)$ and $S(T)$.

\begin{proposition}\label{prop:S vs R}
	On a finite rooted tree $(T,\rho)$, let $R(T)$ and $S(T)$ be the expected survival time of $\cp^\lambda(T;\one_\rho)$ and the expected excursion time of $\cp^\lambda_{\rho^+}(T;\one_\rho)$, respectively. Then, we have
	\begin{equation}\label{eq:S vs R}
	\frac{S(T)}{1+\lambda S(T)} \leq R(T).
	\end{equation}
\end{proposition}

\begin{proof}
	Suppose that we are running a root-added contact process $(X_t)\sim \cp^\lambda_{\rho^+} (T;\one_\rho)$ until time $t_0$.  Further, let  
	$$A_{t_0} = \{t\leq t_0: X_t \neq \zero \}, $$
	and let $|A_{t_0}|$ be the Lebesgue measure of the set $A_{t_0}$. In the root-added contact process, after one excursion we wait $(1/\lambda)$-time in expectation until we start the next excursion. Therefore, we have
	\begin{equation}\label{eq:occupation time by S}
	\lim_{t_0 \rightarrow \infty} \frac{\E_{\textsc{cp}}|A_{t_0}|}{t_0} = \frac{S(T)}{\lambda^{-1} + S(T)} = \frac{\lambda S(T)}{1+ \lambda S(T)}.
	\end{equation}
	
	On the other hand, $\cp^\lambda_{\rho^+} (T; \one_\rho)$ can be considered as the contact process $\cp^\lambda(T;\one_\rho)$ of which the root $\rho$ receives new infections at every ring of an independent Poisson process with rate $\lambda$. Since the rate-$\lambda$ Poisson process rings $\lambda t_0$ times in expectation until time $t_0$, we see that
	\begin{equation*}
	\E_{\textsc{cp} } |A_{t_0}| \leq \lambda t_0 R(T).
	\end{equation*}
	Comparing this to (\ref{eq:occupation time by S}), we obtain that
	$$\frac{\lambda S(T)}{1+\lambda S(T)} \leq \lambda R(T), $$
	implying the conclusion.
\end{proof}

Combining Propositions \ref{prop:recursion of R} and \ref{prop:S vs R}, we obtain that
\begin{equation}\label{eq:recursion eq typical}
\begin{split}
S(T) \leq \frac{R(T)}{1-\lambda R(T)} 
&\leq \frac{1+\lambda \sum_{i=1}^D R(T_i)}{1-\lambda - 2\lambda^2 \sum_{i=1}^D R(T_i) } \\
&\leq
\frac{1+\lambda \sum_{i=1}^D S(T_i)}{1-\lambda - 2\lambda^2 \sum_{i=1}^D S(T_i) }, 
\end{split}
\end{equation}
provided that $\lambda + 2\lambda^2 \sum_{i=1}^D S(T_i) <1$. In the rest of the paper, (\ref{eq:recursion eq atypical}) and (\ref{eq:recursion eq typical}) serve as two major recursive inequalities for $S(T)$.

\subsection{Recursive tail estimate for Galton-Watson trees}\label{subsec:tail estimate survival time}
In this subsection, we establish the primary tail probability estimate on $S(\mathcal{T})$ for a Galton-Watson tree $\mathcal{T}$, and prove Theorem \ref{thm:phasetransition:tree} as its  application.

Let $\xi$ be an integer-valued random variable that satisfy the concentration condition (\ref{eq:concentration condition}) for $\mathfrak{c}=\{c_\delta \}_{\delta\in(0,1]}$. For the Galton-Watson tree $\mathcal{T} \sim \gw(\xi)$, the expected excursion time $S(\mathcal{T})$ is now a random variable driven by the randomness from $\gw(\xi)$. The goal of this subsection is to show that $S(\mathcal{T})$ is finite almost surely if $\lambda\leq (1-\varepsilon)d^{-1} $, where $\ep>0$ is an arbitrarily fixed constant and $d = \E\xi$ is large enough depending on $\ep$. We establish this by proving that the upper tail of  $S(\mathcal{T})$ is very light. In what follows, we denote the law of Galton-Watson trees of depth $l$ by $\gw(\xi)^l$.

\begin{theorem}\label{thm:tail bound of S}
	Let $l\geq0$ be an integer, $\ep\in(0,1)$ and $\mathfrak{c}=\{c_\delta \}_{\delta\in (0,1]}$ be a collection of positive constants. Then there exists $d_0(\ep, \mathfrak{c}) >0$ such that the following holds true. For  any $\xi$ that satisfies $d:=\E \xi \geq d_0$ and (\ref{eq:concentration condition}) with $\mathfrak{c}$, we have for $\lambda = (1-\ep) d^{-1}$ and $\mathcal{T}^l\sim \gw(\xi)^l$ that
	\begin{equation}\label{eq:tail bound of S}
	\P_{\textsc{gw}} \left(S(\mathcal{T}^l) \geq t \right) \leq 
	t^{-\sqrt{d}} (\log t)^{-2} ~~~\textnormal{for all }t\geq \frac{2}{\ep},
	\end{equation}
	where  $S(\mathcal{T}^l)$ is the expected excursion time on $\mathcal{T}^l$.
\end{theorem}


\begin{remark}
	The term $(\log t)^{-2}$ in the r.h.s.$\;$of (\ref{eq:tail bound of S}) may look useless, but this plays a key role in carrying out an inductive argument (see Lemma \ref{lem:recursion principle atypical}). We remark that the exponent $-2$ of $(\log t)^{-2}$ can be replaced by any number smaller than $-1$. Moreover, the exponent $\sqrt{d}$ can be replaced by any term of order $O(d^{1-\eta})$ for $\eta>0$.
\end{remark}

The rest of the subsection is devoted to the proof of Theorem \ref{thm:tail bound of S}. We do this by an induction on $l$, the tree depth. If $l=0$, $\mathcal{T}^0$ is just a single vertex $\rho$ and hence $S(\mathcal{T}^0)\equiv 1$, implying (\ref{eq:tail bound of S}).

Suppose that we have (\ref{eq:tail bound of S}) for $l$. Let $(\mathcal{T}^{l+1},\rho)\sim \gw(\xi)^{l+1}$, and set $D=\deg(\rho) \overset{\textnormal{d}}{=} \xi$. As before, let $T_i,\; i\in[D]$ denote the subtrees of $\mathcal{T}^{l+1}$, rooted at the child $v_i$. Let $c_1$ be the constant in $\mathfrak{c}$ with $\delta =1$, and we  divide $t$ into three  regimes as follows. 
\begin{itemize}
	\item [1.] (small) $2/\ep \leq t\leq d^{\frac{1}{10}}$;
	
	\item [2.] (intermediate) $ d^{\frac{1}{10}}\leq t\leq \exp(\frac{1}{2}c_1 \sqrt{d} ) $;
	
	\item[3.] (large) $\exp(\frac{1}{2} c_1 \sqrt{d}) \leq t$.
\end{itemize}
Then, we establish (\ref{eq:tail bound of S}) on each regime separately. As the proof goes on, we will figure out the conditions for $d_0 (\ep, \mathfrak{c})$ as well.

\begin{remark}\label{rmk:small intermediate regime choice}
	There is much freedom in choosing $d^{\frac{1}{10}}$, the threshold between the small and intermediate regime. Indeed, any $d^\eta$ with $\eta\in(0, \;1-d^{-1/2})$ would work for our purpose. However, the specific choice of $d^{1/10}$ will turn out to be useful later, in the proof of Proposition \ref{prop:SMunicyclic} and Lemma \ref{lem:unicyclic recursion small t}. The choice of $\exp(\frac{1}{2} c_1\sqrt{d})$ will be clear in (\ref{eq:d0 condition 3}).
\end{remark}

\subsubsection{Proof of Theorem \ref{thm:tail bound of S} for small $t$}\label{subsubsection1}

To show (\ref{eq:tail bound of S}) for small $t$, we rely on (\ref{eq:recursion eq typical}). For $\mathcal{T}^{l+1}\sim \gw(\xi)^{l+1}$ and $T_i,\; i\in[D]$ as above, suppose that 
$$\sum_{i=1}^D S(T_i) \leq \frac{2d}{\ep} \left(1+\frac{\ep}{3} \right). $$
Then, from (\ref{eq:recursion eq typical}) and a little bit of algebra we see that
\begin{equation*}
S(\mathcal{T}^{l+1}) \leq \frac{1+\lambda \frac{2d}{\ep} \left(1+\frac{\ep}{3} \right) }{1-\lambda - 2\lambda^2  \frac{2d}{\ep} \left(1+\frac{\ep}{3} \right) }
\leq
\frac{\frac{2}{\ep} -\frac{1}{3} }{1-\frac{7}{d\ep} } 
<
\frac{2}{\ep}, 
\end{equation*}
where the last inequality holds if 
\begin{equation}\label{eq:d0 condition 1}
d\geq \frac{42}{\ep^2}.
\end{equation}
Therefore, for $d$ with (\ref{eq:d0 condition 1}), we have
\begin{equation}\label{eq:tail bound of S split}
\Pgw \left(S(\mathcal{T}^{l+1} ) \geq \frac{2}{\ep}  \right) 
\leq
\P \left(D\geq \left(1+\frac{\ep}{6} \right)d \right) 
+
\Pgw \left(\sum_{i=1}^{(1+\frac{\ep}{6})d } S(T_i) \geq \frac{2d}{\ep}\left(1+\frac{\ep}{3} \right)  \right),
\end{equation}
where $T_i,\; i\in \mathbb{N}$ are i.i.d. $\gw(\xi)^l$. By the assumption on $\xi$, we can bound the first term in the r.h.s. by
\begin{equation}\label{eq:tail bound of S small t 1st part}
\P \left(D\geq \left(1+\frac{\ep}{6} \right)d \right) \leq 
\exp \left(-c_{\ep'}d \right),
\end{equation}
for $\ep' = \ep/6$. To deal with the second term, the induction hypothesis tells us that the c.d.f.$\;$of $S(T_i)$ has an upper bound
$$\Pgw (S(T_i) \geq s) \leq s^{-\sqrt{d}}  $$
for all $s \geq 2/\ep$, and hence we can apply the following lemma. 

\begin{lemma}\label{lem:sum Si small t}
	Let $\ep\in(0,1)$  be a given constant. Then there exists $d_1>0$ independent of $\ep$ such that the following holds for all $d\geq d_1$. Let $Z_i, \; i\in \mathbb{N}$ be i.i.d. positive random variables that satisfies 
	\begin{equation}\label{eq:tail condition for small t}
	\P\left(Z_i \geq t \right) \leq 5t^{-\sqrt{d}}, \quad \textnormal{for all } t\geq \frac{2}{\ep}.
	\end{equation}
	Then, we have
	\begin{equation}\label{eq:tail bound of S small t 2nd part}
	\begin{split}
	\P \left(\sum_{i=1}^{(1+\frac{\ep}{6})d } Z_i \geq \frac{2d}{\ep}\left(1+\frac{\ep}{3} \right)  \right)
	\leq
	\frac{1}{2} d^{-\frac{1}{10}\sqrt{d}} \left(\frac{1}{10} \log d\right)^{-2}.
	\end{split}
	\end{equation}
\end{lemma}

We defer the proof of this lemma until Section \ref{subsec:proof of recursive tail estim}, since it is essentially a special case of Lemma \ref{lem:recursion principle atypical} below. 

Thanks to Lemma \ref{lem:sum Si small t}, we combine (\ref{eq:tail bound of S split}), (\ref{eq:tail bound of S small t 1st part}) and (\ref{eq:tail bound of S small t 2nd part}) to  obtain that
\begin{equation}\label{eq:tail bd of S small t fin}
\Pgw \left(S(\mathcal{T}^{l+1})\geq \frac{2}{\ep} \right)
\leq
\exp(-c_{\ep'}d)+ \frac{1}{2}d^{-\frac{1}{10}\sqrt{d}} \left(\log d^{\frac{1}{10}} \right)^{-2}
\leq d^{-\frac{1}{10}\sqrt{d}} \left({\frac{1}{10}}\log d \right)^{-2},
\end{equation}
for $d$ satisfying 
\begin{equation}\label{eq:d0 condition 2}
\exp(-c_{\ep'}d)
\leq
\frac{1}{2} d^{-\frac{1}{10}\sqrt{d}} \left({\frac{1}{10}}\log d \right)^{-2}. 
\end{equation}	
Finally, we clearly see that (\ref{eq:tail bd of S small t fin}) settles (\ref{eq:tail bound of S}) for all small $t$, namely, $2/\ep \leq t \leq d^{\frac{1}{10}}$.

\subsubsection{Proof of Theorem \ref{thm:tail bound of S} for intermediate $t$}\label{subsubsection2}
For  $t\geq d^{\frac{1}{10}}$, we use (\ref{eq:recursion eq atypical}). That is, we attempt to control 
\begin{equation}\label{eq:tail bd of S intermed t 0}
\Pgw\left(S(\mathcal{T}^{l+1})\geq t\right)
\leq
\Pgw \left(\sum_{i=1}^D \log (1+\lambda S(T_i)) \geq \log t \right).
\end{equation}
Let $c_1$ be the constant given from the assumption of Theorem \ref{thm:tail bound of S} (with $\delta=1$).  Note that we can split the event in the r.h.s.$\;$as follows.
\begin{equation}\label{eq:tail bd of S intermed split}
\Pgw\left(S(\mathcal{T}^{l+1})\geq t\right)
\leq
\P \left(D \geq 2d  \right) +
\Pgw \left(\sum_{i=1}^{2d} \log (1+\lambda S(T_i)) \geq \log t \right).
\end{equation}

The concentration assumption (\ref{eq:concentration condition}) on $\xi$ tells us that the first term is bounded by
\begin{equation}\label{eq:d0 condition 3}
\P(D\geq 2d) \leq \exp(-c_1 d) \leq \frac{1}{2} t^{-\sqrt{d}} (\log t)^{-2}, \quad \textnormal{for } t\leq e^{\frac{1}{2}c_1 \sqrt{d}}.
\end{equation}

To control the second term in the r.h.s.$\;$of (\ref{eq:tail bd of S intermed split}), we use the following lemma.

\begin{lemma}\label{lem:recursion principle atypical}
	Let $\ep\in(0,1)$  be a given constant. Then there exists $d_1(\ep)>0$ such that the following holds true for all $d\geq d_1$. Let $Z_i, \; i\in \mathbb{N}$ be i.i.d. positive random variables that satisfies 
	\begin{equation}\label{eq:recursion principle assumption}
	\P(Z_i \geq t ) \leq 5t^{-\sqrt{ d}} (\log t)^{-2}, ~~~\textnormal{for all } t\geq \frac{3}{\ep}.
	\end{equation}
	Then, for all $t\geq d^{\frac{1}{10}}$, we have
	\begin{equation}\label{eq:tail bd atypical generalized form}
	\P\left( \sum_{i=1}^{2d } \log\left(1+\frac{4}{d}Z_i \right) \geq \log \left(\frac{t}{2} \right) \right) 
	\leq  \frac{1}{2} t^{-\sqrt{d}}(\log t)^{-2}.
	\end{equation}
\end{lemma}

The proof of Lemma \ref{lem:recursion principle atypical} is postponed until Section \ref{subsec:proof of recursive tail estim}, since it requires a substantial technical work. Note that (\ref{eq:tail bd atypical generalized form}) contains a slightly more  generalized form than the r.h.s.$\;$of (\ref{eq:tail bd of S intermed split}), which will be useful in Sections \ref{sec:numberofhits} and \ref{sec:shortsurvival randomgraphs}.

To conclude the proof for intermediate $t$, (\ref{eq:tail bd of S intermed split}), (\ref{eq:d0 condition 3}) and (\ref{eq:tail bd atypical generalized form}) together deduce (\ref{eq:tail bound of S}) for $t\geq d^{\frac{1}{10}}$.  Here, the constant $d_0$ in the statement of the theorem should satisfy $d_0 \geq d_1$ for $d_1$ in Lemma \ref{lem:recursion principle atypical}.

\subsubsection{Proof of Theorem \ref{thm:tail bound of S} for large $t$}\label{subsubsection3} 
For  $t\geq \exp(\frac{1}{2} c_1\sqrt{d})$, we again attempt to control 
(\ref{eq:tail bd of S intermed t 0}). Let
$$\Delta_{d, t} = \frac{4\sqrt{d}}{c_1}\log t,$$
and note that we can split the event in the r.h.s.$\;$of (\ref{eq:tail bd of S intermed t 0}) as follows.
\begin{equation}\label{eq:tail bd of S large split}
\begin{split}
\Pgw\left(S(\mathcal{T}^{l+1})\geq t\right)
\leq&
\;\P \left(D \geq  \Delta_{d,t} \right) +
\Pgw \left(\sum_{i=1}^{2d} \log (1+\lambda S(T_i)) \geq \log t \right)\\
&+
\sum_{2d\, \leq\, r\, \leq\, \Delta_{d,t} } \P(D \geq r) 
\times \Pgw \left(\sum_{i=1}^r \log(1+\lambda S(T_i) ) \geq \log t \right).
\end{split}
\end{equation}
The concentration assumption (\ref{eq:concentration condition}) on $\xi$ tells us that for $t\geq \exp(\frac{1}{2} c_1\sqrt{d})$,
\begin{equation}\label{eq:d0 condition 7}
\begin{split}
\P\left(D \geq \Delta_{d,t} \right)		
\leq 
\P \left(D\geq d+ \frac{\Delta_{d,t}}{2}  \right)
\leq \exp\left(-2\sqrt{d}\log t \right) 
\leq
\frac{1}{4}t^{-\sqrt{d}} (\log t)^{-2},
\end{split}
\end{equation}
where the last inequality is satisfied by large $d$. This controls the first term in the r.h.s.$\;$of (\ref{eq:tail bd of S large split}). The second term is then estimated by Lemma \ref{lem:recursion principle atypical}, which gives
\begin{equation}\label{eq:recursionmid term}
\Pgw \left(\sum_{i=1}^{2d} \log (1+\lambda S(T_i)) \geq \log t \right)
\leq 
\frac{1}{2}t^{-\sqrt{d}} (\log t)^{-2}.
\end{equation}

To bound the last term, we use the following lemma.

\begin{corollary}\label{cor:recursion principle atypical}
	Let $\ep\in(0,1)$  be a given constant. Then there exist $d_1(\ep)>0$ and an absolute constant $C_1>0$ such that the following holds true for all $d\geq d_1$. Let $Z_i, \; i\in \mathbb{N}$ be i.i.d. positive random variables that satisfies 
	\begin{equation*}
	\P(Z_i \geq t ) \leq 5t^{-\sqrt{ d}} (\log t)^{-2}, ~~~\textnormal{for all } t\geq \frac{3}{\ep}.
	\end{equation*}
	Then, for all $t\geq \exp(\frac{1}{2} c_1 \sqrt{d} )$ and $2d\leq r\leq \Delta_{d,t}$ with $\Delta_{d,t}$ as above, we have
	\begin{equation}\label{eq:tail bd atypical gen r}
	\P\left( \sum_{i=1}^{r } \log\left(1+\frac{4}{d}Z_i \right) \geq \log \left(\frac{t}{2} \right) \right) 
	\leq   t^{-\sqrt{d}}(\log t)^{-2} \exp\left(\frac{C_1r}{\ep \sqrt{d}} \right).
	\end{equation}
\end{corollary}

The proof of Corollary \ref{cor:recursion principle atypical} is postponed until Section \ref{subsec:proof of recursive tail estim}, which will be proven together with Lemma \ref{lem:recursion principle atypical}. 


To conclude the proof of Theorem \ref{thm:tail bound of S} for large $t$, observe that for $2d\leq r \leq \Delta_{d,t}$
\begin{equation*}
\P(D \geq r ) \leq \P\left(D \geq d+ \frac{r}{2} \right) \leq \exp\left(-\frac{c_1 r}{2} \right),
\end{equation*}
by the concentration condition (\ref{eq:concentration condition}) of $D$. Therefore, for sufficiently large $d$, we have by (\ref{eq:tail bd atypical gen r}) that
\begin{equation*}
\P(D\geq r) \times \Pgw \left(\sum_{i=1}^r \log(1+\lambda S(T_i) ) \geq \log t \right)
\leq
t^{-\sqrt{d}}(\log t)^{-2} \exp\left(-\frac{c_1 r}{4} \right) .
\end{equation*}
Along with (\ref{eq:tail bd of S large split}), (\ref{eq:d0 condition 7}) and (\ref{eq:recursionmid term}), this proves Theorem \ref{thm:tail bound of S} for  $t\geq \exp(\frac{1}{2}c_1 \sqrt{d} )$.	  \qed

\vspace{3mm}

As the first application of Theorem \ref{thm:tail bound of S}, we establish Theorem \ref{thm:phasetransition:tree}.

\begin{proof}[Proof of Theorem \ref{thm:phasetransition:tree}: the lower bound]
	Let $\ep\in(0,1)$ be an arbitrary constant and suppose that the collection of distributions $\{\xi_k \}_k$ satisfies (\ref{eq:concentration condition}) for $\mathfrak{c} = \{c_\delta\}_{\delta \in(0,1]} $. Let $d_0 = d_0(\ep, \mathfrak{c})$ as in Theorem \ref{thm:tail bound of S}, and assume that $\E \xi_k \geq d_0$. Then, Theorem \ref{thm:tail bound of S} implies that for $\cp^\lambda_{\rho^+}(\mathcal{T}^l_k;\one_\rho)$ on $\mathcal{T}_k \sim \gw(\xi_k)^l$ with intensity $\lambda = (1-\ep) (\E \xi_k)^{-1}$, we have
	\begin{equation*}
	\E_{\textsc{gw}} \left[S(\mathcal{T}^l_k) \right] < \frac{3}{\ep}.
	\end{equation*}
	Thus, we deduce that the expected survival time of $\cp^\lambda(\mathcal{T}^l_k;\one_\rho)$ is also 
	\begin{equation*}
	\E_{\textsc{gw}} \left[R(\mathcal{T}^l_k) \right] <\frac{3}{\ep}.
	\end{equation*}
	Therefore, by the monotone convergence theorem, we have that for $\mathcal{T}_k \sim \gw(\xi_k)$,
	\begin{equation*}
	\E_{\textsc{gw}} [ R(\mathcal{T}_k)] = \lim_{l\rightarrow \infty} \E_{\textsc{gw}} [R(\mathcal{T}_k^l)] \leq \frac{3}{\ep},
	\end{equation*}
	and hence the survival time $\textbf{R}(\mathcal{T}_k)$ is finite almost surely, which implies that 
	$$\lambda = \frac{1-\ep}{\E \xi_k} \leq \lambda_1^{\textsc{gw}}(\xi_k), $$
	for all large enough $k$.
\end{proof}

\subsection{Proof of the induction lemmas}\label{subsec:proof of recursive tail estim}

To conclude Section \ref{sec:subcritical trees}, we discuss the proof of Lemmas \ref{lem:sum Si small t}, \ref{lem:recursion principle atypical} and Corollary \ref{cor:recursion principle atypical}. We first establish Lemma \ref{lem:recursion principle atypical} and Corollary \ref{cor:recursion principle atypical} together, and then the proof of Lemma \ref{lem:sum Si small t} will follow based on similar ideas.

\begin{proof}[Proof of Lemma \ref{lem:recursion principle atypical} and Corollary \ref{cor:recursion principle atypical}]
	We first set 
	\begin{equation}\label{eq:ndt def}
	Y_i = \log \left(1+\frac{4Z_i}{d}\right), \quad \Delta_{d,t} =\frac{4\sqrt{d}}{c_1}\log t, \quad \textnormal{and let} \quad 2d \leq r \leq  2d \,\vee \Delta_{d,t}.
	\end{equation}
	We will attempt to find an upper bound of
	\begin{equation}\label{eq:recprinciplemainobject}
	\P\left( \sum_{i=1}^{r } \log\left(1+\frac{4}{d}Z_i \right) \geq \log \left(\frac{t}{2} \right) \right), 
	\end{equation}
	to cover the case of Corollary \ref{cor:recursion principle atypical} as well. 
	
	Note that the assumption (\ref{eq:recursion principle assumption}) on $Z_i$ gives us $\E Z_i\leq \frac{4}{\ep}$. Based on (\ref{eq:recursion principle assumption}), we can find an upper bound for $\E Y_i$ by 
	\begin{equation}\label{eq:EYi bound}
	\E Y_i \leq \log \E e^{Y_i} = \log\left(1+ \frac{4\E Z_i}{d}\right) \leq \log \left(1+\frac{16}{\ep d} \right) \leq \frac{16}{\ep d}.
	\end{equation}
	Further, (\ref{eq:recursion principle assumption}) tells us that the tail of $Y_i$ is bounded by
	\begin{equation}\label{eq:tail bd of Yi 1}
	\P(Y_i \geq s  ) =\P\left(Z_i \geq \frac{d}{4}(e^s-1) \right)
	\leq
	5
	\left(\frac{d}{4}(e^s -1) \right)^{-\sqrt d}
	\left\{\log\left(\frac{d}{4}(e^s -1) \right) \right\}^{-2}.
	\end{equation}
	
	In order to have $\sum_{i=1}^{r} Y_i \geq \log(t/2)$ for $t\geq \sqrt{d}$ and $r\leq 2d \vee \Delta_{d,t}$, we need to have some of $Y_i$ being larger than $\frac{\log d}{80 d}$. In the following, $k$ denotes the number of $Y_i$ that are at least $\frac{\log d}{80 d}$. Note that in this case, the sum over all $Y_i$ that are smaller than $\frac{\log d}{80d}$ is at most 
	\begin{equation}\label{eq:mdt def}
	(2d\vee \Delta_{d,t}) \cdot \frac{\log d}{80 d}= \frac{1}{40}\log d \, \vee \frac{\log d}{20c_1\sqrt{d}} \log t =: m(d,t).
	\end{equation} 
	Then, we can split the probability in  (\ref{eq:recprinciplemainobject}) as follows.
	\begin{equation}\label{eq:tail bd sum of Yi split}
	\begin{split}
	&\P \left(\sum_{i=1}^{r} Y_i \geq \log \left(\frac{t}{2} \right) \right)\\
	&\leq
	\sum_{k=1}^{r} {r \choose k} \sum_{j=1}^{m(d,t)+1}
	\P\left(\sum_{i=1}^k Y_i \geq \log \left(\frac{t}{2} \right)-j \; ; \; Y_i \geq \frac{\log d}{80d}, \;\forall i\leq k \right)
	\;\\
	&\qquad \qquad\qquad \qquad \qquad \times \P	\left(\sum_{i=k+1}^{r} Y_i \geq j-1 \; ;\; Y_i \leq \frac{\log d}{80 d},\; \forall i>k  \right).
	\end{split}
	\end{equation}
	Our focus is to control the two terms in the inner sum of the r.h.s.. We begin with the first one.
	
	\begin{lemma}\label{lem:recinterm1}
		Under the setting of Lemma \ref{lem:recursion principle atypical} and (\ref{eq:ndt def}), there exist absolute constants $K,d_0>0$ such that for all $d\geq d_0$, $1\leq k\leq r$ and $1\leq j \leq m(d,t)+1$, we have
		\begin{equation*}
		\P\left(\sum_{i=1}^k Y_i \geq \log \left(\frac{t}{2} \right)-j \; ; \; Y_i \geq \frac{\log d}{80d}, \;\forall i\leq k \right)
		\leq
		\left(\frac{ K^{\sqrt{d}} d }{(\log d)^{\sqrt{d}}} \right)^k
		e^{-\sqrt{d}\left(\log t-j \right)} (\log t)^{-2}
		\end{equation*}
	\end{lemma}

	\begin{proof}
		Since $e^x-1\geq\frac{x}{x+1}e^x \geq \frac{y}{2}e^x$ for all $x\wedge 1\geq y>0$, we have that for $s\geq \frac{\log d}{80d}$,
		$$\frac{d}{4}(e^s -1)  \geq \frac{d}{4} \cdot \frac{\log d}{160 d} e^s = \frac{\log d}{640} e^s.$$
		Plugging this into (\ref{eq:tail bd of Yi 1}), we obtain that
		\begin{equation}\label{eq:tail bd of Yi 2}
		\P(Y_i \geq s) \leq 5 \left(\frac{640}{\log d} \right)^{\sqrt{d}} e^{-s\sqrt{d}} \left(s+ \frac{1}{2}\log \log d \right)^{-2},
		\end{equation}
		for all $ s\geq \frac{\log d}{80d} $, if  $d$  is such that
		\begin{equation}\label{eq:d0 condition 4}
		\frac{1}{2}\log \log d \leq \log \left(\frac{\log d}{640} \right), 
		\quad \textnormal{i.e.,}
		\quad
		d\geq \exp\left(640^2 \right).
		\end{equation}
		Keeping this in mind, we can rewrite the probability for each $k$ and $j$ as
		\begin{equation}\label{eq:tail bd of sum Yi large}
		\begin{split}
		&\P\left(\sum_{i=1}^k Y_i \geq \log \left(\frac{t}{2} \right)-j \; ; \; Y_i \geq \frac{\log d}{80d}, \;\forall i\leq k \right) \\
		&\leq
		\sum_{\substack{ u_1,\ldots , u_{k} \geq 1 : \\ \sum_{i=1}^{k} u_i \geq \frac{80d}{\log d}\{\log (t/2)-j\}-k }}
		\left[\prod_{i=1}^k \P\left(Y_i \in \frac{\log d}{80 d} [u_i, \; u_i +1] \right) \right]\\
		&\leq 
		\sum_{\substack{ u_1,\ldots , u_{k} \geq 1 : \\ \sum_{i=1}^{k} u_i = \lfloor \frac{80d}{\log d}\{\log (t/2)-j\} -k\rfloor }}
		\left[\prod_{i=1}^k \P\left(Y_i \geq \frac{\log d}{80 d}u_i \right) \right]	.
		\end{split}
		\end{equation}
		Then, the r.h.s. of (\ref{eq:tail bd of sum Yi large}) is upper bounded by
		\begin{equation}\label{eq:tail bd of sum Yi large 2}
		\begin{split}
		&5^k \left(\frac{640}{\log d} \right)^{k\sqrt{d}}
		e^{-\sqrt{d}\left( \log (t/2) -j -\frac{k\log d}{80 d} \right)}\\
		&\times 
		\sum_{\substack{ u_1,\ldots , u_{k} \geq 1 : \\ \sum_{i=1}^{k} u_i = \lfloor \frac{80d}{\log d}\{\log (t/2)-j\}-k\rfloor }}
		\prod_{i=1}^k \left(\frac{\log d}{80d} u_i + \frac{1}{2} \log\log d \right)^{-2}.
		\end{split}
		\end{equation}
		The sum in the right can be controlled inductively in $k$ according to the following lemma.
		
		\begin{lemma}\label{lem:sum for induction log term}
			Let $a> b>0$ and $m\in \mathbb{N}$. Then,
			\begin{equation*}
			\sum_{u=0}^m (a+bu)^{-2} (a+b(m-u))^{-2} \leq \frac{8}{b(a-b) (a+bm)^2}.
			\end{equation*}
		\end{lemma}
		
		\begin{proof}[Proof of Lemma \ref{lem:sum for induction log term}]
			The l.h.s.$\;$of above can be upper bounded by
			\begin{equation*}
			\begin{split}
			2\sum_{u=0}^{\lceil m/2\rceil } (a+bu)^{-2} \left(a+\frac{bm}{2} \right)^{-2}
			\leq
			\frac{8}{b^2 (2a+bm)^2}\sum_{m=0}^{\infty} \left(u+\frac{a}{b} \right)^{-2}
			\leq
			\frac{8}{b(a-b) (a+bm)^2}.
			\end{split}
			\end{equation*}
		\end{proof}
		
		From now on, $K>0$ is an absolute constant that may vary from line by line. We apply Lemma \ref{lem:sum for induction log term} $(k-1)$-times to (\ref{eq:tail bd of sum Yi large 2}) with $a=\frac{1}{2}\log\log d$, $b=\frac{\log d}{80d}$ and see that 
		\begin{equation}\label{eq:tail bd of sum Yi large 3}
		\begin{split}
		&\sum_{\substack{ u_1,\ldots , u_{k} \geq 1 : \\ \sum_{i=1}^{k} u_i = \lfloor \frac{80d}{\log d}\{\log (t/2)-j\}-k\rfloor }}
		\prod_{i=1}^k \left(\frac{\log d}{80d} u_i + \frac{1}{2} \log\log d \right)^{-2} \\
		&\qquad \qquad \leq
		\left(\frac{K d}{\log d } \right)^{k-1} \left(\log \left(\frac{t}{2} \right) - j -\frac{k\log d}{80 d} \right)^{-2}.
		\end{split}
		\end{equation}
		Therefore, we see from (\ref{eq:tail bd of sum Yi large}), (\ref{eq:tail bd of sum Yi large 2}) and (\ref{eq:tail bd of sum Yi large 3}) that
		\begin{equation}\label{eq:tail bd of sum Yi larger fin}
		\begin{split}
		&\P\left(\sum_{i=1}^k Y_i \geq \log \left(\frac{t}{2} \right)-j \; ; \; Y_i \geq \frac{\log d}{80d}, \;\forall i\leq k \right) \\
		&\leq
		\left(\frac{K^{\sqrt{d}}d }{(\log d)^{\sqrt{d}}} \right)^k
		e^{-\sqrt{d}\left(\log (t/2)-j- \frac{k\log d}{80d} \right)}\left(\log \left(\frac{t}{2} \right) -j -\frac{k \log d}{80d}  \right)^{-2}. 
		\end{split}
		\end{equation}
		Note that the definition of $r$ (\ref{eq:ndt def}) and $m(d,t)$ (\ref{eq:mdt def}) implies that for $t\geq d^{\frac{1}{10}}$,
		$$\log\left(\frac{t}{2} \right) - m(d,t) -
		\frac{r\log d}{80 d} 
		\geq 
		\log \left(\frac{t}{2}\right) - \frac{1}{4}\log t -\frac{1}{4}\log t \geq \frac{1}{3} \log t.$$
		Therefore, the r.h.s. of (\ref{eq:tail bd of sum Yi larger fin}) is at most
		\begin{equation*}
		\left(\frac{ K^{\sqrt{d}} d }{(\log d)^{\sqrt{d}}} \right)^k
		e^{-\sqrt{d}\left(\log (t/2)-j- \frac{k\log d}{80d} \right)} 9(\log t)^{-2},
		\end{equation*}
		Absorbing the constants $9$, $2^{\sqrt{d}}$ and the term $\exp(\frac{k\log d}{80\sqrt{d}})$ into $K$, we obtain the conclusion.
	\end{proof}
	
	Now we turn our attention to bounding the second term in the inner sum of (\ref{eq:tail bd sum of Yi split}) in the r.h.s.. 
	
	\begin{lemma}\label{lem:recintermed2}
		Under the setting of Lemma \ref{lem:recursion principle atypical} and (\ref{eq:ndt def}), there exist  absolute constants $K,d_0>0$  such that for all $d\geq d_0$, $1\leq k\leq r$ and $1\leq j \leq m(d,t)+1$, we have
		\begin{equation*}
		\P	\left(\sum_{i=k+1}^{r} Y_i \geq j-1 \; ;\; Y_i \leq \frac{\log d}{80 d},\; \forall i>k  \right)
		\leq
		1\wedge
		\left( \frac{ \exp\left(19 r \,d^{-1/3} \right)}{\exp\left(\ep d^{2/3}j \right)}  \right).
		\end{equation*}
	\end{lemma}
	
	\begin{proof}
		 We first observe that by (\ref{eq:EYi bound}) and (\ref{eq:tail bd of Yi 1}),
		\begin{equation*}
		\begin{split}
		\E &\left[\exp\left(\ep d^{2/3} Y_i \right) \; ;\; Y_i \leq \frac{\log d}{80 d} \right] \leq \exp\left(\ep d^{2/3} \left(\frac{16}{\ep d} \right) \right) +  \E \left[ \exp \left(\ep d^{2/3} Y_i \right) \one_{ \left\{ \frac{16}{\ep d} \leq Y_i \leq \frac{\log d}{80 d} \right\}} \right]
		\\
		&\qquad \qquad \qquad \qquad \leq 
		\exp\left(16d^{-1/3} \right)
		+
		\int_{\frac{16}{\ep d}}^{\frac{\log d}{80 d}} e^{\ep d^{2/3} y} \left(\frac{d}{4}(e^y-1) \right)^{-\sqrt{d}} \textnormal{d}y\\
		&\qquad \qquad \qquad \qquad \leq 
		1+17d^{-\frac{1}{3}} + \frac{\log d}{80d} \times e^{\frac{\ep\log d}{80d^{1/3} }} \times \left(\frac{d}{4} \cdot \frac{16}{\ep d} \right)^{-\sqrt{d}}
		\leq 1+ 18d^{-\frac{1}{3}},
		\end{split}
		\end{equation*}
		where the last two inequalities hold for $d$ larger than some absolute constant $d_2$. Therefore, we obtain that
		\begin{equation*}
		\begin{split}
		&\P	\left(\sum_{i=k+1}^{r} Y_i \geq j-1 \; ;\; Y_i \leq \frac{\log d}{80 d},\; \forall i>k  \right)
		\leq
		\P	\left(\sum_{i=1}^{r} Y_i \geq j-1 \; ;\; Y_i \leq \frac{\log d}{80 d},\; \forall i>k  \right)\\
		&\qquad \qquad \qquad \leq 
		1\;\wedge\; \exp\left(-\ep d^{2/3}(j-1) \right) \times \left\{\E \left[\exp\left(\ep d^{2/3} Y_i \right) \; ;\; Y_i \leq \frac{\log d}{80 d} \right]\right\}	^r
		\\
		&\qquad \qquad \qquad \leq
		1\;\wedge\; \frac{\left(1+ 18d^{-1/3} \right)^{r} }{\exp\left(\ep d^{2/3}(j-1) \right) } 
		\leq
		1\wedge
		\left( \frac{ \exp\left(19 r \,d^{-1/3} \right)}{\exp\left(\ep d^{2/3}j \right)}  \right) ,
		\end{split}
		\end{equation*}
		concluding the proof.
	\end{proof}
	
	We can now combine (\ref{eq:tail bd sum of Yi split}) with Lemmas \ref{lem:recinterm1} and \ref{lem:recintermed2} to see that
	\begin{equation}\label{eq:tail bd of sum Yi fin}
	\begin{split}
	&\P \left(\sum_{i=1}^{r} Y_i \geq \log \left(\frac{t}{2} \right) \right)\\
	&\leq
	\frac{\;t^{-\sqrt{d}}}{(\log t)^2} \sum_{k=1}^{r} \sum_{j=1}^{m(d,t)+1} {r \choose k}
	\left(\frac{K^{\sqrt{d}}d }{(\log d)^{\sqrt{d}}} \right)^k
	e^{j\sqrt{d}} 
	\left( \frac{ e^{19 r\,d^{-1/3} }}{e^{\ep d^{2/3}j }}   \; \wedge \; 1\right).\\
	&\leq 
	\frac{t^{-\sqrt{d}}}{(\log t)^2} 
	\sum_{k=1}^{r} {r \choose k}
	\left(\frac{K^{\sqrt{d} }d}{(\log d)^{\sqrt{d}}} \right)^k
	\left[\sum_{j=1}^{m(d,t)+1}
	e^{j\sqrt{d}} 	\left( \frac{ e^{19 r\,d^{-1/3} }}{e^{\ep d^{2/3}j }}   \; \wedge \; 1\right) \right],
	\end{split}
	\end{equation}
	
	The last inner sum over $j$ can be split into two parts, $\ep d^{2/3 } j \leq 19 r\, d^{-1/3}$ and $\ep d^{2/3}j > 19r\, d^{-1/3}$. That is, we divide into two regimes based on $j_0= 19r (\ep d)^{-1} $.	 If, for instance, $2\sqrt{d}\leq \ep d^{2/3}$, then this sum  is bounded by $2\exp(j_0 \sqrt{d})=2\exp(19 r\, \ep^{-1} d^{-1/2} ) $, since  the exponent $\ep d^{2/3} j$ in the denominator becomes at least twice as large as $j\sqrt{d} $ in the numerator. This gives another condition for $d_1$, namely,
	\begin{equation}\label{eq:d0 condition 5}
	d \geq \left(\frac{2}{\ep} \right)^6.
	\end{equation}
	
	Further,	the outer sum over $k$ in the r.h.s.$\;$of (\ref{eq:tail bd of sum Yi fin}) is at most
	\begin{equation}\label{eq:tail bd of sum Yi fin2}
	\begin{split}
	2\exp\left(\frac{19r}{\ep \sqrt{d}} \right) \sum_{k=1}^{r} 
	{r \choose k} \left(\frac{K^{\sqrt{d} }d}{(\log d)^{\sqrt{d}}} \right)^k
	&\leq 
	2\exp\left(\frac{19r}{\ep \sqrt{d}}\right) \times \left( 1+\frac{K^{\sqrt{d} }d}{(\log d)^{\sqrt{d}}} \right)^r\\
	&\leq \exp \left(\frac{20r}{\ep \sqrt{d}} \right),
	\end{split}
	\end{equation}
	where the last inequality holds for large $d$ such that 
	\begin{equation}\label{eq:d0 condition 8}
	\frac{1}{\ep \sqrt{d}} > \left(\frac{K}{\log d}\right)^{\sqrt d} d .
	\end{equation}
	Therefore, we establish Corollary \ref{cor:recursion principle atypical} from  (\ref{eq:tail bd of sum Yi fin}) and (\ref{eq:tail bd of sum Yi fin2}) by setting $C_1 =20$ and $d_0$ to satisfy (\ref{eq:d0 condition 4}), (\ref{eq:d0 condition 5}) and (\ref{eq:d0 condition 8}).
	For Lemma \ref{lem:recursion principle atypical}, we plug in $r=2d$ in the l.h.s. of (\ref{eq:tail bd of sum Yi fin2}), and see that if $d$ is so large  that
	\begin{equation}\label{eq:d0 condition 6}
	\log d\geq 3 \cdot K\cdot e^{19/\ep},
	\end{equation}
	then by using ${2d \choose k} \leq (2d)^k$ and the fact that $\sum_{k\geq 1} u^k \leq 2u$ for $u\leq \frac{1}{2}$,
	\begin{equation}\label{eq:constant infrontofthe tail}
	2\exp\left(\frac{19\sqrt{d}}{\ep } \right) \sum_{k=1}^{2d} 
	{2d \choose k} \left(\frac{K^{\sqrt{d} }d}{(\log d)^{\sqrt{d}}} \right)^k
	\leq
	4e^{\frac{19\sqrt{d}}{\ep }}\left(\frac{K^{\sqrt{d}}d^2}{(\log d)^{\sqrt{d}}} \right)
	\leq
	\frac{ 4d^2}{2^{\sqrt{d}}} \leq \frac{1}{2}.
	\end{equation}
	
	Therefore, we obtain the conclusion of Lemma \ref{lem:recursion principle atypical}, by setting $d_1(\ep)$ to satisfy (\ref{eq:d0 condition 5}) and (\ref{eq:d0 condition 6}). (Condition (\ref{eq:d0 condition 4}) is absorbed into (\ref{eq:d0 condition 6}).) 
\end{proof}

We conclude this section by proving Lemma \ref{lem:sum Si small t}. The idea of splitting the probability as (\ref{eq:tail bd sum of Yi split}) is used, but here the computation is simpler than the previous one.

\begin{proof}[Proof of Lemma \ref{lem:sum Si small t}]
	Define $Z_i' = Z_i - \frac{2}{\ep}, $ then $Z_i'$ satisfies
	\begin{equation*}
	\P(Z_i'\geq s) \leq 5\left(s+\frac{2}{\ep} \right)^{-\sqrt{d}}, \quad \textnormal{for all } s\geq 0.
	\end{equation*}
	We claim that
	\begin{equation}\label{eq:tail bound of S small t 2nd part2}
	\begin{split}
	\P \left(\sum_{i=1}^{(1+\frac{\ep}{6})d } Z_i \geq \frac{2d}{\ep}\left(1+\frac{\ep}{3} \right)  \right)
	\leq
	\P \left(\sum_{i=1}^{(1+\frac{\ep}{6}) d} Z_i' \geq \frac{d}{3} \right)
	\leq
	\frac{1}{2} d^{-\frac{1}{10}\sqrt{d}} \left(\log\left(d^{\frac{1}{10}} \right) \right)^{-2}.
	\end{split}
	\end{equation}
	In order to have $\sum_{i=1}^{(1+\frac{\ep}{6}) d} Z_i' \geq \frac{d}{3}$, some $Z_i'$ must be at least $\frac{1}{10}$. In what follows, $k$ denotes the number of such $Z_i'$. Note that on the other hand, the sum over all $Z_i'$ that are at most $\frac{1}{10}$ is bounded by $(1+\frac{\ep}{6})\frac{d}{10} \leq  \frac{d}{6}$ from above. Therefore, we can bound the probability in the middle in (\ref{eq:tail bound of S small t 2nd part2}) by
	\begin{equation}\label{eq:Ziprime split}
	\P\left(\sum_{i=1}^{(1+\frac{\ep}{6}) d} Z_i' \geq \frac{d}{3} \right)
	\leq 
	\sum_{k=1}^{(1+\frac{\ep}{6})d} {(1+\frac{\ep}{6})d \choose k}  \P\left( \sum_{i=1}^k Z_i' \geq \frac{d}{6} \; ; \; Z_i' \geq \frac{1}{10} \right).
	\end{equation}
	We can go through  similar steps as in (\ref{eq:tail bd of sum Yi large}) to see that 
	\begin{equation*}
	\begin{split}
	\P\left( \sum_{i=1}^k Z_i' \geq \frac{d}{6} \; ; \; Z_i' \geq \frac{1}{10} \right)
	&\leq
	\sum_{\substack{u_1,\ldots,u_k \geq 0 \\  u_1 +\ldots + u_k = \lfloor\frac{5}{3}d-{k}\rfloor }} \prod_{i=1}^k \P\left(Z_i' \geq \frac{u_i}{10} \right)\\
	&\leq
	\sum_{\substack{u_1,\ldots,u_k \geq 0 \\  u_1 +\ldots + u_k = \lfloor\frac{5}{3}d-{k}\rfloor }} 5^k \prod_{i=1}^k \left(\frac{2}{\ep} +\frac{u_i}{10} \right)^{-\sqrt{d}}.
	\end{split}
	\end{equation*}
	The choice of $\{u_i \}_{i=1}^k$ that maximizes the product in the r.h.s. is such that all but exactly one $u_i$ are equal to zero. Hence, its maximum is at most
	$$\left(\frac{\ep}{2} \right)^{(k-1)\sqrt{d}}\left[\frac{\ep}{2}+ \frac{1}{10}\left(\frac{5d}{3}-k \right) \right]^{-\sqrt{d}}  \leq \left(\frac{\ep}{2} \right)^{(k-1)\sqrt{d}} \left(\frac{18}{d} \right)^{\sqrt{d}}, $$
	using $\frac{k}{10} \leq \frac{d}{9}$.
	Since the number of $\{u_i \}_{i=1}^k$ satisfying $\sum_{i=1}^k u_i = \lfloor\frac{5}{3}d -k  \rfloor$ is at most $2d \choose k$, the r.h.s. of (\ref{eq:Ziprime split}) is at most
	\begin{equation*}
	\sum_{k=1}^{(1+\frac{\ep}{6})d}  {2d \choose k}^2 5^k \left(\frac{\ep}{2} \right)^{(k-1)\sqrt{d}} \left(\frac{18}{d} \right)^{\sqrt{d}}
	\leq
	2d \left(\frac{36}{\ep d} \right)^{\sqrt{d}} \sum_{k\geq 1} \left[ 20d^2 \left(\frac{\ep}{2} \right)^{\sqrt{d}} \right]^k
	\leq
	80d^3 \left(\frac{18}{d} \right)^{\sqrt{d}},
	\end{equation*}
	using $\sum_{k\geq 1} u^k \leq 2u$ for small enough $u>0$. Thus, there exists $d_1>0$ such that the r.h.s.$\;$is smaller than that of (\ref{eq:tail bound of S small t 2nd part2}) for $d\geq d_1(\ep)$.
\end{proof}

\section{Total infections at leaves on trees}\label{sec:numberofhits}

Let $(T,\rho)$ be a finite rooted tree of \textit{depth} $l$. That is, 
\begin{equation}\label{eq:depth def}
l=\max\{\textnormal{dist}(\rho,v):v\in T \}.
\end{equation}
As discussed in Section \ref{subsubsec:idea2-1}, we desire to control the infection going deep inside the tree for subcritical $\lambda$. To this end, we investigate $M^l(T) $, the \textit{expected total infections at depth-$l$ leaves} defined in Definition \ref{def:tit}. In particular, the goal of this section is to establish the following theorem:

\begin{theorem}\label{thm:M tail bd}
	Let $l\geq 0$ be an integer, $\ep\in(0,1)$ and $\mathfrak{c}=\{c_\delta \}_{\delta\in (0,1]}$ be a collection of positive constants. Then there exists $d_0(\ep, \mathfrak{c}) >0$ such that the following holds true. For  any $\xi$ that satisfies $d=\E \xi \geq d_0$ and (\ref{eq:concentration condition}) with $\mathfrak{c}$, we have for $\lambda = (1-\ep) d^{-1}$ and $\mathcal{T}^l\sim \gw(\xi)^l$ that
	\begin{equation}\label{eq:Mtail:in thm}
	\P_{\textsc{gw}} \left(M^l(\mathcal{T}^l) \geq \left(1- \frac{\ep}{10} \right)^l t \right) \leq 
	t^{-\sqrt{d}} (\log t)^{-2} ~~~\textnormal{for all }t\geq 2,
	\end{equation}
	where  $M^l(\mathcal{T}^l)$ is the expected total infections at depth-$l$ leaves on $\mathcal{T}^l$.
\end{theorem}

To prove this theorem, we first derive two different recursive inequalities for $M^l(T)$ for a deterministic tree $T$  in Sections   \ref{subsec:tit:deterministic2} and \ref{subsec:tit:deterministic}. Then, in Section \ref{subsec:tit:gw}, we verify the theorem, which is  along the lines of proving Theorem \ref{thm:tail bound of S}.

\subsection{The first recursive inequality}\label{subsec:tit:deterministic2}
We begin with  deriving a recursive inequality on $M^l(T)$ described in the following proposition, which is an analogue of (\ref{eq:recursion eq atypical}). 

\begin{proposition}\label{prop:Mrecursion atypcial}
	For a finite rooted tree $(T,\rho)$ of depth $l$, let $D=\deg(\rho)$ and $T_1,\ldots,T_D$ be the subtrees rooted at each child $v_i$ of $\rho$. Then, $M^l(T)$, the expected total infections at depth-$l$ leaves  on $T$, satisfies the following.
	\begin{equation}\label{eq:Mrecursion atypical}
	M^l(T) \leq \lambda \sum_{i=1}^D M^{l-1}(T_i) \prod_{\substack{1\,\leq\, j\, \leq\, D\\ j\neq i}} (1+\lambda S(T_j)).
	\end{equation}
\end{proposition}


\begin{proof}

	Recall the processes $\cp^\otimes(T;\one_{v_i})$ and $(X^\sharp_t)$ defined in the proof of Proposition \ref{prop:Srecursion atypical}. Let $\textbf{S}^\sharp(T)$ and $\textbf{S}^\otimes_i(T)$ be the excursion time of $(X^\sharp_t)$ and $(X^\otimes_{i,t}) \sim \cp^\otimes(T;\one_{v_i})$, respectively, and set $\mathcal{L} = \{v\in T: \textnormal{dist}(v,\rho) = l \}$.  We define the total infections at depth-$l$ leaves $\textbf{M}^\sharp(T)$ (resp., $\textbf{M}_i^\otimes(T)$) of $(X^\sharp_t)$ (resp., $(X^\otimes_{i,t})$) analogously as Definition \ref{def:tit}.
	\begin{itemize}
		\item For $v\in \mathcal{L}$, let
		\begin{equation*}
		\begin{split}
		\textbf{M}^\sharp_v(H)&:= \left|\left\{ s \in [0,\textbf{S}^\sharp(H)]:\; X_s^\sharp(v)=1 \textnormal{ and } X_{s-}^\sharp (v)=0 \right\} \right|;\\
		\textbf{M}^\otimes_{i,v}(H)&:=\left|\left\{ s \in [0,\textbf{S}^\otimes_i(H)]:\; X_{i,s}^\otimes(v)=1 \textnormal{ and } X_{i,s-}^\otimes (v)=0 \right\} \right| .
		\end{split}
		\end{equation*}
		
		\item We further set the total infections at leaves to be
		\begin{equation*}
		\begin{split}
		\textbf{M}^\sharp(T) &:= \sum_{v\in\mathcal{L}} \textbf{M}_v^\sharp(T) \quad \textnormal{and} \quad M^\sharp(T)=\E_{\textsc{cp}} \textbf{M}^\sharp(T);	 \\
		\textbf{M}^\otimes_i(T) &:= \sum_{v\in\mathcal{L}} \textbf{M}^\otimes_{i,v}(T) \quad \textnormal{and} \quad M^\otimes_i(T) =\E_{\textsc{cp}} \textbf{M}^\otimes_i(T).
		\end{split}
		\end{equation*}
	\end{itemize}
	
	By a standard coupling between $(X^\sharp_t)$ and $(X_t)\sim\cp^\lambda_{\rho^+}(T;\one_\rho)$ based on their graphical representations, we have $M^\sharp(T) \geq M^l(T)$.
	
	Moreover, in the perspectives of (\ref{eq:Sbound geometric trials}), the number of excursions of $\{(X_{i,t}^\otimes) \}$ included in a single excursion of $(X^\sharp_t)$ is the same as a geometric random variable with success probability $(1+\lambda D)^{-1}$. Therefore, $\lambda D$ excursions of $\{(X_{i,t}^\otimes) \}$ happen in expectation during an excursion of $(X^\sharp_t)$. Since the initial condition $\one_{v_i}$ of the product chain is selected uniformly at random at each excursion, we see that
	\begin{equation}\label{eq:Msharp and Motimes}
	M^\sharp(T) = \lambda D\times \frac{\sum_{i=1}^D M_i^\otimes(T)}{D} = \lambda D M^\otimes(T),
	\end{equation}
	where we define $M^\otimes(T)$ as the arithmetic mean of $\{M^\otimes_i(T)\}$.
	
	Now we attempt to control $M^\otimes(T)$ in terms of $\{M(T_i) \}$. To this end, let $(X^\otimes_t) \sim \cp^\otimes(T;\zero)$, and we observe that due to the same argument as (\ref{eq:MvsMprime pivot aux}) and (\ref{eq:MvsMprime pivot quant}),
	\begin{equation}\label{eq:Motimes integral}
	\begin{split}
	\lim_{t_0\rightarrow \infty} \frac{1}{t_0} \E_{\textsc{cp}} &
	\left[\,
	\left|\left\{ s \in [0,t_0]:\; X_{s}^\otimes(v)=1 \textnormal{ and } X_{s-}^\otimes (v)=0 \right\} \right| \,
	\right]\\
	&=
	\frac{M^\otimes(T)}{(\lambda D)^{-1} + S^\otimes(T)}.
	\end{split}
	\end{equation}
	On the other hand, let $\mathcal{L}_i = \{v\in T_i: \textnormal{dist}(v,v_i) = l-1 \}$ for each $i\in [D]$ and $X^{(i)}_t \sim \cp^\lambda_\rho (T_i; \zero)$ be the restriction of $(X_t^\otimes)$ on $T_i$. Note that
	\begin{equation*}
	\begin{split}
	&\sum_{v\in\mathcal{L}}\left|\left\{ s \in [0,t_0]:\; X_{s}^\otimes(v)=1 \textnormal{ and } X_{s-}^\otimes (v)=0 \right\} \right| \\
	&= \sum_{i=1}^D\sum_{v\in \mathcal{L}_i}
	\left|\left\{ s \in [0,t_0]:\; X_{s}^{(i)}(v)=1 \textnormal{ and } X_{s-}^{(i)} (v)=0 \right\} \right|.
	\end{split}
	\end{equation*}
	Thus, by the same reasoning as above, the l.h.s. of (\ref{eq:Motimes integral}) is also equal to
	\begin{equation*}
	\sum_{i=1}^D \frac{M^{l-1}(T_i)}{\lambda^{-1} + S(T_i)}.
	\end{equation*}
	Therefore, by (\ref{eq:Sotimes product}), (\ref{eq:Motimes integral}) implies that
	\begin{equation*}
	M^\otimes(T) = \frac{1}{D} \sum_{i=1}^D M^{l-1}(T_i) \prod_{\substack{1\,\leq\, j\,\leq\, D\\j\neq i}} (1+\lambda S(T_j)).
	\end{equation*}
	Finally, combining this with (\ref{eq:Msharp and Motimes}) and $M^\sharp(T) \geq M^l(T)$, we deduce the conclusion.
\end{proof}

\subsection{The second recursive inequality}\label{subsec:tit:deterministic}

The goal of this section is to obtain the recursive inequality on $M^l(T)$ that is in parallel to (\ref{eq:recursion eq typical}). In Definition \ref{def:tit}, $\textbf{M}^l(T)$ and $M^l(T) $ were defined with respect to the root-added contact process $\cp^\lambda_{\rho^+} (T;\one_\rho).$ We begin with defining  $\bar{\textbf{M}}^l(T)$ and $ \bar{M}^l (T)$ for the (usual) contact process $(\bar{X}_t)\sim \cp^\lambda(T;\one_\rho)$ similarly as Definition \ref{def:tit}.

\begin{itemize}
	\item Let $(T,\rho)$ be a rooted tree with depth $l$ and $\mathcal{L}:= \{v\in T: \textnormal{dist}(\rho,v)=l \}$ be the collection of depth-$l$ leaves. For each $v\in \mathcal{L}$, we define $\bar{\textbf{M}}^l_v(T)$ to be the \textit{total infections at} $v$, that is,
	\begin{equation*}
	\bar{\textbf{M}}^l_v(T):= \left| \left\{ s\in [0,\textbf{R}(T)]: \; \bar{X}_s(v)=1 \textnormal{ and } \bar{X}_{s-}(v)=0 \right\}\right|,
	\end{equation*}
	where $\textbf{R}(T)$ is the survival time of $(\bar{X}_t)$.

	\item $\bar{\textbf{M}}^l(T)$ is the \textit{total infections at depth-$l$ leaves} with respect to $(\bar{X}_t)\sim \cp^\lambda(T;\one_\rho)$, given by 
	\begin{equation*}
	\bar{\textbf{M}}^l(T) := 
	\sum_{v\in \mathcal{L}} \bar{\textbf{M}}^l_v(T).
	\end{equation*}
	As before, we set $\bar{\textbf{M}}^{l'}(T)\equiv 0$ for $l'>l.$
	
	\item $\bar{M}^l(T) := \E_{\textsc{cp}} \bar{\textbf{M}}^l(T)$ is the \textit{expected total infections at depth-$l$ leaves} with respect to $\cp^\lambda(T;\one_\rho)$. We also set $\bar{M}^{l'}(T)=0$ for $l'>l$.
\end{itemize}

For  $(T,\rho)$ as above, let $D=\deg(\rho)$ and $v_1,\ldots,v_D$ be the children of $\rho$. We denote the subtrees at each child by $T_1,\ldots, T_D$ as before. Here, each subtree has depth at most $l-1$. We begin with obtaining the following recursive inequality on $\bar{M}^l(T)$ which is parallel to Proposition \ref{prop:recursion of R} on $R(T)$.

\begin{corollary}\label{cor:recursion Mprime}
	Under the above setting, let $R(\cdot)=\E_{\textsc{cp}} \textbf{R}(\cdot)$ and assume that $\lambda^2 \sum_{i=1}^D R(T_i) < 1$. Then, we have
	\begin{equation}\label{eq:Mprime recursion typical}
	\bar{M}^l(T) \leq \frac{\lambda \sum_{i=1}^D \bar{M}^{l-1}(T_i)}{1-\lambda^2 \sum_{i=1}^D R(T_i)}.
	\end{equation}
\end{corollary}

\begin{proof}
	When running $\cp^\lambda(T;\one_\rho)$, recall the notion of the \textit{first  and second round} discussed in the proof of Proposition \ref{prop:recursion of R}. In the first round of infection, the expected total infections at leaves on $T$ is bounded by
	$$\lambda \sum_{i=1}^D \bar{M}^{l-1}(T_i) ,$$
	due to the same reason as before. Also, the expected number of infections sent from the children $\{v_i\}$ to $\rho$ in the first round is at most
	$$\lambda^2 \sum_{i=1}^D R(T_i). $$
	Thus, we get
	\begin{equation*}
	\bar{M}^l(T) \leq {\lambda \sum_{i=1}^D \bar{M}^{l-1}(T_i)} +\left\{\lambda^2 \sum_{i=1}^D R(T_i)\right\}\bar{M}^l(T),
	\end{equation*}
	which leads to our conclusion.
\end{proof}

We would like to translate this result to a recursion on $M^l(T)$, and hence we discuss  the relation between $M^l(T)$ and $\bar{M}^l(T)$, which is analogous to Proposition \ref{prop:S vs R}. Recall the definition of $S(T)$ from Definition \ref{def:excursion and survival time}.

\begin{corollary}\label{cor:MvsMprime}
	On a finite rooted tree $(T,\rho)$ of depth $l$, let $M^l(T)$ and $\bar{M}^l(T)$ be defined as above and let $S(T)$ denote the expected excursion time on $T$. Then, we have
	\begin{equation}\label{eq:MvsMprime ineq}
	M^l(T) \leq (1+\lambda S(T)) \bar{M}^l(T).
	\end{equation}	
\end{corollary}

\begin{proof}
	We rely on a similar idea as in Proposition \ref{prop:S vs R}. Let  $\mathcal{L} = \{v \in T: \textnormal{dist}(v,\rho) = l \}$ and $(X_t) \sim \cp^\lambda_{\rho^+}(T;\one_\rho)$ be the root-added contact process. For  $v\in\mathcal{L}$, consider the following quantity:
	\begin{equation}\label{eq:MvsMprime pivot aux}
	\lim_{t_0\rightarrow \infty} \frac{1}{t_0} \sum_{v\in\mathcal{L}} \E_{\textsc{cp}} \left[\, \left|\left\{s\in [0,t_0]:\,  X_s(v)=1 \textnormal{ and } X_{s-}(v)=0 \right\}\right| \,\right].
	\end{equation}
	First, observe that the above is the same as
	\begin{equation}\label{eq:MvsMprime pivot quant}
	\lim_{t_0\rightarrow \infty} \frac{ \textnormal{the number of excursions in } (X_t)\textnormal{ during }[0,t_0] }{ t_0} \times M^l(T) = \frac{M^l(T)}{\lambda^{-1} +S(T)},
	\end{equation}
	since the sum of the expected excursion time and the expected waiting time until the next excursion is $S(T) + \lambda^{-1}$. On the other hand,  since the rate-$\lambda$ Poisson process rings $\lambda t_0$ times in expectation during $[0,t_0]$,  the l.h.s.$\;$of (\ref{eq:MvsMprime pivot quant}) is at most
	$\lambda \bar{M}^l(T), $
	by  following the discussion in Proposition \ref{prop:S vs R}. 
\end{proof}

Let us combine (\ref{eq:Mprime recursion typical}) and (\ref{eq:MvsMprime ineq}) to deduce a recursion on $M^l(T)$ as follows.
\begin{equation*}
M^l(T) \leq \frac{1+\lambda S(T)}{1-\lambda^2 \sum_{i=1}^D S(T_i) } \lambda \sum_{i=1}^D M^{l-1}(T_i). 
\end{equation*}
By plugging in (\ref{eq:recursion eq typical}) in the r.h.s., we obtain the first recursion for $M^l(T)$, namely,
\begin{equation}\label{eq:Mrecursion typical}
M^l(T) \leq \frac{\lambda\sum_{i=1}^D M^{l-1}(T_i)}{1-\lambda - 2\lambda^2 \sum_{i=1}^D S(T_i)},
\end{equation}
which is valid given that $\lambda + 2\lambda^2 \sum_{i=1}^D S(T_i)<1$.

\subsection{Recursive tail estimate for Galton-Watson trees}\label{subsec:tit:gw}

The goal of this subsection is to establish Theorem \ref{thm:M tail bd}. As one may expect, we go through similar steps as Theorem \ref{thm:tail bound of S} with some appropriate adjustments. 

Let $\xi$ be a positive, integer-valued random variable with mean $d=\E \xi$, and let $(\mathcal{T}^l,\rho) \sim \gw(\xi)^l$. Further, we denote  the children of the root by $v_1,\ldots,v_D$ where $D=\deg(\rho)$.  The subtree of $\mathcal{T}^l$ rooted at $v_i$ is denoted by $T_i$. Set $\lambda = (1-\ep)d^{-1}$ as in the assumption.

We establish the theorem by an induction on $l$. The initial case $l=0$ is obvious, since $\mathcal{T}^0$ is a single vertex and $M^0(\mathcal{T}^0) = 1$ by its definition (Definition \ref{def:tit}). From now on, suppose that the conclusion holds for $l$, and we attempt to prove it for $l+1$. 	As before, we split the inequality into three cases, namely,
\begin{enumerate}
	\item [1.] (small) $2/\ep \leq t\leq d^{\frac{1}{10}}$;
	
	\item[2.] (intermediate) $d^{\frac{1}{10}}\leq t\leq \exp(\frac{1}{2}c_1\sqrt{d})$;
	
	\item[3.] (large) $\exp(\frac{1}{2}c_1\sqrt{d})\leq t$.
\end{enumerate}

\noindent (The reason for the choice of $d^{\frac{1}{10}}$ is explained in Remark \ref{rmk:small intermediate regime choice}.) For convenience, we define
\begin{equation}\label{eq:tildeMdef}
\widetilde{M}(\mathcal{T}^{l+1}) = \left(1-\frac{\ep}{10} \right)^{l+1} M^{l+1}(\mathcal{T}^{l+1}), \quad \textnormal{and} \quad
\widetilde{M}(T_i) = \left(1-\frac{\ep}{10} \right)^l M^l(T_i).
\end{equation}
\subsubsection{Proof of Theorem \ref{thm:M tail bd} for small $t$}\label{subsubsec:M1} 

Suppose that
\begin{equation*}
\sum_{i=1}^D \widetilde{M}(T_i) \leq (2+\ep)d, \quad \textnormal{and} \quad \sum_{i=1}^D S(T_i) \leq \frac{2d}{\ep} \left(1+\frac{\ep}{3} \right).
\end{equation*}
Then, the second recursive inequality (\ref{eq:Mrecursion typical}) tells us that
\begin{equation*}
\widetilde{M}(\mathcal{T}^{l+1}) \leq \left(1-\frac{\ep}{10} \right)^{-1}  
\frac{(1-\ep)(2+\ep) }{1-\frac{1}{d}-\frac{4}{\ep d} (1+\frac{\ep}{3}) } \leq 2,
\end{equation*}
where the last inequality holds for all $d\geq d_0'(\ep)$ for appropriate $d_0'(\ep)$. Therefore, for $t\geq 2$, we have
\begin{equation}\label{eq:Mrecursion small split}
\begin{split}
\Pgw \left(\widetilde{M}(\mathcal{T}^{l+1}) \geq t \right)
&\leq
\P\left(D\geq \left(1+\frac{\ep}{6}\right)d\right)
+
\Pgw \left(\sum_{i=1}^{(1+\frac{\ep}{6})d } \widetilde{M} (T_i) \geq (2+\ep) d  \right)\\
&+
\Pgw \left(\sum_{i=1}^{(1+\frac{\ep}{6})d} S(T_i) \geq \frac{2d}{\ep }\left(1+\frac{\ep}{3} \right) \right),
\end{split}
\end{equation}
where $T_i, \;i\in\mathbb{N}$ are i.i.d. $\gw(\xi)^l$. The first term in the r.h.s. is estimated by the concentration condition (\ref{eq:concentration condition}), namely,
\begin{equation}\label{eq:Dbd for M small}
\P\left(D\geq \left(1+\frac{\ep}{6}\right)d\right)
\leq
\exp(-c_{\ep'} d)
\leq
\frac{1}{4} d^{-\frac{1}{10}\sqrt{d}} \left(\frac{1}{10}\log d\right)^{-2},
\end{equation}
for $\ep'=\frac{\ep}{6}$, where the last inequality holds if $d\geq d_0'(\ep)$ for appropriate constant $d_0'(\ep)>0$. Then, Lemma \ref{lem:sum Si small t} controls the last term in the r.h.s.$\;$of (\ref{eq:Mrecursion small split}), implying the bound (\ref{eq:tail condition for small t}). For the second term, we claim that there exists $d_1(\ep)>0$ such that if $d\geq d_1(\ep)$, then 
\begin{equation}\label{eq:Msum small}
\Pgw\left(\sum_{i=1}^{(1+\frac{\ep}{6})d } \widetilde{M} (T_i) \geq (2+\ep) d\right) \leq \frac{1}{4} d^{-\frac{1}{10}\sqrt{d}} \left(\frac{1}{10}\log d\right)^{-2}.
\end{equation}
Indeed, almost the same argument from the proof of Lemma \ref{lem:sum Si small t} can be applied to deduce (\ref{eq:Msum small}). The only two changes we need are to set
\begin{equation*}
\widehat{M}(T_i) = \widetilde{M}(T_i) - 2,
\end{equation*}
(in the lemma, it was $Z_i' = Z_i-\frac{2}{\ep}.$) and to split $\widehat{M}$ based on $\frac{\ep}{10}$, not $\frac{1}{10}$. Also, Lemma \ref{lem:sum Si small t} holds if $d\geq d_1$ for an absolute constant $d_1$, but here $d_1$ depends on $\ep$. We omit the remaining details.

Finally, we obtain the conclusion (\ref{eq:Mtail:in thm}) for $l+1$ and $2\leq t\leq d^{\frac{1}{10}}$ by combining (\ref{eq:Dbd for M small}), Lemma \ref{lem:sum Si small t} and (\ref{eq:Msum small}).

\subsubsection{Proof of Theorem \ref{thm:M tail bd} for intermediate $t$}\label{subsubsec:M2}

For $t\geq d^{\frac{1}{10}}$, we rely on the second recursive inequality (\ref{eq:Mrecursion atypical}). However, one issue here is that the  the quantity in the r.h.s. of (\ref{eq:Mrecursion atypical}) is no longer a single product of i.i.d. random variables as (\ref{eq:recursion eq atypical}), (\ref{eq:tail bd of S intermed t 0}). To overcome this difficulty, we define
\begin{equation}\label{eq:Widef}
W_i := \max\left\{ \widetilde{M}(T_i), \; S(T_i) \right\}, \quad \textnormal{for each } i\in \mathbb{N},
\end{equation}
for $\widetilde{M}(T_i)$ as in (\ref{eq:tildeMdef}).
Then, from a little bit of algebra we see that
\begin{equation*}
\begin{split}
\lambda \sum_{i=1}^D \widetilde{M}(T_i) \prod_{\substack{1\,\leq\, j\, \leq\, D\\ j\neq i}} (1+\lambda S(T_j))
\;&\leq\;
\sum_{i=1}^D \lambda W_i \prod_{\substack{1\,\leq\, j\, \leq\, D\\ j\neq i}} (1+\lambda W_j)\\
&\leq
\; \prod_{i=1}^D (1+2\lambda W_i).
\end{split}
\end{equation*}
Based on this observation, we attempt to control
\begin{equation}\label{eq:Mrecursion intermed split}
\begin{split}
\Pgw \left(\widetilde{M}(\mathcal{T}^{l+1}) \geq t \right)
&\leq
\Pgw\left( \sum_{i=1}^D \log(1+2\lambda W_i ) \;\geq
\; \log\left\{\left(1-\frac{\ep}{10} \right)t\right\}  \right)\\
&\leq
\Pgw\left( \sum_{i=1}^D \log(1+2\lambda W_i ) \;\geq
\; \log\left(\frac{t}{2} \right)  \right).
\end{split}
\end{equation}

We are also aware of the tail estimate on $W_i$ by the induction hypothesis and Theorem \ref{thm:tail bound of S}, namely,
\begin{equation}\label{eq:Witail}
\P(W_i \geq s) \leq \P\left(\widetilde{M}(T_i) \geq s\right) +
\P(S(T_i) \geq s) \leq 2s^{-\sqrt{d}} (\log s)^{-2},\quad\textnormal{for all }s\geq \frac{2}{\ep}.
\end{equation}
This falls into the assumption of Lemma \ref{lem:recursion principle atypical}. Therefore, we deduce the tail estimate on $\widetilde{M}(\mathcal{T}^{l+1})$ for intermediate $t$ by
\begin{equation}\label{eq:Mrecursion intermed computation}
\begin{split}
\Pgw \left(\widetilde{M}(\mathcal{T}^{l+1})  \geq t\right)
&\leq
\P(D\geq 2d)
+
\Pgw\left( \sum_{i=1}^{2d} \log(1+2\lambda W_i ) \;\geq
\; \log\left(\frac{t}{2} \right)  \right)\\
&\leq
\exp(-c_1 d) + \frac{1}{2} t^{-\sqrt{d}} (\log t)^{-2}
\leq t^{-\sqrt{d}} (\log t)^{-2},
\end{split}
\end{equation}
where we used the concentration condition (\ref{eq:concentration condition}) to bound the tail of $D$, and the last inequality is true if $d\geq d_0'(\ep)$ for appropriate constant $d_0'(\ep)>0$.

\subsubsection{Proof of Theorem \ref{thm:M tail bd} for large $t$}\label{subsubsec:M3}

We continue to rely on $W_i$ defined from (\ref{eq:Widef}). Set
$$\Delta_{d,t} = \frac{4\sqrt{d}}{c_1}\log t.$$
For large $t$, we modify (\ref{eq:Mrecursion intermed split}) as in 
(\ref{eq:tail bd of S large split}), which is
\begin{equation}\label{eq:Mrecursion large split}
\begin{split}
\Pgw\left(\widetilde{M}(\mathcal{T}^{l+1})\geq t\right)
\leq&
\;\P \left(D \geq  \Delta_{d,t} \right) +
\Pgw \left(\sum_{i=1}^{2d} \log (1+2\lambda W_i) \geq \log\left(\frac{t}{2}\right)  \right)\\
&+
\sum_{2d\, \leq\, r\, \leq\, \Delta_{d,t} } \P(D \geq r) 
\times \Pgw \left(\sum_{i=1}^r \log(1+2\lambda W_i ) \geq \log \left(\frac{t}{2} \right) \right).
\end{split}
\end{equation}

\noindent Note that $t/2$ in the log in the r.h.s. is a lower bound of $(1-\frac{\ep}{10})t$ (cf. (\ref{eq:Mrecursion intermed computation})). The first term in the r.h.s. can be controlled by the concentration condition (\ref{eq:concentration condition}). By (\ref{eq:Witail}), we can use Lemma \ref{lem:recursion principle atypical} to obtain the estimate for the second, and the third term is bounded by the same reasoning as in Section \ref{subsubsection3}. This concludes the proof of Theorem \ref{thm:M tail bd}. \qed

\section{Short survival on random graphs: proof of Theorem \ref{thm:phasetransition:graph}}\label{sec:shortsurvival randomgraphs}

Let $\mu$ be a given degree distribution satisfying (\ref{eq:concentration condition}) for some positive constants $\mathfrak{c}= \{c_\delta \}_{\delta\in(0,1]}$.  Define $\widetilde{\mu}$ to be its size-biased distribution (\ref{eq:def:sizebiased}) and $d=\E_{D\sim \widetilde{\mu}} D.$  In this section, we are interested in the contact process on $G_n \sim\mathcal{G}(n,\mu)$, particularly on its short survival. Our goal is to establish the following theorem, which implies the lower bound of Theorem \ref{thm:phasetransition:graph}, that is,
$$\limsup_{k\rightarrow \infty } \;\limsup_{n\rightarrow \infty} \lambda_c^-\left(\mu_k \right) d_k \geq 1. $$

\begin{theorem}\label{thm:lowerbd of lamdac}
	Let $\ep\in(0,1)$ and $\mathfrak{c}=\{c_\delta \}_{\delta\in (0,1]}$ be a collection of positive constants. Then there exists $d_0(\ep, \mathfrak{c}) >0$ such that the following holds true: Let  $\mu$ be a probability measure on $\mathbb{N}$ whose sized-biased distribution $\widetilde{\mu}$ satisfies $d:=\E_{D\sim \widetilde{\mu}} D \geq d_0$ and the concentration condition (\ref{eq:concentration condition}) with $\mathfrak{c}$. Further, let $\lambda = (1-\ep) d^{-1}$ and $G_n\sim \mathcal{G}(n,\mu) $. Then, there exists an event $\mathcal{E}_n$ over graphs such that $\P(G_n \in \mathcal{E}_n) = 1- o(1)$ and 
	\begin{equation}\label{eq:R on G:in thm}
	\Pcp \left(\left.\textbf{R}(G_n) \leq n^2 \;\right|\, G_n\in\mathcal{E}_n \right) = 1-o(1),
	\end{equation}
	where  $\textbf{R}(G_n) $ is the survival time of $\cp^\lambda(G_n; \one_{G_n})$.
\end{theorem}

To establish the theorem, we would like to  study the structure of local neighborhoods $N(v,l)$ of $G_n\sim \mathcal{G}(n,\mu)$. Although a neighborhood selected uniformly at random converges weakly to a Galton-Watson tree, there are some neighborhoods who contain a cycle. In Section \ref{subsec:unic}, we extend the properties from Sections \ref{sec:subcritical trees} and \ref{sec:numberofhits} to  certain Galton-Watson-type  random graphs with a cycle that are relevant to the local neighborhoods of $G_n$.  Then in Section \ref{subsec:coupling local nbd}, we develop a coupling between the local neighborhoods and the aforementioned graphs in Section \ref{subsec:unic}, following the ideas from \cite{BNNASurvival}, Section 4.1. Finally, we conclude the proof of Theorem \ref{thm:lowerbd of lamdac} and establish the lower bound of Theorem \ref{thm:phasetransition:graph} in Section \ref{subsec:pf of thm2}, based on all the properties we obtained in the previous sections.

\subsection{Recursive analysis for unicyclic graphs}\label{subsec:unic}

 In this subsection, we do the final preliminary work before delving into the proof of Theorem \ref{thm:phasetransition:graph}. Although we need to consider the neighborhoods $N(v,l)$ inside  $G_n$ that contain a cycle, fortunately, it turns out that it is enough to look at the case with exactly one cycle (see the discussion in Section \ref{subsubsec:idea2-2} for a sketchy review, or Section \ref{subsec:coupling local nbd} for a detailed explanation). Therefore,  we are interested in the Galton-Watson type processes with a single cycle, particularly the ones which are introduced in this subsection. 

\begin{definition}[Galton-Watson-on-cycle process of type one]
	Let $\xi$ be a positive, integer-valued random variable, and let $m,l\geq 1$ be nonnegative integers. Then, $\mathcal{H}^{m,l} \sim \gwc^1(\xi, m)^l$, the \textit{Galton-Watson-on-cycle process of type one} (in short, $\gwc^1$-\textit{process}), is generated according to the following procedure:
	\begin{enumerate}
		\item [1.] Consider a length-$m$ cycle $C= v_1 v_2 \ldots v_m v_1$.
		
		\item[2.] At each $v_j$ for $j=1,\ldots,m$, attach $\mathcal{T}_j^l\sim$ i.i.d.$\; \gw(\xi)^l$ by setting $v_j$ as its root. 
	\end{enumerate}
	The resulting graph is called $\mathcal{H}^{m,l}$. We designate vertex $v_1$ as the root of $\mathcal{H}^{m,l}$ and denote $\rho = v_1$. Note that $m=1$ corresponds to the usual Galton-Watson trees.
\end{definition}

We are again interested in the excursion time and the total infections at leaves on the $\gwc^1$-processes. For concreteness, we present the following definitions.

\begin{definition}\label{def:SMunicyclic}
	Let $m,l\geq 1$ be integers, and $H$ be a graph that consists of a length-$m$ cycle $C=v_1v_2\ldots v_m v_1$ and  depth $\leq l$ trees $T_1,\ldots,T_m$ rooted at $v_1,\ldots,v_m$, respectively (recall the definition of \textit{tree depth} in (\ref{eq:depth def})).  
	\begin{enumerate}
		\item [1.] The \textit{root-added contact process} $\cp^\lambda_{v_1^+} (H;x)$ is the contact process on the graph $H\cup{v_1^+}$ with the permanently infected parent $\rho^+$ having a single connection with $v_1$, and with the initial condition $x\in \{0,1\}^H$. 
		
		\vspace{2mm}
		\item [2.] $\textbf{S}(H)$ (resp., $\textbf{R}(H)$) is the \textit{excursion} (resp. \textit{survival}) \textit{time}, which is the first time when $\cp^\lambda_{v_1^+}(H;\one_{v_1})$ (resp., $\cp^\lambda(H;\one_{v_1})$) returns to the all-healthy state. $S(H)=\E_{\textsc{cp}} \textbf{S}(H)$ (resp., $R(H) = \E_{\textsc{cp}} \textbf{R}(H)$) denotes  the \textit{expected excursion} (resp., \textit{survival}) \textit{time}.
		
		\vspace{2mm}
		\item [3.] Let $\mathcal{L}_j = \{v\in T_j : \textnormal{dist}(v,v_j)=l \}$, $(X_t) \sim \cp^\lambda_{v_1^+} (H;\one_{v_1}) $ and $(\bar{X}_t)\sim \cp^\lambda(H;\one_{v_1})$. For $v\in \mathcal{L}_j$ for some $j$, we define \textit{the total infections at} $v$ by
		\begin{equation*}
		\begin{split}
		\textbf{M}^l_v(H)&:=\left|\left\{s\in [0,\textbf{S}(H)]:\; X_s(v)=1 \textnormal{ and } X_{s-}(v)=0 \right\}\right| ;\\
		\bar{\textbf{M}}^l_v(H)&:=\left|\left\{s\in [0,\textbf{R}(H)]:\; \bar{X}_s(v)=1 \textnormal{ and } \bar{X}_{s-}(v)=0 \right\}\right|.
		\end{split}
		\end{equation*}
		Then, the \textit{total infections at leaves} and its expectation are given as
		\begin{equation*}
		\begin{split}
		\textbf{M}^l(H) &:= \sum_{j=1}^m \sum_{v\in\mathcal{L}_j} \textbf{M}^l_v(H), \quad \textnormal{and} \quad M^l(H) := \E_{\textsc{cp}}\textbf{M}^l(H) ;\\
		\bar{\textbf{M}}^l(H) &:= \sum_{j=1}^m \sum_{v\in\mathcal{L}_j} \bar{\textbf{M}}^l_v(H), \quad \textnormal{and} \quad \bar{M}^l(H) := \E_{\textsc{cp}}\bar{\textbf{M}}^l(H).
		\end{split}
		\end{equation*}
		Note that $\textbf{M}^{l}(H)\equiv\bar{\textbf{M}}^{l}(H)\equiv 0$, $M^{l}(H)= \bar{M}^{l}(H)=0$ if all the depths of $T_1,\ldots,T_m$ are  smaller than $l$.
		
	\end{enumerate}
\end{definition}

The goal of this section is to establish the following theorem.

\begin{proposition}\label{prop:SMunicyclic}
	Let $m,l\geq 1$ be  integers,
	$\ep\in(0,1)$ and $\mathfrak{c}=\{c_\delta \}_{\delta\in (0,1]}$ be a collection of positive constants. Then there exists $d_0(\ep, \mathfrak{c}) >0$ such that the following holds true. For any $\xi$ that satisfies $d:=\E \xi \geq d_0$ and (\ref{eq:concentration condition}) with $\mathfrak{c}$, we have for $\lambda=(1-\ep) d^{-1}$ and $\mathcal{H}^{m,l}\sim \gwc^1(\xi,m)^l$ that
	\begin{equation}\label{eq:unicyc tail:in prop}
	\begin{split}
	\P_{\textsc{gw}} \left(S(\mathcal{H}^{m,l}) \geq  t \right) &\leq 
	3t^{-\sqrt{d}} (\log t)^{-2} ~~~\textnormal{for all }t\geq \frac{2}{\ep};
	\\
	\P_{\textsc{gw}} \left(M^l(\mathcal{H}^{m,l}) \geq \left(1- \frac{\ep}{10} \right)^l t \right) &\leq 
	3t^{-\sqrt{d}} (\log t)^{-2} ~~~\textnormal{for all }t\geq 2,
	\end{split}
	\end{equation}
	where $S(\mathcal{H}^{m,l})$, $M^l(\mathcal{H}^{m,l})$ are given as in Definition \ref{def:SMunicyclic}.
\end{proposition}

The main idea of the proof is the same as what we saw in Theorems \ref{thm:tail bound of S} and \ref{thm:M tail bd}. However, as one can easily expect, the analysis becomes much more complicated due to the existence of a cycle. In particular, we cannot apply the tree recursion techniques directly.  Due to its technicality, the proof of Proposition \ref{prop:SMunicyclic} is presented in Appendix \ref{sec:unicyclic}.

 Next, we introduce the \textit{Galton-Watson-on-cycle process of type two} (in short, $\gwc^2$-process), which can
be thought as a certain subgraph of $\gwc^1$-processes. Although it is a very similar object to $\gwc^1$-process,  the \textit{root-added process} on $\gwc^2$ is defined in a different way, as presented in the following definitions. 

\begin{definition}[Galton-Watson-on-cycle process of type two]\label{def:gwc2}
	Let $\xi$ be a positive, integer-valued random variable, and let $m\geq 2$, $l\geq 1$ be integers. Then, $\dot{\mathcal{H}} \sim \gwc^2(\xi, m)^l$, the \textit{Galton-Watson-on-cycle process of type two} (in short, $\gwc^2$-\textit{process}), is generated according to the following procedure:
	\begin{enumerate}
		\item [1.] Consider a length-$m$ cycle $C= v_1 v_2 \ldots v_m v_1$.
		
		\item[2.] At each $v_j ,\; j\in\{2,\ldots,m\}$, attach $\mathcal{T}_j^l\sim$ i.i.d.$\; \gw(\xi)^l$ by setting $v_j$ as its root. At $v_1$, we do nothing.
	\end{enumerate}
	The resulting graph is called $\dot{\mathcal{H}}$. We designate vertex $v_1$ as the root of $\dot{\mathcal{H}}$ and denote $\rho = v_1$. 
\end{definition}

\begin{remark}
	The $\gwc^2$-process is the same object as the $\gwc$-process defined in \cite{BNNASurvival}, Section 4.
\end{remark}

\begin{definition}[Root-added contact process on $\gwc^2$-processes]\label{def:SMdef gwc2}
	For $\dot{\mathcal{H}} \sim \gwc^2(\xi, m)^l$, we define the \textit{root-added contact process} on $\dot{\mathcal{H}}$ without adding a new parent to the root. Instead, we fix the root $\rho=v_1$ to be permanently infected by itself, and we denote this process by $\cp^\lambda_{v_1} (\dot{\mathcal{H}}^{m,l} ; x )$, for an initial configuration $x\in\{0,1\}^{\dot{\mathcal{H}}\setminus \{v_1\} }$. Then, we define  $S_2(\dot{\mathcal{H}})$, $M^l_2(\dot{\mathcal{H}})$ (resp., $S_m(\dot{\mathcal{H}})$, $M^l_m(\dot{\mathcal{H}})$) analogously as Definition \ref{def:SMunicyclic}, with respect to $\cp_{v_1}^\lambda(\dot{\mathcal{H}};\one_{v_2})$ (resp., $\cp_{v_1}^\lambda(\dot{\mathcal{H}};\one_{v_m})$). We also write
	\begin{equation*}
	S(\dot{\mathcal{H}}) = \frac{1}{2} \left(S_2(\dot{\mathcal{H}})+ S_m(\dot{\mathcal{H}}) \right), \quad
	\textnormal{and} \quad
	M^l(\dot{\mathcal{H}}) = \frac{1}{2} \left(M^l_2 (\dot{\mathcal{H}}) + M^l_m (\dot{\mathcal{H}}) \right).
	\end{equation*}
\end{definition}

Then, the analog of Proposition \ref{prop:SMunicyclic} can be derived on $\gwc^2$-process as follows. 

\begin{corollary}\label{cor:SMgwc2}
	Let $m\geq 2$, $l\geq 1$ be  integers,
	$\ep\in(0,1)$ and $\mathfrak{c}=\{c_\delta \}_{\delta\in (0,1]}$ be a collection of positive constants. Then there exists $d_0(\ep, \mathfrak{c}) >0$ such that the following holds true. For any $\xi$ that satisfies $d:=\E \xi \geq d_0$ and (\ref{eq:concentration condition}) with $\mathfrak{c}$, we have for $\lambda=(1-\ep) d^{-1}$ and $\dot{\mathcal{H}}^{m,l}\sim \gwc^2(\xi,m)^l$ that
	\begin{equation}\label{eq:gwc2 tail:in prop}
	\begin{split}
	\P_{\textsc{gw}} \left(S(\dot{\mathcal{H}}^{m,l}) \geq  t \right) &\leq 
	4t^{-\sqrt{d}} (\log t)^{-2} ~~~\textnormal{for all }t\geq \frac{3}{\ep};
	\\
	\P_{\textsc{gw}} \left(M^l(\dot{\mathcal{H}}^{m,l}) \geq \left(1- \frac{\ep}{10} \right)^l t \right) &\leq 
	4t^{-\sqrt{d}} (\log t)^{-2} ~~~\textnormal{for all }t\geq 3,
	\end{split}
	\end{equation}
	where $S(\dot{\mathcal{H}}^{m,l})$, $M^l(\dot{\mathcal{H}}^{m,l})$ are given as in Definition \ref{def:SMdef gwc2}.
\end{corollary}

As before, we explain the proof of the corollary in Appendix \ref{subsec:unic prop large t} due to its technicality.

\subsection{Coupling the local neighborhood of random graphs}\label{subsec:coupling local nbd}

Let $G_n \sim \mathcal{G}(n,\mu)$ and $\widetilde{\mu}$ be the size-biased distribution of $\mu$. Define $\gw(\mu,\mu')^l$ to be the Galton-Watson tree of depth $l$ such that the offspring distribution of the root is $\mu$ while that of all other descendants is $\mu'$.

For a fixed vertex $v$ in $G_n$, $N(v,l)$ converges locally weakly to $\gw(\mu,\widetilde{\mu})^l$ as $n\to \infty$, as we briefly saw in Lemma \ref{lem:lwc} and the explanation below it. However, the standard coupling between $N(v,l)$ and $\gw(\mu,\widetilde{\mu})^l $ always has an error at least $\Theta(n^{-1})$. To diminish this error, we introduce the notion of \textit{augmented distribution} (Definition 4.2, \cite{BNNASurvival}), which allows us to stochastically dominate $N(v,l)$ by a larger geometry. 
For our purpose, the definition will be slightly different from \cite{BNNASurvival} to yield a smaller distortion of the mean of $\mu$.
Note that if $\{p_k\}_{k\in \mathbb{N}} $ satisfies $\sum_k p_k e^{ck} <\infty$ for some $c>0$, then it also has $\sum_k k\sqrt{p_k}  <\infty$ by Cauchy-Schwarz inequality (see, for instance, Lemma 4.1 of \cite{BNNASurvival}).

\begin{definition}[Augmented distribution, \cite{BNNASurvival}]\label{def:aug}
	Let $\ep\in(0,1)$ and $\mu\equiv \{p_k\}_{k\in\mathbb{N}}$ be a probability measure on $\mathbb{N}$ such that $\E_{D\sim \mu} e^{cD} <\infty$ for some $c>0$.  Let $k_0 := \max \{k: \sum_{j\geq k} j\sqrt{p_j} \geq \ep/10 \}$ (which is finite by the above discussion),  and $k_{\textnormal{max}} :=\max\{k:p_k>0 \}$, with $k_{\textnormal{max}}=+\infty$ if the maximum does not exist. If $k_0 < k_{\textnormal{max}}$, we define the \textit{augmented distribution} $\mu^\sharp = \mu^\sharp_{\ep}$ of $\mu$ by
	\begin{equation*}
	\begin{split}
	\mu^\sharp (k) := \frac{1}{Z} 
	\begin{cases}
	\left(1-\frac{\ep}{10} \right)p_k  &\textnormal{if } k\leq k_0 ;\\
	\sqrt{p_k} &\textnormal{if } k>k_0,
	\end{cases}
	\end{split}
	\end{equation*}	
	where $Z= \sum_{k\leq k_0} (1-\frac{\ep}{10}) p_k+ \sum_{k>k_0} \sqrt{p_k}$ is the normalizing constant. When $k_0 = k_{\textnormal{max}}$, we set
	\begin{equation*}
	\mu^\sharp (k) := \frac{1}{Z} 
	\begin{cases}
	\left(1-\frac{\ep}{10} \right)p_k  &\textnormal{if } k< k_0 ;\\
	\sqrt{p_k} &\textnormal{if } k=k_0,
	\end{cases}
	\end{equation*}
	for the normalizing constant $Z= \sum_{j<k_0} (1-\frac{\ep}{10}) p_k + \sqrt{p_{k_0}}$.
\end{definition}

Suppose that we generated the i.i.d. degrees $\{D_i\}_{i\in[n]}$ of $G_n$. Consider the exploration procedure starting from a fixed vertex $v$ and its half-edges, which picks an unmatched half-edge in the \textit{explored} neighborhood of $v$, reveals its pair half-edge uniformly at random among all the \textit{unmatched} half-edges and absorbs the matched vertex. If  the matched half-edge  is from the \textit{unexplored} half-edges, then we include the half-edges adjacent to the matched vertex in our explored neighborhood. and its (unmatched) half-edges into the explored neighborhood. Then, during the early steps of exploration, the number of newly added (unmatched) half-edges is roughly distributed as $\widetilde{\mu}$, as long as we discover a new half-edge out of the previous explored neighborhood. 

However, as we mentioned at the beginning of this subsection, their exact distributions are not precisely $\widetilde{\mu}$. The role of the augmented distribution is to provide a unified law that stochastically dominates the number of newly explored half-edges in all the early steps.  The following lemma describes this property.

\begin{lemma}[Lemma 4.3, \cite{BNNASurvival}]\label{lem:aug properties}
	Let $\ep \in (0,1)$ and $\mathfrak{c}=\{c_\delta \}_{\delta\in(0,1]}$ be a collection of positive constants. Then, there exists $d_0(\ep,\mathfrak{c})>0$ such that the following holds. For any  probability measure $\nu$ on $\mathbb{N}$ having $d:=\E_{D\sim \nu} D \geq d_0$ and  the concentration condition (\ref{eq:concentration condition}) with $\mathfrak{c}=\{c_\delta \}_{\delta\in(0,1]}$, 
	\begin{enumerate}
		\item [\textnormal{1.}] Let $\nu^\sharp := \nu^\sharp_\ep$ and $d^\sharp:=\E_{D^\sharp\sim\nu^\sharp} D^\sharp$ be its mean. Then, $$d^\sharp \leq \left(1+\frac{\ep}{9} \right)d. $$ 
		Further, let  $\widetilde{d}$ (resp. $\widetilde{d}^\sharp$) be the mean of $\widetilde{\nu}$ (resp. $\widetilde{\nu}^\sharp$), the size distribution of $\nu$ (resp. $\nu^\sharp$). Then, $\widetilde{d}^\sharp \leq (1+\frac{\ep}{9}) \widetilde{d}$.

		\item [\textnormal{2.}] There is a collection of positive constants $\{ c'_{\ep'}\}_{\ep'\in[\frac{\ep}{10}, 1]}$ such that $\nu^\sharp$ satisfies
		\begin{equation}\label{eq:concentrationcond2}
		\begin{split}
		\P_{D^\sharp \sim \nu^\sharp} \left( D^\sharp \geq (1+\ep') d \right)
		&	\leq 
		\exp\left(-c'_{\ep'} d \right),\quad  \textnormal{for all } \ep' \in \left[\frac{\ep}{10}, 1\right] 
		;\\
		\P_{D^\sharp \sim \nu^\sharp} \left( D^\sharp \geq (1+a) d \right)
		&	\leq 
		\exp\left(-c'_{1} ad \right),\quad  \textnormal{for all } a\geq 1.
		\end{split}
		\end{equation}
		
		\item [\textnormal{3.}] Let $\{D_i\}_{i\in [n]}$ be a collection of i.i.d. samples of $\nu$. For a subset $\Delta \subset [n]$, let $\{p_k^{\Delta} \}_{k \in \mathbb{N}}$ denote the empirical distribution of $\{D_i \}_{i\in [n] \setminus \Delta}$. \textsf{Whp} over the choice of $\{D_i\}_{i\in[n]}$, $\{p_k^{\Delta} \}_{k \in \mathbb{N}}$ is stochastically dominated by $\nu^\sharp$, for any $\Delta \in[n]$ with $|\Delta| \leq n^{\frac{1}{2}}$.
		
	\end{enumerate}
\end{lemma}

The third property in the lemma is almost analogous to that of Lemma 4.3, \cite{BNNASurvival}. In Section \ref{subsec:pfof aug}, we discuss the proof of the lemma, focusing on the aspects which are different from \cite{BNNASurvival}.

Now we want to derive a coupling between the local neighborhoods of $G_n\sim \mathcal{G}(n,\mu)$ and the Galton-Watson-type processes. The first thing we should handle
is to control the number of cycles in a local neighborhood $N(v,l)$. For a fixed constant $\gamma>0$, we define the event $\mathcal{E}^C_n(\gamma)$ over the graphs with $n$ vertices to be
\begin{equation*}
\mathcal{E}^C_n (\gamma) := \{G_n : \forall v\in G_n, \; N(v,\gamma \log n) \textnormal{ contains at most one cycle} \}.
\end{equation*}
Then, we adopt the following lemma that shows the event $\mathcal{E}^C_n(\gamma)$ is indeed typical for some $\gamma >0$. The proof of the lemma is presented in Section \ref{subsec:pfof 1cyc} in the Appendix, due to its similarity to \cite{BNNASurvival}, Lemma 4.5 and \cite{ls10}, Lemma 2.1.

\begin{lemma}[\cite{BNNASurvival,ls10}]\label{lem:1cyc}
	Let $\ep \in (0,1)$ and $\mathfrak{c}=\{c_\delta \}_{\delta\in(0,1]}$ be a collection of positive constants. Then, there exists a constant $d_0(\ep,\mathfrak{c})>0$ such that the following holds. For any  probability measure $\mu$ on $\mathbb{N}$ having $d:=\E_{D\sim \mu} D \geq d_0$ and  the concentration condition (\ref{eq:concentration condition}) with $\mathfrak{c}=\{c_\delta \}_{\delta\in(0,1]}$, we have for $G_n \sim \mathcal{G}(n,\mu)$ that
	\begin{equation}\label{eq:1cyc:in lem}
	\P \left(G_n \in \mathcal{E}^C_n \left(\frac{1}{10\log d} \right) \right) = 1-o(1).
	\end{equation}
\end{lemma}

\begin{remark}
	In \cite{BNNASurvival}, Lemma 4.5, it was proven that
	$\P(G_n \in \mathcal{E}^C_n (\gamma)) = 1-o(1)$ for some constant $\gamma$ depending on $d$ in a rather implicit sense. Here, we have an additional assumption of (\ref{eq:concentration condition}) and it makes it possible to deduce Lemma \ref{lem:1cyc}, a stronger result. The improved information on the constant $\gamma$ turns out to be crucial in Section \ref{subsec:pf of thm2} when proving Theorem \ref{thm:lowerbd of lamdac}.
\end{remark}

Set $\gamma_1 = \frac{1}{10}$ and $l_n:= \frac{\gamma_1}{2}\log_d n$. The remaining task of this subsection is to define the Galton-Watson type process which will be coupled with the local neighborhoods $N(v,l_n)$. The notion was already introduced in \cite{BNNASurvival}, Definition 4.7.

\begin{definition}[Edge-added Galton-Watson process, \cite{BNNASurvival}]
	Let $h,m,l$ be nonnegative integers such that $m\geq 2$ and $l\geq h+1$, and let $\nu$ be a probability distribution on $\mathbb{N}$. We define the \textit{Edge-added Galton-Watson process} (in short, \egw-\textit{process}), denoted by $\egw(\nu;h,m)^l$ as follows.
	\begin{itemize}
		\item [1.] Generate a $\gw(\nu)^l$ tree, conditioned on survival until depth $l$. The root $\rho$ of this tree is also the root of $\egw(\nu;h,m)^l$.

		\item [2.] At each vertex $u$ at depth $h$, add an independent $\gwc^2(\nu,m)^{l-h}$ process (see Definition \ref{def:gwc2} rooted at $u$. Here we preserve the existing subtrees at $u$ which comes from $\gw(\nu)^l$ tree from Step 1.
	\end{itemize}
	Let $\nu'$ be another probability measure on $\mathbb{N}$. Then, $\egw(\nu,\nu';h,m)^l$ denotes the \egw-process whose root has offspring distribution $\nu$ while all other descendants have $\nu'$. Here we also add $\gwc^2(\nu',m)^{l-h}$ in Step 2 of the definition. We also remark that  $\egw(\nu;0,m)^l= \gwc^1(\nu,m)^l.$
\end{definition}

We develop a coupling between $N(v,\gamma_1 \log_d n)$ and the \egw-processes. To this end, we define the notion of stochastic domination between two probability measures $\eta, \eta'$ over rooted graphs: we write $\eta \leq_{\textsc{st}} \eta'$ if there exists a coupling $((\mathcal{G},\rho), (\mathcal{G}',\rho'))$ such that  $(\mathcal{G},\rho) \sim \eta$, $(\mathcal{G}',\rho')\sim\eta'$ and $\mathcal{G}\subset \mathcal{G}'$, that is, there is a graph isomorphism that maps $\rho$ to $\rho'$ and  embeds $\mathcal{G}$ into $\mathcal{G}'$.

Let $v$ be a fixed vertex in $G_n\sim \mathcal{G}(n,\mu)$, and recall that $\gamma_1=\frac{1}{10}$ and $l_n= \frac{\gamma_1}{2} \log_d n$. We also abbreviate $\mathcal{E}^C_n = \mathcal{E}^C_n(\gamma_1/\log d)$. In addition to $\mathcal{E}^C_n$, we define $\mathcal{B}_{h,m}(v)$ to be the subevent of $\mathcal{E}^C_n$ such that $N(v,l_n)$ contains a cycle of length $m$ that is at distance $h$ from $v$.

Moreover, for a given $\ep\in(0,1)$ we set ${\mu}^\sharp=\mu^\sharp_\ep$ to be the augmented distribution of $\mu$, and $\widetilde{\mu}^\sharp$ to be the size-biased distribution of $\mu^\sharp.$ Further, let $\eta$, $\eta_0$ and $\eta_{h,m}$ denote the law of $N(v,l_n)$, $\gw(\mu^\sharp, \widetilde{\mu}^\sharp)^{l_n}$ and $\egw(\mu^\sharp, \widetilde{\mu}^\sharp;h,m)^{l_n}$, respectively. Then, Lemma 4.8 of \cite{BNNASurvival} provides us the following coupling lemma, which plays a crucial role in settling the lower bound of Theorem \ref{thm:phasetransition:graph} in the following subsection.

\begin{lemma}[Lemma 4.8, \cite{BNNASurvival}]\label{lem:couplinglocalnbd}
	Under the above setting, let $b_{h,m}:=\P(\mathcal{B}_{h,m}(v)) $ and $b_0 := 1-\sum_{h,m}b_{h,m}$. Then we have
	\begin{equation*}
	\eta \one_{\mathcal{E}^C_n} \leq_{\textsc{st}}
	b_0\eta_0 + \sum_{h\geq 0, m\geq 2} \eta_{h,m}.
	\end{equation*}
\end{lemma}

So far, we have collected almost all the elements we need in the proof of Theorem \ref{thm:phasetransition:graph}. The last thing we want is the tail probability estimates for \egw-processes. To be concrete, we begin with defining the quantities of interest.

\begin{definition}[Excursion time and total infections at leaves for \egw]\label{def:SMegw}
	Let $h,m,l$ be nonnegative integers with $m\geq 2$, $\nu, \nu'$ be probability measures on $\mathbb{N}$, and $\mathcal{H}\sim \egw(\nu,\nu';h,m)^l$. We connect a permanently infected parent $\rho^+$ to the root $\rho$ of $\mathcal{H}$, and the contact process on the resulting graph is the \textit{root-added contact process} $\cp^\lambda_{\rho^+}(\mathcal{H};x)$ on $\mathcal{H}$ with initial condition $x\in\{0,1 \}^\mathcal{H}$.
	\begin{itemize}
		\item   The \textit{excursion time} $\textbf{S}(\mathcal{H})$ is the first time when $\cp^\lambda_{\rho^+}(\mathcal{H};\one_\rho)$ becomes all-healthy on $\mathcal{H}$, and  $S(\mathcal{H}) = \E_{\textsc{cp}} \textbf{S}(\mathcal{H})$ denotes the \textit{expected excursion time} on $\mathcal{H}$.
		
		\item Let $\mathcal{L}$ be the collection of \textit{bottom leaves} of $\mathcal{H}$, that is, denoting $\{C_j\}$ to be the length-$m$ cycles in $\mathcal{H}$,
		\begin{equation*}
		\mathcal{L}= \{v\in\mathcal{H}: \textnormal{dist}(v,\rho)\geq l \textnormal{ and dist}(v,C_j)\geq l-h \textnormal{ for all }j \}.
		\end{equation*}
		
		\item Let $v\in \mathcal{L}$. Denoting $(X_t)\sim \cp^\lambda_{\rho^+} (\mathcal{H};\one_\rho)$, we define
		the \textit{total infections at} $v$ by
		\begin{equation*}
		\textbf{M}^l_v(\mathcal{H}):=
		\left|\left\{s \in [0, \textbf{S}(\mathcal{H})]:\;  
		X_s(v)=1 \textnormal{ and } X_{s-}(v)=0
		\right\}\right|.
		\end{equation*}
		Then, we set  the \textit{total infections  at depth-$l$ leaves}  as
		\begin{equation*}
		\textbf{M}^l(\mathcal{H}) := \sum_{v\in\mathcal{L}} \textbf{M}^l_v(\mathcal{H}).
		\end{equation*}

		\item We also let $M^l(\mathcal{H}):=\E_{\textsc{cp}} \textbf{M}^l(\mathcal{H})$ to be the \textit{expected total infections at leaves}.
	\end{itemize}
\end{definition}

For an \egw-process $\mathcal{H}$, the tail probabilities of $S(\mathcal{H})$ and $M^l(\mathcal{H})$ can be estimated using Theorems \ref{thm:tail bound of S}, \ref{thm:M tail bd}, and Proposition \ref{prop:SMunicyclic}, which can be described as follows.

\begin{proposition}\label{prop:egw tail}
	Let $m,l,h$ be nonnegative integers such that $m\geq 2$ and $l\geq h+1$.
	Also, let	$\ep\in(0,1)$ and $\mathfrak{c}=\{c_\delta \}_{\delta\in (0,1]}$ be a collection of positive constants. Then, there exists $d_0(\ep, \mathfrak{c}) >0$ such that the following holds true. For any $\mu$ that satisfies $d:=\E \xi \geq d_0$ and (\ref{eq:concentration condition}) with $\mathfrak{c}$, let $\mu^\sharp:=\mu^\sharp_\ep$, and set $\widetilde{\mu}^\sharp$ to be  its size-biased distribution  (Definition \ref{def:aug}). Then, we have for $\lambda=(1-\ep) d^{-1}$ and $\mathcal{H}\sim \egw(\mu^\sharp, \widetilde{\mu}^\sharp;h,m)^l$ that
	\begin{equation}\label{eq:egw tail:in prop}
	\begin{split}
	\P_{\textsc{gw}} \left(S(\mathcal{H}) \geq  t \right) &\leq 
	3t^{-\sqrt{d}} (\log t)^{-2} ~~~\textnormal{for all }t\geq \frac{2}{\ep};
	\\
	\P_{\textsc{gw}} \left(M^l(\mathcal{H}) \geq \left(1- \frac{\ep}{10} \right)^l t \right) &\leq 
	3t^{-\sqrt{d}} (\log t)^{-2} ~~~\textnormal{for all }t\geq 2,
	\end{split}
	\end{equation}
	where $S(\mathcal{H}^{m,l})$, $M^l(\mathcal{H}^{m,l})$ are given as in Definition .
\end{proposition}

\begin{proof}
	Along with (\ref{eq:egw tail:in prop}), we establish the same inequalities for $\mathcal{H}\sim \egw( \widetilde{\mu}^\sharp;h,m)^l$ in tandem, by an induction on $h$. When $h=0$, we have (\ref{eq:egw tail:in prop}) for both $\egw(\widetilde{\mu}^\sharp;0,m)^l$ and $\egw(\mu^\sharp,\widetilde{\mu}^\sharp;h,m)^l$,
	since the former one is the same as $ \gwc^1(\widetilde{\mu}^\sharp,m)^l$ and the latter is
	$\mathcal{H}\sim \gwc^1 (\mu^\sharp,\widetilde{\mu}^\sharp,m)^l$, the $\gwc^1$-process in which only the root of the Galton-Watson tree rooted at $v_1$ has offspring $\mu^\sharp$ and all others have offspring $\widetilde{\mu}^\sharp$. By Lemma \ref{lem:aug properties}, the mean $\tilde{d}:=\E_{D^\sharp\sim \widetilde{\mu}^\sharp} D^\sharp$ of $\widetilde{\mu}^\sharp$ is such that
	\begin{equation*}
	\tilde{d} \leq \left(1+\frac{\ep}{9}\right)d,
	\end{equation*}
	and $\widetilde{\mu}^\sharp$ also satisfies  (\ref{eq:concentrationcond2}) with $\mathfrak{c}'$ as in the lemma. Note that this implies (\ref{eq:concentration condition}) for $\widetilde{\mu}^\sharp.$ 
	Therefore, Proposition \ref{prop:SMunicyclic} gives (\ref{eq:egw tail:in prop}) for $h=0$ and all $m\geq 2$, $l\geq 0$. 
	
	Suppose that we have (\ref{eq:egw tail:in prop}) for $\mathcal{H}'\sim \egw(\widetilde{\mu}^\sharp;h,m)^l$ with $m\geq 2$, $h\geq 0$ and $l\geq h+1$. Let $\rho_1$ be the root of $\mathcal{H}_1\sim \egw(\widetilde{\mu};h+1,m)^l$ and $D_1\sim \widetilde{\mu}^\sharp$ be its degree. Then the subgraphs $H_1^1,\ldots , H_{D_1}^1$  rooted at each child of $v_1$  are i.i.d. $\egw(\widetilde{\mu}^\sharp;h,m)^{l-1}$. Since we have the tail probability estimates for $H_1^1,\ldots,H_{D_1}^1$ by the induction hypothesis, we follow the proof of Theorems \ref{thm:tail bound of S} (Sections \ref{subsubsection1}---\ref{subsubsection3}) and
	\ref{thm:M tail bd} (Sections \ref{subsubsec:M1}---\ref{subsubsec:M3}) to deduce (\ref{eq:egw tail:in prop}) for $\mathcal{H}_1$. The result for $\mathcal{H}_2 \sim \egw(\mu^\sharp,\widetilde{\mu}^\sharp;h+1,m)^l$ follows similarly, where the only difference  is that $\mathcal{H}_2$ has subgraphs $H^2_1,\ldots, H^2_{D_2} \sim$ i.i.d. $\egw(\widetilde{\mu^\sharp}; h,m)^{l-1}$ with $D_2 \sim \mu^\sharp.$ 
\end{proof}

\subsection{Proof of Theorem \ref{thm:phasetransition:graph}}\label{subsec:pf of thm2}

In this subsection, we conclude the proof of the lower bound of Theorem \ref{thm:phasetransition:graph}, by showing Theorem \ref{thm:lowerbd of lamdac}.

Let $\ep \in(0,1)$ and a collection of positive constants $\mathfrak{c}=\{c_\delta \}_{\delta \in (0,1]}$ be given, and let $d_0=d_0(\ep, \mathfrak{c})$ be the maximal  among those from Theorems \ref{thm:tail bound of S}, \ref{thm:M tail bd} and Proposition \ref{prop:egw tail}. For $\mu$ and its size-biased distribution with mean $d:= \E_{D\sim \mu'}$, let $\gamma_1 \in(0, \frac{1}{10})$ be a constant satisfying Lemma \ref{lem:1cyc}, that is,
\begin{equation}\label{eq:gamma1}
\P_{G_n \sim \mathcal{G}(n,\mu)} \left(G_n \in \mathcal{E}^C_n \right) = 1-o(1), \quad \textnormal{with } \mathcal{E}^C_n = \mathcal{E}^C_n\left(\frac{\gamma_1}{\log d} \right) \textnormal{ as in the lemma,}
\end{equation}
and set $l_n = \lfloor\frac{\gamma_1}{2} \log_d n\rfloor$.

In the remaining of this section, we assume that $G_n\sim \mathcal{G}(n,\mu)$ satisfies $\mathcal{E}^C_n$. For each vertex $v\in G_n $, we define the \textit{block} $B_v$ at $v$ as follows.
\begin{itemize}
	\item $B_v := N(v,l_n)$, if $N(v,l_n)$ is a tree.
	
	\item If $N(v,l_n)$ contains the unique cycle $C$ at distance $h$ from $v$, then
	$$B_v:= N(v,l_n)  \cup \left[\bigcup_{u\in C}  N(u, l_n -h) \right]. $$
\end{itemize}
Note that we always have $B_v\subset N(v,\gamma_1 \log_d n)$, so that $B_v$ can only have at most one cycle. 

Let $\mu^\sharp=\mu^\sharp_\ep$ be the augmented distribution (Definition \ref{def:aug}) of $\mu$ and $\widetilde{\mu}^\sharp$ be its size-biased distribution. Also, set $\eta_0, \eta_{h,m}$ to be the laws of $\gw(\mu^\sharp, \widetilde{\mu}^\sharp)^{l_n}$, $\egw(\mu^\sharp,\widetilde{\mu}^\sharp; h,m)^{l_n}$, respectively, and let  $\mathcal{H}$ be a sample from 
$$\mathcal{H} \;\sim\;\; b_0\eta_0 + \sum_{h\geq 0, m\geq 2} \eta_{h,m}, $$
where the constants $b_0$ and $\{b_{h,m}\}_{h\geq 0, m\geq 2} $ are given as in Lemma \ref{lem:couplinglocalnbd}. Then, Lemma \ref{lem:couplinglocalnbd} tells us that for fixed $v$, we can couple $N(v, l_n)$ and $\mathcal{H}$ so that $N(v, l_n) \subset \mathcal{H}$ (in terms of isomorphic embeddings). Furthermore, the same proof as in the lemma gives  the coupling between $B_v$ and $\mathcal{H}$ as follows. We refer \cite{BNNASurvival}, Lemma 4.8 for its proof.

\begin{corollary}\label{cor:couplingblock}
	Under the above setting, let $v$ be a fixed vertex in $G_n$ and $\eta'$ denote the law of $B_v$. Then, we have
	\begin{equation}\label{eq:couplingblock:in cor}
	\eta' \one_{\mathcal{E}^C_n} \leq_{\textsc{st}} b_0\eta_0 + \sum_{h\geq 0, m\geq 2} \eta_{h,m}.
	\end{equation}
\end{corollary}

Further, let us consider the the collections of the \textit{bottom leaves} of $B_v$. To this end, let $C$ denote the cycle in $B_v$ if exists, and let 
\begin{equation}\label{eq:def blockbottomleaves}
\mathcal{L}(B_v) := \begin{cases}
\{u\in B_v : \textnormal{dist}(u,v) \geq l_n \} &\textnormal{if } C \textnormal{ does not exist};\\
\{u\in B_v : \textnormal{dist}(u,v) \geq l_n \textnormal{ and } \textnormal{dist}(u,C) \geq l_n-h \} &\textnormal{if } C \textnormal{ exists}, 
\end{cases}
\end{equation}
where $h =\textnormal{dist}(v,C)$. Also, set $\mathcal{L}(\mathcal{H})$ to be the bottom leaves of $\mathcal{H}$, where we follow  Definition \ref{def:SMegw} if $\mathcal{H}$ is an \egw-process. An important thing to note is that if $B_v \subset \mathcal{H}$, then we have $$\mathcal{L}(B_v) \subset \mathcal{L}(\mathcal{H}), $$
in terms of isomorphic embedding. That is, there exists an isomorphism $\phi$ mapping $B_v$ into $\mathcal{H}$ such that $\phi(v) = \rho$, $\phi(\mathcal{L}(B_v)) \subset \mathcal{L}(\mathcal{H})$. 

Since $\mathcal{H}$ is a mixture of Galton-Watson and $\egw$-processes, Theorems \ref{thm:tail bound of S}, \ref{thm:M tail bd} and Proposition \ref{prop:egw tail} imply the tail probabilities of $S(\mathcal{H})$ and $M^{l_n}(\mathcal{H})$. Thus, $R(B_v)$ and $\bar{M}^{l_n}(B_v)$, the expected survival time and the expected total infections at leaves with respect to $\cp^\lambda(B_v;\one_v)$ (Definition \ref{def:SMegw}), satisfy
\begin{equation}\label{eq:blockSM tail}
\begin{split}
&\P (R(B_v)\geq t ) \leq 
\P \left(S(\mathcal{H}) \geq t \right) \leq 3t^{-\sqrt{d}}(\log t)^{-2}, \quad \textnormal{for all } t\geq \frac{2}{\ep};\\
&\P\left(\bar{M}^{l_n}(B_v) \geq \left(1-\frac{\ep}{10} \right)^{l_n} t\right) \\
&\leq
\P \left(M^{l_n}(\mathcal{H}) \geq \left(1-\frac{\ep}{10} \right)^{l_n} t \right)
\leq 
3t^{-\sqrt{d}} (\log t)^{-2}, \quad \textnormal{for all }t\geq 2.
\end{split}
\end{equation}

Based on the above discussions, we can deduce the following lemma.

\begin{lemma}\label{lem:RMblocks}
	Under the above setting, define the events $\mathcal{E}^R_n$ and $\mathcal{E}^M_n$ over the graphs with $n$ vertices by
	\begin{equation*}
	\begin{split}
	\mathcal{E}^R_n &:= 
	\{ R(B_v) \leq n^{2/\sqrt{d}}, \textnormal{ for all } v\in G_n
	\}\\
	\mathcal{E}^M_n
	&:=
	\{\bar{M}^{l_n}(B_v) \leq (\log n)^{-1}, \textnormal{ for all } v\in G_n
	\}.
	\end{split}
	\end{equation*}
	Then, $\P(G_n \in \mathcal{E}^R_n \cap \mathcal{E}^M_n) = 1-o(1).$
\end{lemma}

\begin{proof}
	Note that (\ref{eq:blockSM tail}) implies 
	\begin{equation*}
	\begin{split}
	\P \left(R(B_v) \geq n^{2/\sqrt{d}} \right) \leq n^{-2},
	\end{split}
	\end{equation*}
	and also
	\begin{equation*}
	\begin{split}
	\P \left(\bar{M}^{l_n}(B_v) \geq \frac{1}{\log n} \right)
	&\leq \P \left( M^{l_n}(\mathcal{H}) \geq \left(1-\frac{\ep}{10}\right)^{l_n} \frac{\exp(\frac{\ep\gamma_1}{20} \log_{d}n)}{\log n} \right)\\
	&\leq \P \left( M^{l_n}(\mathcal{H}) \geq \left(1-\frac{\ep}{10}\right)^{l_n} \exp\left(\frac{\ep\gamma_1}{30} \log_{d}n\right) \right) \\
	&\leq \exp\left(-\frac{\ep\gamma_1\sqrt{d}}{30\log d }\log n \right) \leq n^{-2},
	\end{split}
	\end{equation*}
	where the second and the last inequalities are true for $d$  larger than some  constant depending on $\ep$ (note that $\gamma_1$ was an absolute constant). Therefore, the conclusion follows from a  union bound over all vertices.
\end{proof}



We conclude this subsection by establishing  Theorem \ref{thm:lowerbd of lamdac}.

\begin{proof}[Proof of Theorem \ref{thm:lowerbd of lamdac}]

	Let $\mathcal{E}^C_n, \mathcal{E}^R_n$ and $\mathcal{E}^M_n$ be the events over graphs of $n$ vertices defined in Lemmas \ref{lem:1cyc} and \ref{lem:RMblocks}. Further, 
	set $\gamma_1\in(0,\frac{1}{10})$ as in (\ref{eq:gamma1}) and define 
	$$\mathcal{E}_n := \mathcal{E}^C_n \cap \mathcal{E}^R_n \cap \mathcal{E}^M_n .$$
	Then, Lemmas \ref{lem:1cyc} and \ref{lem:RMblocks}  tell us that
	$$\P(G_n \in \mathcal{E}_n) = 1-o(1).$$
	In the remaining proof, we settle (\ref{eq:R on G:in thm}) for $\mathcal{E}_n$ as above.
	
	Let  $\textbf{R}_v(G_n)$ be the survival time of  $\cp^\lambda(G_n;\one_v)$. Then, the standard coupling argument (also called graphical representation, see, for instance, Section 2 of \cite{BNNASurvival} for a brief review, and \cite{liggett:ips} for a detailed introduction) tells us that
	\begin{equation}\label{eq:RleqsumRv}
	\textbf{R}(G_n) = \max\{\textbf{R}_v(G_n) : v\in G_n \}\leq \sum_{v\in G_n} \textbf{R}_v(G_n) .
	\end{equation}
	
	To investigate $\textbf{R}_v(G_n)$, we introduce a \textit{decomposition} of $\cp^\lambda(G_n;\one_v)$ using the \textit{blocks} $\{B_u\}_u$.  Recall the graphical representation from Section \ref{subsubsec:cpprelim} and let $\{N_x(t)\}_{x\in V(G_n)}$ and $\{N_{\vec{xy}}(t)\}_{\vec{xy}\in \vec{E}(G_n)}$ be the Poisson processes defining the recoveries and infections of $\cp^\lambda(G_n;\one_{v})$, respectively. Also, recall the definition of the bottom leaves $\mathcal{L}(B_u)$ defined in (\ref{eq:def blockbottomleaves}).
	
	 \begin{definition}[Decomposition]\label{def:decomposition}
		Let $G_n$ and $\{B_u\}_{u\in G_n}$ be as above. The \textit{decomposition} of $\cp^\lambda(G_n;\one_{v})$ by $\{B_u\}_u$ is the coupled process generated  as follows.
	
		\begin{itemize}
			\item [1.] Initially, run $\cp^\lambda(B_v;\one_v)$ whose recoveries and infections are given by $\{(N_x(s))_{s\geq 0}\}_{x\in V(B_v)}$ and $\{(N_{\vec{xy}}(s))_{s\geq 0}\}_{\vec{xy}\in \vec{E}(B_v)}$.
			
			\item[2.] In $B_v$, when some $u\in \mathcal{L}(B_v)$ becomes infected in $B_v$ at time $t$ (and has been healthy until time $t-$), initiate a copy of $\cp^\lambda(B_u;\one_u)$ with event times given by $\{(N_x(s))_{s\geq t}\}_{x\in V(B_u)}$ and $\{(N_{\vec{xy}}(s))_{s\geq t}\}_{\vec{xy}\in \vec{E}(B_u)}$.
			
			\item[3.] Repeat Step 2 to every running copies of $\{\cp^\lambda(B_u;\one_u)\}_{u\in G_n}$ until the process terminates, that is, when all vertices in every generated copy are healthy.
		\end{itemize}
	\end{definition}
	
	 Since the copies generated in this definition shares the event times with the original process, it is just another way of interpreting $\cp^\lambda(G_n;\one_{v})$. In particular, both $\cp^\lambda(G_n;\one_{v})$ and its decomposition terminate at the same time. Also, even though the generated copies can be highly correlated with each other, the law of each copy is the same as the contact process with intensity $\lambda$. 	In this approach, we can see that controlling the number of infections at $\mathcal{L}(B_u)$ is crucial, which is done by the second and third lines of (\ref{eq:blockSM tail}).
	
	Let $\widetilde{\textbf{R}}_v(G_n)$ denote the termination time of the decomposed process. Since the copies  in the above procedure shares the event times with the original process, we see that $\widetilde{\textbf{R}}_v(G_n) = \textbf{R}_v(G_n)$ and hence
	\begin{equation}\label{eq:Rtilde}
	R_v(G_n):= \E_{\textsc{cp}} \textbf{R}_v(G_n) = \widetilde{R}_v(G_n) := \E_{\textsc{cp}} \widetilde{\textbf{R}}_v(G_n).
	\end{equation} 
	Further, let $\textbf{U}_v(G_n)$ be the enumeration of vertices $u$ of which a copy of $\cp^\lambda(B_u;\one_u)$ has been generated during this process. If the copies of $\cp^\lambda(B_u;\one_u)$ for $u$ appeared multiple times, we include $u$ multiple times in $\textbf{U}_v(G_n)$, respecting the multiplicities. Then, we observe that
	\begin{equation}\label{eq:Rtilde2}
	\widetilde{\textbf{R}}_v(G_n) \leq \sum_{u\in \textbf{U}_v(G_n)} \textbf{R}_u(B_u),
	\end{equation}
	where we have a slight abuse of notation in the r.h.s.$\;$since we omitted the initiation times of the copies $\cp^\lambda(B_u;\one_u)$, $u\in \textbf{U}_v(G_n)$.
	
	On the event $\mathcal{E}^M_n$, the expected number of new copies generated in Step 2 and 3 from a single $\cp^\lambda(B_u;\one_u)$ is at most $(\log n)^{-1}$, uniformly in $u$. This gives that $$\E_{\textsc{cp}} \textbf{U}_v(G_n) \leq \sum_{k=0}^\infty (\log n)^{-k} \leq 2,$$
	for large enough $n$. Moreover, on $\mathcal{E}^R_n$, each generated copy survives at most $n^{2/\sqrt{d}}$-time in expectation, so we have from (\ref{eq:Rtilde}) and (\ref{eq:Rtilde2}) that for all $v\in G_n$,
	\begin{equation*}
	R_v(G_n) \leq \E_{\textsc{cp}} \widetilde{\textbf{R}}_v(G_n) \leq 2n^{2/\sqrt{d}}.
	\end{equation*}
	Thus, (\ref{eq:RleqsumRv}) gives $R(G_n) \leq 2n^{1+2d^{-1/2}}$, and hence by Markov's inequality we get
	\begin{equation*}
	\Pcp \left(\left. \textbf{R}(G_n) \geq n^2 \,\right|\,G_n\in\mathcal{E}_n \right) 
	\leq \frac{1}{\sqrt{n}} = o(1),
	\end{equation*}
	for any $d$ larger than some absolute constant. This finishes the proof of Theorem \ref{thm:lowerbd of lamdac}.
\end{proof}

\section{Supercritical phase for general distribution: proof of Theorems \ref{thm:supercritical:general} and \ref{thm:super:onevertex}}\label{sec:supercrit} 

In this section, we establish Theorems \ref{thm:supercritical:general} and \ref{thm:super:onevertex}. We first observe that the phase transitions for $\gw(\widetilde{\mu})$ and $\gw(\mu, \widetilde{ \mu})$ occur at the same intensity (Lemma \ref{lm:lambda1}) and then show the second inequality in \eqref{eq:thm:super:ER} (Lemma \ref{lem:super:2ndineq}).  After verifying the two lemmas, we prove the first inequality in \eqref{eq:thm:super:ER} and Theorem \ref{thm:super:onevertex}, which is the main goals of this section. Finally, in Section \ref{subsec:finale}, we combine the main results obtained throughout the paper and conclude the proof of Theorems \ref{thm:phasetransition:tree} and \ref{thm:phasetransition:graph}.

\begin{lemma} \label{lm:lambda1}
	We have
	\begin{equation}\label{key}
	\lambda_1^{\textsc{gw}}(\widetilde{\mu}) = \lambda_1^{\textsc{gw}}(\mu, \widetilde{\mu}).\nonumber
	\end{equation}
\end{lemma}
\begin{proof} Let $\lambda>\lambda_1^{\textsc{gw}}(\widetilde{\mu})$. The contact process on $\GW(\mu, \widetilde{\mu})$ with rate $\lambda$ and the root initially infected has a positive probability of reaching a vertex in the first generation. Once such a vertex is infected, the contact process on its subtree has a positive probability to survive forever. Thus, $\lambda>\lambda_1^{\textsc{gw}}(\mu, \widetilde{\mu})$. 
	
	On the other hand, let $\lambda<\lambda_1^{\textsc{gw}}(\widetilde{\mu})$. Let $\tau$ be the (random) survival time of the contact process on $\GW(\widetilde{\mu})$ with the root initially infected. We have
	$$\P(\tau<\infty) = 1.$$
	Let $m\sim \mu$ be the number of vertices in the first generation of $\T \sim \GW(\mu, \widetilde{\mu})$. Let $(X_t)$ be the contact process with rate $\lambda$ on $\T$ and with the root initially infected. Conditioned on $m$, if $\T$ remains infected at time $t$, the probability that $\T$ is healed in a finite time before any infection between the root and its children happens is at least
	$$\P_{\tau'}\left (\textbf{Ber}(e^{-2 m\lambda \tau'})\right )>0$$
	where $\tau' = \max\{\tau_0, \tau_1, \dots, \tau_m\} <\infty$ a.s., $\tau_0\sim \exp(1)$ upper bounds the survival time of the root and $\tau_1, \dots, \tau_m \overset{iid}{\sim}\tau$ upper bound the survival time of the contact process on the $m$ subtrees rooted at the vertices in the first generation after time $t$. Thus, $(X_t)$ dies out with probability 1, proving $\lambda<\lambda_1^{\textsc{gw}}(\mu, \widetilde{\mu})$ and completing the proof.
\end{proof}

So, for the rest of the proof, we simply write $\lambda_1$ for $\lambda_1^{\textsc{gw}}(\widetilde{\mu})$ and $\lambda_1^{\textsc{gw}}(\mu, \widetilde{\mu})$.  Next, we prove the second inequality in \eqref{eq:thm:super:ER}.

\begin{lemma}\label{lem:super:2ndineq}
	Let $\mu$ be a degree distribution satisfying the giant component condition \eqref{eq:condition on mu}, and let $\widetilde{\mu}$ be its size-biased distribution with $d:=\E_{D\sim\widetilde{\mu}}D$. We have
	\begin{equation*}
	\lambda_1^{\textsc{gw}}(\widetilde{ \mu}) \leq \frac{1}{d-1}.
	\end{equation*}
\end{lemma}

\begin{proof}

Let $\lambda>\frac{1}{d-1}$. Let $\T\sim \gw(\widetilde{\mu})$ and $(X_t)$ be the contact process on $\T$ with the root $\rho$ initially infected and infection rate $\lambda$. Construct a subtree $\T'$ rooted at $\rho$ level by level as follows. Let $v$ be a vertex of $\T$ and let $u$ be its parent. The vertex $v$ belongs to $\T'$ if and only if $u\in X_t\cap \T'$ for some $t$ and after the first time that $u$ is infected, it passes its infection to $v$ before being healed. Observe that $\T'$ is a Galton-Watson tree with branching rate $\frac{\lambda d}{1+\lambda}$ and so, as this branching rate is greater than 1, $\T'$ is infinite with positive probability which means that $(X_t)$ lasts forever with positive probability. In other words, $\lambda>\lambda_1^{\textsc{gw}}(\widetilde{\mu})$. Hence, we obtain the conclusion.
\end{proof}

The rest of this section is devoted to the proof of the first inequality in \eqref{eq:thm:super:ER} and the proof of Theorem \ref{thm:super:onevertex}. Let $\lambda>\lambda_1^{\textsc{gw}}(\widetilde{\mu})$. We want to show that \textsf{whp} in $G_n\sim \G(n, \mu)$, the contact process on $G_n$ starting with all vertices infected survives until time $e^{\Theta(n)}$ \textsf{whp} and the contact process starting with one random vertex infected survives until time $e^{\Theta(n)}$ with positive probability. Our strategy is to show that \whp, if there are many ``good" vertices (Definition \ref{def:good:vertex}) infected at certain time $t$, then there are at least twice as many good vertices infected at a later time. Roughly speaking, a good vertex $v$ has an \textit{well-expanding} neighborhood that looks like a Galton-Watson tree. Since $\lambda$ is in the supercritical regime of the tree, with positive probability, the contact process starting from $v$ infects many vertices on the boundary of that neighborhood.

\subsection{Preprocessing the graph} We first preprocess the graph $G_n$ to get rid of high degree vertices. This will be useful to control the number of explored vertices during our exploration of $G_n$ (we refer to the paragraph following Definition \ref{def:aug} for a formal discussion of the term \textit{exploration}). If $\mu$ has an infinite support, we let $\eta_0$, $j_0$, and $j_1$ be positive constants such that $\mu(j_0)>0$, $j_1\in [j_0, e^{j_0/10})$, and
\begin{equation}\label{def:j0}
\mu(j_0, \infty) = \P_{D\sim \mu} \left (D> j_0\right ) \ge 10 \E _{D\sim \mu} D\textbf{1}_{\{D> j_1\}}=: 10\eta_0.
\end{equation}
Note that the requirement $j_1<e^{j_0/10}$ is guaranteed by the assumption that $\mu$ has an exponential tail. If $\mu$ has a finite support, we let $j_0=j_1$ be the largest number in the support of $\mu$ and $\eta_0$ be a sufficiently small constant satisfying 
\begin{equation}\label{def:mu0}
\mu(j_0)\ge 10\eta_0.
\end{equation}

Sample the degrees $D_1, D_2, \dots, D_n$ of $G_n$ independently according to the measure $\mu$. Let $\nu$ be the empirical measure of this sequence $(D_i)$. We have \whp,
\begin{equation}\label{eq:nu:mu:eta0}
|\nu[0, l]-\mu[0, l]|\le \eta_0, \quad\forall l\le j_1 \quad\text{ by Chernoff inequality}
\end{equation}
and 
\begin{equation}\label{eq:sumd:j0}
\sum_{i=1}^{n} D_i \textbf{1}_{\{D_i> j_1\}}\le 2\eta_0 n \quad\text{ by Chebyshev's inequality}.
\end{equation}

Consider the subgraph $G_n'$ obtained from $G_n$ by deleting all vertices with degree greater than $j_1$ together with their matched half-edges. The remaining randomness is the random perfect matching of the remaining half-edges. Let $n'$ be the number of vertices of $G_n'$. By \eqref{eq:sumd:j0}, the number of vertices that are affected by this removal is at most $3\eta_0 n$ and in particular, $n'\ge n-3\eta_0 n$. Let $\nu'$ be the empirical measure of the degree sequence of $G_n'$. We have 
$$|\nu'[0, l]-\nu[0, l]|\le 3\eta_0, \quad\forall l.$$
Combining this with \eqref{eq:nu:mu:eta0} gives 
$$\whp, \quad |\nu'[0, l]-\mu[0, l]|\le 4\eta_0, \quad\forall l\le j_1.$$
Let $\mu_{\eta_0, j_0}$ be the probability measure given by 
\begin{equation}\label{key}
\mu_{\eta_0, j_0}(A) := \begin{cases}
(1-4\eta_0)\frac{\mu(A\cap [0, j_0]) }{\mu[0, j_0]}+ 4\eta_0\textbf{1}_{\{0\in A\}} \quad\text{if $\mu$ has an infinite support},\\
\mu(A\cap[0, j_0))+(\mu(j_0) - 4\eta_0) \text{1}_{j_0\in A} + 4\eta_0\textbf{1}_{\{0\in A\}}\quad\text{if $\mu$ has a finite support}.
\end{cases}
\end{equation}
By \eqref{def:j0} and \eqref{def:mu0}, we have for every $l\le j_0$,
$$\nu'[0, l]\le \max\{\mu[0, l]+4\eta_0, 1\}\le \mu_{\eta_0, j_0}[0, l]$$
and so
\begin{equation}\label{eq:nu':ge:mu}
\whp, \quad \nu'\ge_{st} \mu_{\eta_0, j_0}.  
\end{equation}


Similarly, we have
\begin{lemma}\label{lm:dominate:GW}
	With high probability, for any subset $\Delta$ of vertices in $G_n'$ with at most $\eta_0 n$ vertices, the empirical measure $\nu''$ of the degree sequence in $G_n'\setminus \Delta$ also satisfies $\nu''\ge_{st} \mu_{\eta_0, j_0}$. 
\end{lemma}
Therefore, as long as we have explored at most $\eta_0 n$ vertices, we can use $\mu_{\eta_0, j_0}$ to bound from below the the neighborhood of an unexplored vertex.

Let $\widetilde{\mu}_{\eta_0, j_0}$ be the size-biased distribution of $\mu_{\eta_0, j_0}$.
As $\eta_0\to 0$ (and $j_0\to \infty$ when $\mu$ has an infinite support), $\widetilde{\mu}_{\eta_0, j_0}$ approaches $\widetilde{\mu}$. Let $\tilde d$ be the mean of $\widetilde{\mu}_{\eta_0, j_0}$.

\subsection{Good trees} We define a notion of ``good" trees to describe those trees on which the contact process spreads well: the number of infected vertices grows exponentially. As $\lambda$ is in the supercritical regime of the Galton-Watson tree $ \gw(\widetilde{\mu})$, we then show that with positive probability, a random tree is good.

\begin{definition}[Good trees] \label{def:good:tree}
	A tree $T$ with root $\rho$ is said to be $(s_0, L_0, k, \theta, \ep)$-good if 
	\begin{equation}\label{key}
	\P_{\subcp}\left (|X^{k L_0}_{k s_0}\cap T_{k L_0}| \ge \theta^{k}\right )>\ep
	\end{equation}
	where $\left (X^{k L_0}_{t}\right )\sim \cp^\lambda(T_{\le k L_0}; \one_\rho)$ and $T_{k L_0}$ is the set of vertices at the $k L_0$-generation.
\end{definition}

The following lemma asserts that there is a positive chance for a random tree to be good.
\begin{lemma}\label{lm:good:tree}
	Let $\T\sim \gw(\widetilde{\mu}_{\eta_0, j_0})$. There exist constants $s_0, L_0, \ep>0$ and $\theta>1$ such that
	\begin{equation}\label{eq:good:tree}
	\P_{\subgw}(\T \text{ is $(s_0, L_0, k, \theta, \ep)$-good})>\ep
	\end{equation}
	for all $k$ sufficiently large where the probability is taken over the random graph and the contact process.
\end{lemma}
For the proof of Lemma \ref{lm:good:tree}, we first show that 
\begin{lemma}\label{lm:good:tree:1}
	Let $\T\sim \gw(\widetilde{\mu})$. There exist constants $s_0, L_0, \ep>0$ and $\theta>1$ such that
	\begin{equation}\label{eq:good:tree:1}
	\E (|X^{L_0}_{s_0}\cap \T_{L_0}|)\ge \theta
	\end{equation}
	where  $(X^{L_0}_{t})\sim \cp^\lambda(\T_{\le L_0}; \one_\rho)$ and the expectation is taken over the random graph and the contact process.
\end{lemma}
Assuming Lemma \ref{lm:good:tree:1}, we prove Lemma \ref{lm:good:tree}.
\begin{proof} [Proof of Lemma \ref{lm:good:tree}] As $\widetilde{\mu}_{\eta_0, j_0}\to \widetilde{\mu}$, it suffices to prove \eqref{eq:good:tree} for $\T\sim \GW(\widetilde{\mu})$.
	
	Let $W_k = |X^{k L_0}_{ ks}\cap \T_{k L_0}|$. Then $(W_k)$ dominates a Galton-Watson tree with branching rate $\theta>1$. Recall that for a Galton-Watson tree $(Z_k)$ with branching rate $\theta>1$, there exists a nonzero random variable $Z$ such that 
	$$\frac{|Z_k|}{\theta^{k}}\to Z\quad a.e.$$ 
	Thus, there exists an $\ep>0$ such that for sufficiently large $k$, 
	\begin{equation}\label{eq:Wk:ep}
	\P\left (W_k> \ep\theta^{k}\right )>\ep.
	\end{equation}
	Let $F(\T) = \P_{\subcp}\left (|X^{k L_0}_{ kt}\cap \T_{k L_0}|\ge \ep\theta^{k}  \right )$. We have
	$$\E_{\subgw} F(\T) >\ep.$$
	Since $F(\T)$ is a random variable taking values in $[0,1]$, the above inequality implies \eqref{eq:good:tree} which is 
	$$\P_{\subgw}(F(\T)>\ep/2)>\ep/2$$
	because otherwise, 
	$$\E_{\subgw} F(\T) \le \E_{\subgw} F(\T) \textbf{1}_{F(\T)\ge \ep/2} + \ep/2 \le \ep/2+\ep/2 = \ep$$
	a contradiction.
\end{proof}

\begin{proof}[Proof of Lemma \ref{lm:good:tree:1}]
	Let $\lambda'\in (\lambda_1, \lambda)$. Note that $(X_t)$ has the same distribution, up to rescaling the time axis, as a contact process $(Y_t)$ with infection rate $\lambda'$ and recovery rate $\frac{\lambda'}{\lambda} =: 1-\ep_1$. Let $(Y^{L_0}_{t})$ be the contact process $(Y_t)$ restricted to $\T_{\le L_0}$, namely, it only uses the infection and recovery events inside $\T_{\le L_0}$. It suffices to show that for some $s_0, L_0$,
	\begin{equation}\label{eq:good:tree:Y}
	\E_{\T}(|Y^{L_0}_{\le s_0}\cap \T_L|)\ge \theta
	\end{equation}
	where $Y^{L_0}_{\le s_0}:=\bigcup_{t\le s_0} Y^{L_0}_{t}$. Indeed, assuming \eqref{eq:good:tree:Y}, we have
	\begin{equation} 
	\E (|X^{L_0}_{\le s_0}\cap \T_L|)\ge \theta\nonumber.
	\end{equation}
	Let $W_k = |X^{k L_0}_{\le ks}\cap \T_{k L_0}|$ then by the same argument that derives \eqref{eq:Wk:ep}, we have
	\begin{equation}\label{eq:Wk:ep:2}
	\P\left (W_k> \ep\theta^{k}\right )>\ep
	\end{equation}
	for all sufficiently large $k$. As a consequence, 
	$$\E |W_k|\ge \ep^{2}\theta^{k}.$$
	Thus, there exists $t\in [0, ks]$ such that 
	$$\E |X^{k L_0}_{[t, t+1]}\cap \T_{k L_0}|\ge \ep^{2}\theta^{k}/ks.$$
	Hence, by choosing $k$ large, we have
	$$\E |X^{k L_0}_{t+1}\cap \T_{k L_0}|\ge e^{-1}\ep^{2}\theta^{k}/ks>1$$
	which proves \eqref{eq:good:tree:1} by letting the new $s_0$ to be $t+1$ and the new $L_0$ to be $k L_0$.
	
	Let $(Y'_t)$ be the contact process on $\T$ with the root initially infected and with infection rate $\lambda'$ and recovery rate 1.
	
	Since $\lambda'>\lambda_1$, the contact process $(Y'_t)$ survives forever with positive probability and in such event, there is a path of infection from the root to infinity with positive probability. Let $\A_{s_0, L_0}$ be the event that there is a path of infection of $(Y'_t)_{t\le s_0}$ from the root to a leaf of the $L_0$-generation that lies entirely in $\T_{\le L_0}$. We have
	$$\P_{\T, Y'}\left (\bigcup _{s_0\in (0, \infty)}\A_{s_0, L_0}\neq \emptyset \quad \forall L_0\right )>\ep$$
	for some constant $\ep>0$. 
	Thus, for all $L_0$, 
	\begin{equation}\label{key}
	\P_{\T, Y'}\left (\bigcup _{s_0\in (0, \infty)}\A_{s_0, L_0}\neq \emptyset\right )>\ep.\nonumber
	\end{equation}
	We fix $L_0$ to be a sufficiently large constant. Since $\A_{s_0, L_0}$ is an increasing event in $s_0$, there exists $s_0\ge 1$ such that
	\begin{equation}\label{key}
	\P_{\T, Y'}\left (\A_{s_0, L_0}\neq \emptyset\right )>\ep.\nonumber
	\end{equation}
	Note that $\A_{s_0, L_0}\subset \A_{s_0, L_0'}$ for any $L_0'\le L_0$ and hence
	\begin{equation}\label{eq:GW:AtL}
	\P_{\T, Y'}\left (\A_{s_0, L_0'}\neq \emptyset\right )>\ep, \quad \forall L_0'\le L_0.
	\end{equation}
	Conditioned on the event $\A_{s_0, L_0}$, let $W_{s_0,L_0}$ be the set of vertices in $\T_{\le L_0}$ that are infected by time $s_0$ by an infection path that lies entirely in $\T_{\le L_0}$ then $W_{s_0, L_0}$ is a subtree of $\T$ that contains $\rho$ and at least one vertex of depth $L_0$. Let $\partial W_{s_0, L_0}$ be the set of vertices of $W_{s_0, L_0}$ that either belong to $\T_{L_0}$ or has at least one child not in $W_{s_0, L_0}$.  
	
	Let $C$ be such that 
	\begin{equation}\label{eq:def:C}
	C\ep^{2} \ep_1  \frac{\lambda}{1+\lambda} = 10
	\end{equation}
	where we recall that $\ep_1 =1-  \frac{\lambda'}{\lambda}$. The reason for this choice of $C$ will be clear in \eqref{eq:def:C:use}.

	Let $\mathcal B_{L_0}$ be the event that there exists a subtree $W$ of $\T_{\le L_0}$ containing the root, intersecting $\T_{L_0}$ and satisfying $|\partial W|\le C$. Note that this event $\mathcal B_{L_0}$ depends only on the randomness of the tree, but not on the contact process. We shall show later that for any fixed constant $C$,
	\begin{equation}\label{eq:A:L_0}
	\lim_{L_0\to \infty} \P\left (\mathcal B_{L_0}\right ) = 0.
	\end{equation}
	
	Assuming \eqref{eq:A:L_0}, we let $L_0$ be sufficiently large so that $\P(\mathcal B_{L_0})\le \ep/2$.

	We now consider the event $\A_{s_0, L_0}\setminus \mathcal B_{L_0}$ (which happens with probability at least $\ep/2$). The process $(Y_t)$ can be obtained from $(Y'_t)$ by censoring each recovery clock with probability $1-\ep_1$; in other words, with probability $1-\ep_1$, a recovery clock of $(Y'_t)$ is kept and with probability $\ep_1$, it is ignored. For each vertex $v\in \partial W _{s_0, L_0} \cap \T_{<L_0}$ with a child $u\notin W_{s_0, L_0}$, if $v\in Y'_{s_0}$ then with probability at least $(1-e^{-s_0})\frac{\lambda}{1+\lambda}$, $u\in Y_{\le 2 s_0}$. If $v\notin Y'_{s_0}$, then there is at least one recovery clock at $v$ in time $[0, s_0]$. By censoring these recovery clocks, with probability at least $\ep_1 (1-e^{-s_0})\frac{\lambda}{1+\lambda}$, we have $u\in Y_{\le 2 s_0}$. Consider the subtree $T_u \subset \T$ rooted at $u$. We observe that the trees $T_u$ are disjoint and are independent of $W_{s_0, L_0}$ and $(Y'_t)_{t\le s_0}$. By \eqref{eq:GW:AtL}, if $u$ is infected at time $t$, with probability at least $\ep$, there is a path of infection from $u$ to a leaf of the $L_0$-generation of $\T$ by time $s_0+t$. In particular, the expected number of infected descendants of $u$ that belong to the $L_0$-generation of $\T$ by time $s_0+t$ is at least $\ep$.  
	
	Hence, by the choice of $C$ in \eqref{eq:def:C},
	\begin{eqnarray} 
	\E  (|Y^{L_0}_{\le 3 s_0}\cap \T_{L_0}|) &\ge&  \P\left (\A_{s_0, L_0}\setminus \mathcal B_{L_0}\right ) \E\left (|\partial W_{s_0, L_0}|\bigg| \A_{s_0, L_0}\setminus \mathcal B_{L_0}\right ) \ep_1    (1-e^{-s_0})\frac{\lambda}{1+\lambda} \ep \nonumber\\
	&\ge&  \frac{\ep}{2}C \ep_1 \frac{1}{2} \frac{\lambda}{1+\lambda} \ep >1\label{eq:def:C:use} 
	\end{eqnarray}
	proving \eqref{eq:good:tree:Y} by setting $3 s_0$ to be the new $s_0$.

	It remains to prove \eqref{eq:A:L_0}. Assume that such a $W$ exists. Let $P$ be a path of length $L_0$ in $W$ connecting $\rho$ and a vertex in the $L_0$-generation of $\T$. Let $v$ be a vertex in $P$ followed by a child $u\in P$ and consider the subtree $\T_{P}(v)$ consisting of $v$ and its descendants who are not descendants of $u$. If this subtree intersects $\T_{L_0}$ then it must contains a vertex in $\partial W$. Observe also that these subtrees $\T_{P}(v)$ are disjoint for all $v\in P$. Therefore, there are at most $C$ vertices $v$ in $P$ whose $\T_{P}(v)$ intersects $\T_{L_0}$. Assume that the depth of these vertices is $p_1< \dots < p_k$ ($k\le C$).
	
	There are at most $C(L_0+1)^{C}$ ways to choose the sequence $p_1<\dots < p_k$. For each such choice, it suffices to show that the probability that there exists a $W$ corresponding to this sequence is exponentially small in $L_0$, namely there exists a positive constant $c$ such that
	\begin{equation}\label{eq:A:L_0:k:p}
	\P\left (\mathcal B_{L_0, p_1, \dots, p_k}\right ) \le  e^{-c L_0}
	\end{equation}
	where $\mathcal B_{L_0, p_1, \dots, p_k}$ is the event that there exists a path $P = (\rho, v_1, \dots, v_{L_0})$ of length $L_0$ in $\T_{\le L_0}$ connecting the root and a vertex of depth $L_0$ such that for all vertices $v\in P$ except those at depths $p_1, \dots, p_k$, $\T_{P}(v)$ does not intersect $\T_{L_0}$. Since 
	$$\mathcal B_{L_0} \subset \bigcup_{p_1, \dots, p_k} \mathcal B_{L_0, p_1, \dots, p_k},$$
	by the union bound, we get \eqref{eq:A:L_0}.

	It is left to prove \eqref{eq:A:L_0:k:p}. We say that a vertex $v$ is \textit{thin} if it has exactly one child with at least one descendant in the $L_0$-generation of the tree. Let $\ep_2>0$ be the probability that the root of $\GW(\widetilde{\mu})$ has at least 2 children with infinitely many descendants then for a given vertex $v$, the probability that $v$ is thin  is at most $1 - \ep_2$ (with or without conditioning on the event that $v$ has a descendant in the $L_0$-generation). We have
	\begin{equation} \label{eq:A:L_0:k:p:1}
	\P\left (\mathcal B_{L_0, p_1, \dots, p_k}\right ) \le  (1-\ep_2)^{L_0-k} d^{k}\ep_2^{-k}\le (1-\ep_2)^{L_0} \left ((1-\ep_2)^{-1} d\ep_2^{-1}\right )^{C}
	\end{equation}
	which gives \eqref{eq:A:L_0:k:p} for sufficiently large $L_0$.
	
	To see \eqref{eq:A:L_0:k:p:1}, take for an example that the sequence $p_1, \dots, p_k$ is $2, 3$. The probability that $\rho$ is thin is at most $1 - \ep_2$. Conditioned on this event, there is only one choice for $v_1$, which is the only child of $\rho$ with a descendant in the $L_0$-generation. The probability that $v_1$ is thin is again at most $1 - \ep_2$ and conditioned on this event, there is only one choice for $v_2$. Now, $v_2$ is not thin and there can be many choices for $v_3$ among the children of $v_2$. Likewise, there can be many choices for $v_4$ among the children of $v_3$. Given $v_2$, the expected number of choices for $v_4$ is the same as 
	\begin{equation}\label{key}
	\E_{\T\sim \GW(\widetilde{\mu})} \left (|\T_{2}|\big| \T_{L_0-2}\neq \emptyset\right )\le \frac{d^{2}}{\P_{\T\sim \GW(\widetilde{\mu})}\left (\T_{L_0-2}\neq \emptyset\right )}\le d^{2}\ep_2^{-1}.
	\end{equation}
	Repeating this argument for the chance that $v_4, \dots, v_{L_0-1}$ are thin, we get 
	\begin{equation} \label{eq:A:L_0:k:p:2}
	\P\left (\mathcal B_{L_0, 2, 3}\right ) \le  (1-\ep_2)^{L_0-2} d^{2} \ep_2^{-1}\le (1-\ep_2)^{L_0-2} d^{2} \ep_2^{-2}.\nonumber
	\end{equation}
	The proof of \eqref{eq:A:L_0:k:p:1} for a general sequence $p_1, \dots, p_k$ follows by the same reasoning. We leave the details to the interested reader.
\end{proof}

\subsection{Good vertices} We are now ready to define the notion of ``good" vertices as mentioned earlier. Let $s_0, L_0, \theta$ be the constants in Lemma \ref{lm:good:tree}. 


Let $c_0$ be a small constant such that for $\T\sim \GW(\widetilde{\mu}_{\eta_0, j_0})$ and for all $h\ge c_0^{-1}$, 
\begin{equation}\label{def:c0}
\P_{\subgw}\left (|\T_{h}|\ge c_0 \tilde d^{h}\right )\ge 2c_0.
\end{equation}
The existence of $c_0$ follows from the fact that $\frac{|\T_{h}|}{d^{h}}$ converges to a nonzero random variable almost surely and that $\widetilde{\mu}_{\eta_0, j_0}$ converges $\widetilde{\mu}$ and $\tilde d$ converges to $d$.

For a set of $\kappa$ ``good" vertices, we shall show that there are a lot more vertices infected at a later time (Lemma \ref{lm:supercritical:iteration}). Generally, a good vertex is a vertex with an expanding neighborhood, namely a distance-$h$ neighborhood contains at least $c_0 \tilde d^{h}$ vertices (similar to $\T$ as in \eqref{def:c0}).

For the proof of Theorem \ref{thm:supercritical:general}, we only need to take $\kappa=\delta n$ while for Theorem \ref{thm:super:onevertex}, we shall need the range of $\kappa$ to be in $[(\log n)^{C}, \delta n]$ where $C$ is a large constant and $\delta$ is a small constant to be chosen. Thus, for the reader who is only interested in Theorem \ref{thm:supercritical:general}, it suffices to assume that $\kappa = \delta n$ for the rest of the proof. Here, we write for any $\kappa\in [(\log n)^{C}, \delta n]$. 

Let 
\begin{equation}\label{def:l1:kappa}
l_1(\kappa) = 2\log_{\tilde d} \log \frac{n}{\kappa}
\end{equation}
be the depth (up to an additive constant) of the neighborhood for the definition of good vertices. Roughly speaking, the reason for this choice is to have
\begin{equation}\label{eq:d:l}
\tilde d^{l_1(\kappa)}\gg \log \frac{n}{\kappa} \quad\text{and}\quad j_1^{C l_1(\kappa)}\kappa \le \eta_0 n  
\end{equation} 
where the second inequality ensures the hypothesis of Lemma \ref{lm:dominate:GW} and the first inequality, whose purpose will be clear later in Lemma \ref{lm:graphproperty}, is to facilitate the union bound over choices of $\kappa$ vertices.
\begin{definition}[\textit{Good} vertices] \label{def:good:vertex}
	A vertex $v\in G_n'$ is said to be \textit{$\kappa$-good} if 
	\begin{itemize}
		\item $N(v, 3)$ has $j_0 (j_0-1)^{3}$ out-going half-edges,
		\item Each out-going half-edge of $N(v, 3)$ expands well to depth $l_1(\kappa)$, namely, it is matched to a half-edge of a vertex $u$ and the number of out-going half-edges of $N(u, l_1(\kappa))\setminus N(v, 3+l_1(\kappa))$ is at least $c_0 \tilde d^{l_1(\kappa)}$.
	\end{itemize}
\end{definition}
By \eqref{def:c0}, we have for a uniformly chosen vertex $v$,
\begin{equation}\label{eq:good:p0}
\P\left (v \text{ is good}\right )\ge  (\mu_{\eta_0, j_0}(j_0))^{j_0^4} c_0^{j_0(j_0-1)^{3}}=:p_0
\end{equation}
where the term $ (\mu_{\eta_0, j_0}(j_0))^{j_0^4}$ bounds the probability that the first condition in Definition \ref{def:good:vertex} happens with $j_0^{4}$ bounding the total number of vertices in $N(v, 3)$ and the term $c_0^{j_0(j_0-1)^{3}}$ bounds the probability of the second condition for which we use Lemma \ref{lm:dominate:GW} and \eqref{def:c0}.

\subsection{A graph property}
Let 
\begin{equation}\label{def:l2:kappa}
l_2(\kappa) = A l_1(\kappa)
\end{equation}
where $A$ is defined by the equation
\begin{equation}\label{key}
(\lambda \tilde d/2) \theta^{A/2}=2>1.\nonumber
\end{equation}
Intuitively, if there are $\kappa$ infected vertices that are $\kappa$-good, at distance $4+l_1(\kappa)$ from these vertices, there are about $\kappa \tilde d^{l_1(\kappa)}$ vertices by the Definition \ref{def:good:vertex}. The number of these vertices that are infected at some point in the future is at least $\kappa \tilde d^{l_1(\kappa)} (\lambda/2)^{l_1(\kappa)}$. A decent portion of these vertices expands good trees of depth $l_2(\kappa)L_0$ by Lemma \ref{lm:good:tree} and so the number of infected vertices at depth $4+l_1(\kappa)+ l_2(\kappa) L_0$ is about  $\kappa \tilde d^{l_1(\kappa)} (\lambda/2)^{l_1(\kappa)} \theta^{l_2(\kappa)} \ge\kappa 2^{l_1(\kappa)}\gg \kappa$. The rest of this section is to make this heuristic rigorous.

Note that $A$ is independent of $C$ and $\delta$.
Let 
\begin{equation}\label{def:A:l1:l2}
l(\kappa) =4+l_1(\kappa)+ l_2(\kappa) L_0.
\end{equation}

Following the above intuition, we consider a sequence of events on the random graphs $G'_n$.
\begin{itemize}
	\item $\mathcal B$ is the event that for every $\kappa\in [(\log n)^{C}, \delta n]$, for every set $D$ of $\kappa$ vertices of $G_n'$ that are $\kappa$-\textit{good} and whose 3-neighborhoods $N(\cdot, 3)$ are all disjoint, it holds that
	\begin{enumerate}
		\item there are at least $c_0 \tilde d^{l_1(\kappa)}\kappa$ out-going half-edges from $\partial N(D, 4+l_1(\kappa))$,
		\item at least $c_0 \ep \tilde d^{l_1(\kappa)}\kappa/4$ of these half-edges expands to an $(s_0, L_0, l_2(\kappa), \theta, \ep)$-good tree of depth $l_2(\kappa)L_0$. Moreover, these good trees are disjoint. 
	\end{enumerate}

	\item Let $C_1$ be a large constant that does not depend on $\delta$ and $C$ (in other words, we shall choose $C\gg C_1$ and $\delta^{-1}\gg C_1$). Let $\mathcal C$ be the event that the following holds for every $\kappa\in [(\log n)^{C}, \delta n]$, every set $D$ of at most $\kappa$ vertices in $G'_n$, every $a \ge p_0^{-1}\kappa \log \frac{n}{\kappa}+ (\log n)^{C_1}$ and $L_1\le C\log\log \frac{n}{\kappa}$. Let $(X_t)$ be the contact process on $G_n'$ with $D$ initially infected and let $s$ be any positive number. Let $D_{s}$ be the set of vertices in $ \partial N(D, L_1)$ that are infected at time $X_s$ using only the infection and recovery events within $N(D, L_1)$ and let $\mathcal F_{\kappa}$ be the set of $\kappa$-good vertices in $G'_n$.  Then  
	\begin{equation}\label{eq:event:C}
	\P_{\subcp}\left (|D_s\cap \mathcal F_{\kappa}|\le a p_0/2\bigg| |D_s|\ge a \right )\le  \exp\left (-3\kappa\log\frac{n}{\kappa}\right ).
	\end{equation}
\end{itemize}
Under the randomness of the perfect matching of half-edges in $G_n'$, we show in Lemmas \ref{lm:graphproperty} and \ref{lm:graphproperty:2} that $\mathcal B$ and $\mathcal C$ happen \whp, respectively.
\begin{lemma} \label{lm:graphproperty}
	We have
	\begin{equation}\label{key}
	\P(G_n'\in \mathcal B) = 1-o(1).
	\end{equation}
\end{lemma}
For the proof of this lemma, we shall use the cut-off line algorithm introduced in \cite{k06} to find a uniform perfect matching of the half-edges in $G'_n$.	
\begin{definition}[Cut-off line algorithm]\label{def:cut-off line}
	A perfect matching of the half-edges of $G'_n$ is obtained through the following algorithm.
	\begin{itemize}
		\item Each half-edge of a vertex $v$ is assigned a height uniformly chosen in $[0, 1]$ and is placed on the line of vertex $v$.
		\item The cut-off line is initially set at height 1.
		\item Pick an unmatched half-edge independent of the heights of all unmatched half-edges and match it to the highest unmatched half-edge. Move the cut-off line to the height of the latter half-edge.
	\end{itemize}
\end{definition}
Figure \ref{fig2} illustrates the algorithm.
\begin{figure}[H]
	\centering
	\begin{tikzpicture}[thick,scale=1, every node/.style={transform shape}]
	
\node at (2,-0.4) {$v_3$};
		\node at (4,-0.4) {$\dots$};
	\draw[black] (0,0)--(8,0);

	\node at (1,-0.4) {$v_2$};

	\node at (8,-0.4) {$v_n$};
		\node at (6,-0.4) {$\dots$};
	\draw[red] (-1,3.2)--(9,3.2);
\draw (2, 0.5) node[black]{\textsf{o}};
\foreach \x in {0,1, 2, 4, 6, 8}{
	\draw[black] (\x,0) -- (\x,4.8);
}
	\draw[red, dashed] (-1,2.8)--(9,2.8);
	\node at (10.4,2.8) {new cut-off line}; 
	\draw (0, 3.2) node[black]{\textsf{o}};
	
	\draw (4, 4) node[black]{\textsf{o}};

		\draw (2,1.6) node[cross=2.2pt,black]{};
	\draw (0, 0.5) node[cross=2.2pt,black]{}; 
	\draw (0,2.2) node[cross=2.2pt,black]{};

	\draw (4,1.36) node[cross=2.2pt,black]{};
	\draw (4,2.8) node[cross=2.2pt,red]{};
		\node at (10,3.2) {cut-off line}; 
	\draw (6,0.3) node[cross=2.2pt,black]{};
		\draw (1,2.4) node[cross=2.2pt,black]{};
		
	\draw (8, 3.7) node[black]{\textsf{o}};
	
	\draw (6,2.5) node[cross=2.2pt,black]{};
	
		\node at (0,-0.4) {$v_1$};

	\draw[->, blue] (5.8,1.6) -- (4.2,2.7);
	
	\draw (8, 3) node[black]{\textsf{o}};
		\draw (8,.88) node[cross=2.2pt,black]{};
	\draw (6,1.5) node[cross=2.2pt,blue]{};
	\draw (6, 3.5) node[black]{\textsf{o}};
	\end{tikzpicture}
	\caption{The circles `o' represent matched half-edges and the crosses `$\times$' represent unmatched half-edges. The blue half-edge is chosen and matched to the red half-edge which is the highest unmatched half-edge. Then the cut-off line is moved to the new cut-off line (the dashed line).} \label{fig2}
\end{figure}
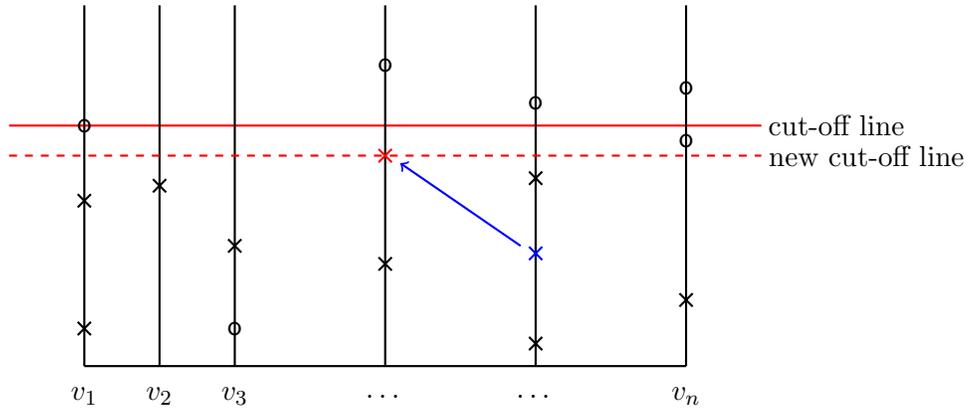
\begin{proof}[Proof of Lemma \ref{lm:graphproperty}] 
	Let $D$ be a set of $\kappa$ vertices.
	
	Perform the cut-off line algorithm to explore $N(D, 4+l_1(\kappa))$: in other words, we run the cut-off line algorithm to match the half-edges of $D$ and then $\partial N(D, 1), \dots, \partial N(D, 3+l_1(\kappa))$. Starting with the cut-off line at 1, let $1-a$ be the position of the cut-off line after this exploration.

	Let 
	$$\gamma = \frac{\kappa}{n}\le \delta.$$
	
	There are at most $2j_1^{4+l_1(\kappa)}\kappa$ half-edges of $G_n'$ that lie above this cut-off line while \textsf{whp}, there are more than $n$ half-edges in $G_n'$. Hence, by Chernoff inequality,
	\begin{equation}\label{key}
	\P(a\ge \gamma^{0.99}) \le \P\left (\Bin (n, \gamma^{0.99})\le 2j_1^{4+l_1(\kappa)}\kappa \right ) \le \exp\left (- \gamma^{0.99} n/12 \right )\le \exp(-1.01   \kappa \log \gamma) .\nonumber
	\end{equation}
	For each $j\ge 2$, let $n_j$ be the number of vertices in $G'_n$ with $j$ half-edges above the cutoff line. 
	Under the event that $a\le \gamma^{0.99}$, if $\sum_{ j\ge i} n_j \ge 1.09  \kappa/(i-1)$ for some $i$ then there are $\ge 1.09   \kappa/(i-1)$ vertices in $G_n'$ with $\ge i$ half-edges above the line $1- \gamma^{0.99}$. This happens with probability at most
	$$\P(\Bin(n, j_1^{i}\gamma^{0.99 i})\ge 1.09   \kappa/(i-1))$$
	which is at most 
	$$\exp\left (-0.99\times 1.09 \frac{  \kappa}{i-1} \log \frac{1.09\gamma ^{-0.99 i+1}}{j_1^{i} (i-1)}\right ) \le \exp\left (-0.99\times 1.09 \frac{  \kappa}{i-1} \log \frac{1.09\gamma ^{-0.99 i+1}}{j_1^{j_1+1}}\right ) $$
	by the following Chernoff inequality for binomial distribution
	\begin{equation}\label{chernoff:bin}
	\P\left (\Bin(N, p)\ge x Np\right )\le \exp\left (-xNp\log x+xNp\right ),\quad\forall N, p>0, x>1.
	\end{equation}
	
	Since $\gamma\le \delta$ is sufficiently small compared to $j_1$, the above exponential is at most $\exp\left (-1.01   \kappa\log \gamma\right )$ for $i\ge 2$.
	
	Taking union bound over $i$ and $D$ and noting that ${n\choose {  \kappa}} = o\left (\exp\left (1.01   \kappa\log \gamma\right )\right )$, we conclude that \whp, for all $D$ and all $2\le i\le j_1$, we have $\sum_{j\ge i} n_j \le 1.09  \kappa/(i-1)$. Thus
	\begin{equation}\label{eq:sum:ni:i}
	\sum_{i=2}^{j_1} in_i\le 1.09   \kappa\left (\frac{j_1}{j_1-1}+\sum_{i=2}^{j_1-1}\left (\frac{1}{i-1}-\frac{1}{i}\right )i\right )\le 1.1   \kappa(2+ \log j_1).
	\end{equation}
	
	Now, assume that $D$ consists of $\kappa$-good vertices with disjoint 3-neighborhoods. By definition, the number of out-doing half-edges of $N(D, 3)$ is $j_0(j_0-1)^{3}\kappa$. By \eqref{eq:sum:ni:i}, all but at most $1.1   \kappa(2+ \log j_1)$ of these half-edges expand disjoint $l_1(\kappa)$-neighborhoods. By the definition of $\kappa$-good vertices, each of these neighborhoods has at least $c_0 \tilde d^{l_1(\kappa)}$ out-going half-edges. Therefore, the number of out-going half-edges from $\partial N(D, 4+l_1)$ is at least 
	\begin{equation}\label{key}
	\left (j_0(j_0-1)^{3}\kappa-  1.1   \kappa(2+ \log j_1)\right ) c_0 \tilde d^{l_1(\kappa)}\ge c_0 \tilde d^{l_1(\kappa)}\kappa.
	\end{equation}
	This establishes the first part in the definition of $\mathcal B$. For the second part, we enumerate these half-edges by $e_1, e_2, \dots, e_{m}$ and explore their $l_2(\kappa)L_0$ neighborhoods one by one. Note that since the maximal degree is $j_1$, at any point during the exploration, the number of explored vertices is at most 
	$\kappa j_1^{5+l_1(\kappa)+l_2(\kappa)L_0} = \kappa \left (\log \frac{n}{\kappa}\right )^{A_1}$ where $A_1$ is a constant independent of $\delta$ and $C$. This number is at most $\eta_0 n$ for sufficiently small $\delta$. Thus, by Lemma \ref{lm:dominate:GW} and Lemma \ref{lm:good:tree}, for each $k\in [1, m]$, the probability that $e_k$ expands to an $(s_0, L_0, l_2(\kappa), \theta, \ep)$-good tree that does not intersect the previously explored neighborhoods (i.e., $N(D, 4+l_1(\kappa))\cup \bigcup _{i=1}^{k-1} N(e_i, l_2(\kappa) L_0)$) is at least $\ep/2$, conditioned on these previously explored neighborhoods. Hence, the number of half-edges that expand to disjoint $(s_0, L_0, l_2(\kappa), \theta, \ep)$-good trees is at least $\Bin(c_0 \tilde d^{l_1(\kappa)}\kappa, \ep/2)$. By Chernoff inequality, 
	\begin{equation}\label{key}
	\P\left (\Bin(c_0 \tilde d^{l_1(\kappa)}\kappa, \ep/2)\le c_0 \tilde d^{l_1(\kappa)}\kappa \ep/4\right ) \le \exp\left (-\Omega\left (\tilde d^{l_1(\kappa)} \kappa\right )\right ) \le \exp\left (-1.1 \kappa\log\frac{n}{\kappa}\right )
	\end{equation}
	by \eqref{eq:d:l}. By taking the union bound over the choices of $D$, this completes the proof. 	 
\end{proof}
\begin{lemma}\label{lm:graphproperty:2}
	We have
	\begin{equation}\label{key}
	\P \left (G'_n\in \mathcal C\right ) = 1-o(1). \nonumber 
	\end{equation}
\end{lemma}
\begin{proof}[Proof of Lemma \ref{lm:graphproperty:2}]
	The event $|D_s|\ge a$ depends only on the $L_1$-neighborhood of $D$. Thus, conditioning on this event, we expose the complement of this neighborhood to find the number of $\kappa$-good vertices in $D_s$. Since the probability for a vertex to be $\kappa$-good is at least $p_0$, 
	\begin{equation}\label{key}
	\E _{\subrg, \subcp}\left (|D_s\cap \mathcal F_\kappa|\bigg | |D_s|\ge a\right ) \ge a p_0.\nonumber
	\end{equation}
	By Azuma's inequality and noting that $|D_s\cap \mathcal F_\kappa| $ depends only on $N(D_s, 4+l_1(\kappa))$ whose size at most $j_1^{5+l_1(\kappa)}$,
	\begin{equation}\label{key}
	\P _{\subrg, \subcp}\left (|D_s\cap \mathcal F_\kappa| \le ap_0/2\bigg | |D_s|\ge a\right ) \le \exp\left (-\frac{(ap_0/2)^{2}}{2 a j_1^{10+ 2  l_1(\kappa)}}\right )    \le \exp\left ( -5 \kappa \log \frac{n}{\kappa} \right ).
	\end{equation}
	where we chose $C_1$ much larger than $A$ in \eqref{def:l2:kappa}.
	
	Writing $\P_{\subrg, \subcp}( \cdot)$ as $\E_{\subrg}\left (\P_{\subcp} (\cdot)\bigg| G'_n\right )$ and applying Markov's inequality, we yield from the above inequality that
	\begin{eqnarray}\label{key}
	&&\P_{\subrg} \left (\P_{(X_t)} \left (|D_s\cap \mathcal F_\kappa| \le ap_0/2 \right ) \ge  \exp\left ( -3 \kappa \log \frac{n}{\kappa} \right )\bigg | |D_s|\ge a\right )\nonumber\\
	&& \qquad  \le \exp\left ( -2 \kappa \log \frac{n}{\kappa} \right ) = o\left ({n\choose \kappa}^{-1}\right ).\nonumber
	\end{eqnarray}
	Taking union bound over at most $\kappa {n\choose \kappa}$ choices of $D$, at most $n$ choices of $a$ and over $\kappa$ and $L_1$, we obtain
	$$\P_{\subrg}\left (\mathcal C^{c}\right )\le  \sum_{\kappa=(\log n)^{C}}^{\delta n} \exp\left ( -\kappa \log \frac{n}{\kappa} \right )=o(1).$$
\end{proof}

 Let $\mathcal E$ be the event that the following holds for all $\kappa\in [\log^{C}n, \delta n]$. Let $\mathcal F_\kappa$ be the set of $\kappa$-good vertices in $G'_n$. For any set $D$ of $\kappa$ vertices in $G'_n$ that are $\kappa$-good with disjoint $3$-neighborhoods, there exists an $s\ge \ep$ such that
\begin{equation}\label{eq:pf:thm:supercritical:pr}
\P_{\subcp}\left (\left |D_s\cap \mathcal F_{2\kappa}\right |\le \kappa\log\frac{n}{\kappa} \right )\le e^{-\kappa}
\end{equation}
where $(X_t)$ is the contact process on $G_n'$ with $D$ initially infected and $D_{s}$ is the set of vertices in $ \partial N(D, l(\kappa))$ that are infected at time $X_s$ using only the infection and recovery events within $N(D, l(\kappa))$.
\begin{lemma}\label{lm:supercritical:iteration}
	We have for a sufficiently large constant $C$,
	$$\P\left (G'_n\in \mathcal E\right ) = 1-o(1).$$
\end{lemma}
\begin{proof} Under the event $\mathcal B$ in Lemma \ref{lm:graphproperty}, there exists a constant $t>0$ such that  
	\begin{equation} \label{eq:pf:thm:supercritical:ex}
	\E_{\subcp}\left (|X_{\le t l_2(\kappa)}\cap \partial N(D, l(\kappa))|  \right ) \ge \frac{\ep}{4} \kappa c_0 \tilde d^{l_1(\kappa)} \left (\frac{(1-e^{-1})\lambda}{1+\lambda}\right )^{4+l_1(\kappa)}\theta^{l_2(\kappa)} \ge \theta^{l_2(\kappa)/2}  \kappa
	\end{equation}
	where the last inequality follows from \eqref{def:A:l1:l2}.
	Dividing the time interval $[0, t l_2(\kappa)]$ into intervals of length $1$, there exists an interval $[r, r+1]$ such that 
	\begin{equation}\label{}
	\E_{\subcp}\left (|X_{[r, r+1]} \cap\partial N(D, l(\kappa))| \right ) \ge   \theta^{l_2(\kappa)/2} \kappa/(t l_2(\kappa))\nonumber
	\end{equation}
	where $X_{[r, r+1]} := \bigcup_{x\in [r, r+1]} X_{x}$. 
	Since 
	$$\E \textbf{1}_{v\in X_{r, r+1}}\le e \E \textbf{1}_{v\in X_{r+1}},$$
	we obtain
	\begin{equation}\label{}
	\E_{\subcp}\left (|X_{r+1} \cap\partial N(D, l(\kappa))|\right ) \ge e \theta^{l_2(\kappa)/2} \kappa/(t l_2(\kappa))\ge 4p_0^{-1}\kappa \log \frac{n}{\kappa}.\nonumber
	\end{equation}
	By Azuma's inequality, 
	\begin{equation}\label{eq:x:nd}
	\P_{\subcp}\left (|X_{r+1} \cap\partial N(D, l(\kappa))|\le 2 p_0^{-1}\kappa \log \frac{n}{\kappa} \right )\le \exp\left ( - \frac{\kappa^{2} \log^{2} \frac{n}{\kappa}}{\kappa j_1^{3l(\kappa)}}\right )\le  \exp\left (-3\kappa\log\frac{n}{\kappa}\right ).
	\end{equation}
	In the notations of Lemma \ref{lm:graphproperty:2}, $X_{r+1} \cap\partial N(D, l(\kappa))$ is $D_{r}$. Thus, \eqref{eq:x:nd} together with \eqref{eq:event:C} give 
	\begin{equation} 
	\P_{\subcp}\left (|D_s\cap \mathcal F_\kappa|\le \kappa \log \frac{n}{\kappa}  \right )\le  2\exp\left (-3\kappa\log\frac{n}{\kappa}\right )\le e^{-\kappa}\nonumber
	\end{equation}
	completing the proof.
\end{proof}
\subsection{Finishing the proof of Theorems \ref{thm:supercritical:general} and \ref{thm:super:onevertex}} We are now ready to finish the proof of the first inequality in \eqref{eq:thm:super:ER} and Theorem \ref{thm:super:onevertex}.
Let $G_n'$ be a graph belonging to both events $\mathcal B$ and $\mathcal C$ in Lemmas \ref{lm:graphproperty} and \ref{lm:graphproperty:2}.

Let $v$ be a uniformly chosen random vertex in $G'_n$. By Lemma \ref{lm:dominate:GW}, \textsf{whp} over $G'_n$, for any constant $C_2$, $N(v, C_2\log \log n)\ge_{st} \GW(\mu_{\eta_0, j_0}, \widetilde{\mu}_{\eta_0, j_0})$. And so, by Lemma \ref{lm:good:tree}, there exist constants $s_0, L_0, \ep>0$ and $\theta>1$ such that
\begin{equation} 
\P(N(v, C_2\log \log n) \text{ contains an $(s_0, L_0, C_2\log \log n/L_0, \theta, \ep)$-good tree})>\ep\nonumber
\end{equation}
where the probability is over the randomness of $v$ and the randomness of the perfect matching of half-edges in $G'_n$.
This implies that the set, denoted by $\mathcal E_1$, of vertices $v$ such that $N(v, C_2\log \log n)$ contains an $(s_0, L_0, C_2\log \log n/L_0, \theta, \ep)$-good tree has expectation at least $\ep n'$ where $n' = |G'_n|\ge n/2$. By Azuma's inequality, 
$$\P \left (|\mathcal E_1|\le \ep n/4\right )\le \exp\left (-\frac{\ep^{2}n}{32 j_1^{2C_2\log \log n}}\right ) = o(1).$$

Thus, with high probability over $G'_n$, $|\mathcal E_1|\ge \ep n/4$ which means that with positive probability over the choice of $v$, $N(v, C_2\log \log n)$ contains an $(s_0, L_0, C_2\log \log n/L_0, \theta, \ep)$-good tree. 

Let $(X_t)$ be the contact process on $G'_n$ starting with $v$ infected. By the definition \ref{def:good:tree} of good trees, with probability at least $\ep$ over the contact process, there exists a time at which the number of infected vertices is at least 
$$\theta^{C_2\log \log n/L_0}=( \log n)^{C_2\log \theta/L_0}.$$

Let 
$$\kappa:=( \log n)^{C_2\log \theta/(2L_0)}p_0/(2j_1^{4}).$$
As we are under the event $\mathcal C$ in Lemma \ref{lm:graphproperty:2}, among these infected vertices, there are at least $( \log n)^{C_2\log \theta/L_0}p_0/2 \ge \kappa j_1^{4}$ vertices that are $\kappa$-good.

Once there are at least $\kappa j_1^{4}$ infected vertices that are $\kappa$-good, we observe that a 3-neighborhood contains at most $j_1^{4}$ vertices and so at least $\kappa$ of these vertices have disjoint $3$-neighborhoods. Applying Lemma \ref{lm:supercritical:iteration}, we find that with probability at least $e^{-\kappa}$, at a later time, there are at least $2\kappa j_1^{4}$ infected vertices that are $2\kappa$-good. Repeating this argument for $2\kappa, 4\kappa, \dots$, we see that with probability at least $\ep/2$, there exists a time at which there are $\ge \delta n$ infected vertices that are $\delta n$-good. Now, applying Lemma \ref{lm:supercritical:iteration} for $e^{\delta n/2}$ times with $\kappa = \delta n$ gives that with probability at least $\ep/4$, the contact process survives up to time $\ep e^{\delta n/2}$, concluding the proof of Theorem \ref{thm:super:onevertex}. If all vertices are initially infected, then the first inequality in \eqref{eq:thm:super:ER} of Theorem \ref{thm:supercritical:general} follows directed by applying Lemma \ref{lm:supercritical:iteration} for $e^{\delta n/2}$ times with $\kappa = \delta n$ as above.

\subsection{Finishing the proof of Theorems \ref{thm:phasetransition:tree} and \ref{thm:phasetransition:graph}}\label{subsec:finale}

Finally, we combine the results obtained in this article and establish Theorems \ref{thm:phasetransition:tree} and \ref{thm:phasetransition:graph}. Let us begin with Theorem \ref{thm:phasetransition:tree}. Its lower bound 
$$\lim_{k\to\infty} \lambda_1^{\textsc{gw}}(\xi_k)d_k \geq 1 $$
was proven at the end of Section \ref{subsec:tail estimate survival time}. Also, the other inequality was obtained by Theorem \ref{thm:supercritical:general}. For Theorem \ref{thm:phasetransition:graph}, Theorem \ref{thm:lowerbd of lamdac} implies the lower bound
$$ \lim_{k\to\infty}\lambda_c^- (\mu_k) d_k \geq 1. $$
Then the upper bound is again given by Theorem \ref{thm:supercritical:general}. Thus, we conclude the proofs of Theorems \ref{thm:phasetransition:tree} and \ref{thm:phasetransition:graph}. \qed

 \section*{Acknowledgements} We are grateful to Rick Durrett and Xiangying Huang for fruitful discussions on Theorem \ref{thm:supercritical:general}.
 
 \bibliographystyle{plain}
 \bibliography{reference}
 
 \vspace{15mm}

 \appendix

 \section{Recursive analysis for unicyclic graphs}\label{sec:unicyclic}
 
Appendix \ref{sec:unicyclic} is devoted to the proof of Proposition \ref{prop:SMunicyclic} and Corollary \ref{cor:SMgwc2}. The main technical work needed here is to  carry out the tree recursions as in Sections \ref{sec:subcritical trees} and \ref{sec:numberofhits} despite the presence of a cycle. As before, our approach is based on an induction argument, which differs depending on whether $t$ in (\ref{eq:unicyc tail:in prop}) is small or large.
 
 In Section \ref{subsec:unic prop small t1}, we introduce an appropriate way of covering a unicyclic graph by trees. Based on this notion, in Section \ref{subsec:unic prop small t2} we  can establish (\ref{eq:unicyc tail:in prop}) for \textit{small} values of $t$. Then, we conclude the proof of Proposition \ref{prop:SMunicyclic} and Corollary \ref{cor:SMgwc2} in tandem inductively in Section \ref{subsec:unic prop large t}.

 
 \subsection{Reduction of unicyclic graphs to trees}\label{subsec:unic prop small t1}
 
 Let $H$ be a graph that consists of a length-$m$ cycle $C=v_1v_2\ldots v_m v_1$ and  depth $\leq l$ trees $T_1,\ldots,T_m$ rooted at $v_1,\ldots,v_m$, respectively. The goal of this subsection is to introduce an approach that interprets the contact process on $H$ by decomposing it into processes on trees. To this end, we begin with defining the \textit{cover} of $H$.

 \begin{definition}[Cover of $H$]\label{def:cover of unic}
 	Let $m' = \lceil \frac{m+1}{2} \rceil$ and $H$ be a unicyclic graph defined as above. The \textit{cover of} $H$ is a pair of graphs $A_1$ and $A_2$ defined as follows (an illustration of the graphs can be found in Figure \ref{fig:cover}):
 	\begin{itemize}
 		\item $A_1$ consists of a length-$m$ line $v_{m'} \ldots v_m v_1 v_2 \ldots v_{m'-1} \tilde{v}_{m'}$ and the trees $T_j$  rooted at $v_j$ except at the two endpoints $v_{m'}$, $\tilde{v}_{m'}$. The root of $A_1$ is $v_1$.
 		
 		\item $A_2$ consists of a length-$m$ line $v_{1} \ldots v_m \tilde{v}_1$ and the trees $T_j$ rooted at $v_j$ except at the two endpoints $v_{1}$, $\tilde{v}_1$. The root of $A_2$ is $v_{m'}$.
 	\end{itemize} 
 \end{definition}
 
 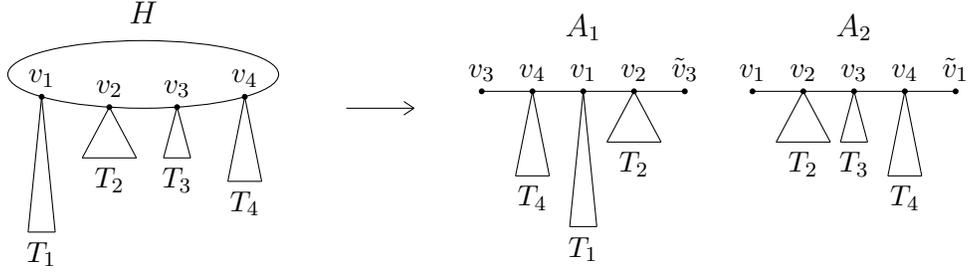
\begin{figure}
 	\centering
 	\begin{tikzpicture}[scale=0.45]
 	\draw \boundellipse{0,0}{4}{1};
 	\filldraw[black] (-3,-2.64575131106/4) circle (2pt) node[black, anchor=south]{$v_1 $};
 	\filldraw[black] (3,-2.64575131106/4) circle (2pt) node[black, anchor=south]{$v_4 $};
 	\filldraw[black] (1,-3.87298334621/4) circle (2pt) node[black, anchor=south]{$v_3 $};
 	\filldraw[black] (-1,-3.87298334621/4) circle (2pt) node[black, anchor=south]{$v_2 $};
 	
 	\draw[black] (-3,-2.64575131106/4)--(-3.4,-2.64575131106/4-4);
 	\draw[black] (-3,-2.64575131106/4)--(-3+.4,-2.64575131106/4-4);
 	\draw[black] (-3.4,-2.64575131106/4-4)--(-3+.4,-2.64575131106/4-4);
 	\filldraw[black] (-3,-2.64575131106/4-4) circle (0pt) node[black, anchor=north]{$T_1 $};
 	
 	\draw[black] (-1,-3.87298334621/4)--(-1.8,-3.87298334621/4-1.5);
 	\draw[black] (-1,-3.87298334621/4)--(-1+.8,-3.87298334621/4-1.5);
 	\draw[black] (-1.8,-3.87298334621/4-1.5)--(-1+.8,-3.87298334621/4-1.5);
 	\filldraw[black] (-1,-3.87298334621/4-1.5) circle (0pt) node[black, anchor=north]{$T_2 $};
 	
 	\draw[black] (1,-3.87298334621/4)--(1.4,-3.87298334621/4-1.5);
 	\draw[black] (1,-3.87298334621/4)--(1-.4,-3.87298334621/4-1.5);
 	\draw[black] (1.4,-3.87298334621/4-1.5)--(1-.4,-3.87298334621/4-1.5);
 	\filldraw[black] (1,-3.87298334621/4-1.5) circle (0pt) node[black, anchor=north]{$T_3 $};
 	
 	\draw[black] (3,-2.64575131106/4)--(3.5,-2.64575131106/4-2.5);
 	\draw[black] (3,-2.64575131106/4)--(2.5,-2.64575131106/4-2.5);
 	\draw[black] (3.5,-2.64575131106/4-2.5)--(2.5,-2.64575131106/4-2.5);
 	\filldraw[black] (3,-2.64575131106/4-2.5) circle (0pt) node[black, anchor=north]{$T_4 $};
 	
 	\filldraw[black] (0,1.2) circle (0pt) node[black, anchor=south]{\large{$H$} };
 	\draw[black] (6,-1) -- (8,-1);
 	\draw[black] (8,-1) -- (7.7,-.8);
 	\draw[black] (8,-1) -- (7.7,-1.2);
 	
 	\draw[black] (10,-.5) -- (16,-0.5);
 	\draw[black] (18,-0.5) -- (24,-0.5);
 	\filldraw[black] (10,-0.5) circle (2pt) node[black,anchor=south]{$v_3$};
 	\filldraw[black] (11.5,-0.5) circle (2pt) node[black,anchor=south]{$v_4$};
 	\filldraw[black] (13,-0.5) circle (2pt) node[black,anchor=south]{$v_1$};
 	\filldraw[black] (14.5,-0.5) circle (2pt) node[black,anchor=south]{$v_2$};
 	\filldraw[black] (16,-0.5) circle (2pt) node[black,anchor=south]{$\tilde{v}_3$};
 	\filldraw[black] (13,0.7) circle (0pt) node[black,anchor=south]{\large{$A_1$}};
 	
 	\filldraw[black] (21,0.7) circle (0pt) node[black,anchor=south]{\large{$A_2$}};
 	\filldraw[black] (18,-0.5) circle (2pt) node[black,anchor=south]{$v_1$};
 	\filldraw[black] (19.5,-0.5) circle (2pt) node[black,anchor=south]{$v_2$};
 	\filldraw[black] (21,-0.5) circle (2pt) node[black,anchor=south]{$v_3$};
 	\filldraw[black] (22.5,-0.5) circle (2pt) node[black,anchor=south]{$v_4$};
 	\filldraw[black] (24,-0.5) circle (2pt) node[black,anchor=south]{$\tilde{v}_1$};
 	
 	\draw[black] (11.5,-0.5)--(12,-0.5-2.5);
 	\draw[black] (11.5,-0.5)--(11,-0.5-2.5);
 	\draw[black] (12,-0.5-2.5)--(11,-0.5-2.5);
 	\filldraw[black] (11.5,-0.5-2.5) circle (0pt) node[black, anchor=north]{$T_4 $};
 	
 	\draw[black] (22.5,-0.5)--(23,-0.5-2.5);
 	\draw[black] (22.5,-0.5)--(22,-0.5-2.5);
 	\draw[black] (23,-0.5-2.5)--(22,-0.5-2.5);
 	\filldraw[black] (22.5,-0.5-2.5) circle (0pt) node[black, anchor=north]{$T_4 $};
 	
 	\draw[black] (13,-0.5)--(12.6,-0.5-4);
 	\draw[black] (13,-0.5)--(13.4,-0.5-4);
 	\draw[black] (13.4,-0.5-4)--(12.6,-0.5-4);
 	\filldraw[black] (13,-0.5-4) circle (0pt) node[black, anchor=north]{$T_1 $};
 	
 	\draw[black] (14.5,-0.5)--(14.5-.8,-0.5-1.5);
 	\draw[black] (14.5,-0.5)--(14.5+.8,-0.5-1.5);
 	\draw[black] (14.5-.8,-0.5-1.5)--(14.5+.8,-0.5-1.5);
 	\filldraw[black] (14.5,-0.5-1.5) circle (0pt) node[black, anchor=north]{$T_2 $};
 	
 	\draw[black] (19.5,-0.5)--(19.5-.8,-0.5-1.5);
 	\draw[black] (19.5,-0.5)--(19.5+.8,-0.5-1.5);
 	\draw[black] (19.5-.8,-0.5-1.5)--(19.5+.8,-0.5-1.5);
 	\filldraw[black] (19.5,-0.5-1.5) circle (0pt) node[black, anchor=north]{$T_2 $};
 	
 	\draw[black] (21,-0.5)--(21-.4,-0.5-1.5);
 	\draw[black] (21,-0.5)--(21+.4,-0.5-1.5);
 	\draw[black] (21-.4,-0.5-1.5)--(21+.4,-0.5-1.5);
 	\filldraw[black] (21,-0.5-1.5) circle (0pt) node[black, anchor=north]{$T_3 $};
 	\end{tikzpicture}
 	\caption{An example of the cover $(A_1,A_2)$ of the unicyclic graph $H$ with $m=4$.}\label{fig:cover}
 \end{figure}

 \noindent Using the cover $A_1$ and $A_2$ of $H$, we  define the \textit{decomposition} of  $\cp^\lambda(H;\one_{v_1})$, similarly as Definition \ref{def:decomposition}. Recall the graphical representation of the contact process discussed in Section \ref{subsubsec:cpprelim}, and let $\{N_v(t) \}_{v\in V(H)}$ and $\{N_{\vec{uv}}(t) \}_{\vec{uv}\in \vec{E}(H)}$ be the Poisson processes that define the recoveries and infections of $\cp^\lambda(H;\one_{v_1})$.

 	\begin{enumerate}
 		\item [1.] Initially, run $ \cp^\lambda(A_1;\one_{v_1})$, whose recoveries and infections are given by $\{(N_v(s))_{s\geq 0} \}_{v\in V(A_1)} $ and $\{(N_{\vec{uv}}(s))_{s\geq 0} \}_{\vec{uv}\in \vec{E}(A_1) } $, respectively. 
 		
 		\item[2.] In $A_1$, if either $v_{m'}$ or $\tilde{v}_{m'}$ becomes infected at time $t$ (and has been healthy until time $t-$), we start running a copy  of $ \cp^\lambda(A_2;\one_{v_{m'}})$ whose recoveries and infections are defined by $\{(N_v(s))_{s\geq t} \}_{v\in V(A_2)} $ and $\{(N_{\vec{uv}}(s))_{s\geq t} \}_{\vec{uv}\in \vec{E}(A_2) } $, respectively.
 		
 		\item [3.] When either $v_1$ or $\tilde{v}_1$ becomes infected at time $t'$ (while it has been healthy until $t'-$) in  a copy of $\cp^\lambda(A_2;\one_{v_{m'}})$, initiate a copy of $\cp^\lambda (A_1;\one_{v_1})$ that has $\{(N_v(s))_{s\geq t'} \}_{v\in V(A_1)} $ and $\{(N_{\vec{uv}}(s))_{s\geq t'} \}_{\vec{uv}\in \vec{E}(A_1) } $ as its event times, and return to Step 2.
 		
 		\item [4.] The process terminates when all vertices in all copies of two processes are healthy.
 	\end{enumerate}

Recalling the explanation below Definition \ref{def:decomposition}, we need to control the number of infections at the  endpoints $v_{m'},\tilde{v}_{m'}$ and $v_1, \tilde{v}_1$ of $A_1$ and $A_2$, respectively. Therefore, we first consider the following graphs $F^{m,l}$ and $\mathcal{F}^{m,l}$ which are basically a one-sided version of $A_1$ and $A_2$.
 
 \begin{itemize}
 	\item $F^{m,l}$ consists of a length-$(m-1)$ line $v_1v_2\ldots v_m$ and  depth $\leq l$ trees $T_j$  rooted at  $v_j$ except at $v_m$. $\rho=v_1$ is designated as its root.
 	
 	\item $\mathcal{F}^{m,l}$ denotes the above graph when $T_j, \; j\in [m-1]$ are i.i.d. $\gw(\xi)^l$. In this case, we write $\mathcal{T}_j$ instead of $T_j$ for each tree rooted at $v_i$.
 	
 	\item If $m=1$, we set $\mathcal{F}^{1,l}\sim \gw(\xi)^l$ to be a single vertex $v_1$.
 \end{itemize}
 One can also define the root-added contact process  $(X_t) \sim \cp^\lambda_{v_1^+}( F^{m,l}; \one_{v_1})$ by adding a permanently infected parent $v_1^+$ to $v_1$, and set the quantities $S(F^{m,l})$ and $M^l(F^{m,l})$ as Definition \ref{def:SMunicyclic}. In addition, we consider $B(F^{m,l})$, the  \textit{total infections at the end} and its expectation as follows.
 \begin{equation}\label{eq:def:BofF}
 \textbf{B}(F^{m,l}) :=
 \left|\left\{s \in  [0,\textbf{S}(F^{m,l})]: \; X_s(v_m)=1 \textnormal{ and } X_{s-}(v_m)=0 \right\}\right| ,
 \end{equation}
 and 
 $B(F^{m,l}) :=\E_{\textsc{cp}} \textbf{B}(F^{m,l})$. Then, we have the following lemma.

 \begin{lemma}\label{lem:recursion for F}
 	Let $m,l\geq 1$ be integers, 	$\ep\in(0,1)$ and $\mathfrak{c}=\{c_\delta \}_{\delta\in (0,1]}$ be a collection of positive constants. Then there exists $d_0(\ep, \mathfrak{c}) >0$ such that the following holds true. For any $\xi$ that satisfies $d:=\E \xi \geq d_0$ and (\ref{eq:concentration condition}) with $\mathfrak{c}$, we have for $\lambda=(1-\ep) d^{-1}$ and $\mathcal{F}^{m,l}$ as above that
 	\begin{eqnarray}
 	\P_{\textsc{gw}} \left(S(\mathcal{F}^{m,l}) \geq  t \right) &\leq 
 	t^{-\sqrt{d}} (\log t)^{-2} ~~~\textnormal{for all }t\geq \frac{2}{\ep}; \label{eq:unicycF tail S}
 	\\
 	\P_{\textsc{gw}} \left(M^l(\mathcal{F}^{m,l}) \geq \left(1- \frac{\ep}{10} \right)^l t \right) &\leq 
 	t^{-\sqrt{d}} (\log t)^{-2} ~~~\textnormal{for all }t\geq 2;\label{eq:unicycF tail M}
 	\\
 	\P_{\textsc{gw}} \left(B(\mathcal{F}^{m,l}) \geq d^{-\frac{3}{4}(m-1) } t \right) &\leq 
 	t^{-\sqrt{d}} (\log t)^{-2} ~~~\textnormal{for all }t\geq 2,\label{eq:unicycF tail B}
 	\end{eqnarray}
 	where $S(\mathcal{F}^{m,l})$, $M^l(\mathcal{F}^{m,l})$ and $B(\mathcal{F}^{m,l}) $  are given as above.
 \end{lemma}

 \begin{proof}
 	We prove this lemma by an induction on $m$, for each fixed $l$. Establishing (\ref{eq:unicycF tail S}) and (\ref{eq:unicycF tail M}) is along the same lines as (\ref{thm:tail bound of S}) and (\ref{thm:M tail bd}), so we postpone those parts of the proof to Section \ref{subsec:pfof recursion for F} in the Appendix.
 	
 	In the remaining proof, we show  (\ref{eq:unicycF tail B}). 
 	When $m=1$, $\mathcal{F}^{1,l}$ is a single vertex $v_1$ and we have the conclusion. Suppose that (\ref{eq:unicycF tail B}) holds for $m$ and we want to prove the same thing for $\mathcal{F}^{m+1,l}$.  Let $\{\mathcal{T}_j\}_{j=1}^{m-1}$ be i.i.d. $\gw(\xi)^l$ trees in $\mathcal{F}^{m,l}$, each rooted at $v_i$.
 	Dividing $\mathcal{F}^{m+1,l}$ into subgraphs rooted at each child of $v_1$, we have the following description:
 	\begin{itemize}
 		\item Let $u_1,\ldots, u_D$ be the children of $v_1$ inside $\mathcal{T}_1$ and $T_1,\ldots, T_D$ be the subtrees in $\mathcal{T}_1$ rooted at each of them, respectively. Hence, $T_1,\dots,T_D \sim$ i.i.d. $\gw(\xi)^{l-1}$ given $D$.
 		
 		\item Let $F$ be the subtree of $\mathcal{F}^{m+1,l}$ rooted at $v_2$. That is, $F$ is the graph consists with the length-$(m-1)$ line  $v_2\ldots v_{m+1}$ with each $v_i, \;2\leq i\leq m$ being a root of the tree $\mathcal{T}_i$. Hence, $F \sim \mathcal{F}^{m,l}$.
 	\end{itemize}
 	Based on this decomposition of $\mathcal{F}= \mathcal{F}^{m+1,l}$, we can establish the two recursive inequalities as before. 
 	To derive the first one, let $(\bar{X}_t)\sim \cp^\lambda(\mathcal{F},\one_{v_1})$ and
 	\begin{equation*}
 	\bar{\textbf{B}}(\mathcal{F}) := 
 	\left|\left\{s \in [0, \textbf{R}(\mathcal{F})]:\; \bar{X}_s(v_{m+1})=1 \textnormal{ and } \bar{X}_{s-}(v_{m+1})=0 \right\}\right| ,
 	\end{equation*}
 	where $\textbf{R}(\mathcal{F})$ is the survival time of $(\bar{X}_t)$.
 	We also set $\bar{B}(\mathcal{F}) :=\E_{\textsc{cp}} \bar{\textbf{B}}(\mathcal{F})$, and let $\bar{\textbf{B}}(F)$ and $\bar{B}(F)$ be the analogue for $F$.
 	Then, based on the same idea as Propositions \ref{prop:recursion of R}, we have
 	\begin{equation*}
 	\bar{B}(\mathcal{F}^{m+1,l}) \leq \frac{\lambda \bar{B}(F)}{1-\lambda^2 \left(R(F) + \sum_{i=1}^D R(T_i)\right)}, \quad \textnormal{if } \lambda^2 \left(R(F) + \sum_{i=1}^D R(T_i)\right)<1.
 	\end{equation*}
 	Thus, Proposition \ref{prop:S vs R} and (\ref{eq:recursion eq typical}) tell us that
 	\begin{equation}\label{eq:Brecursion on F typical}
 	B(\mathcal{F}^{m+1,l}) \leq \frac{\lambda B(F)}{1-\lambda-2\lambda^2 \left(S(F) + \sum_{i=1}^D S(T_i)\right)}, \quad \textnormal{if } \lambda+2\lambda^2 \left(S(F) + \sum_{i=1}^D S(T_i)\right)<1,
 	\end{equation}
 	and hence we obtain the first recursive inequality.
 	
 	The second inequality comes from the same idea as in Proposition \ref{prop:Mrecursion atypcial}:
 	\begin{equation}\label{eq:Brecursion on F atypical}
 	B(\mathcal{F}^{m+1,l}) \leq \lambda B(F) \prod_{i=1}^D (1+\lambda S(T_i)).
 	\end{equation}
 	
 	As before, we  show (\ref{eq:unicycF tail B})  separately on different regimes of $t$. Details in the proof are also similar to the proof of Theorem \ref{thm:tail bound of S}.
 	
 	\vspace{2mm}
 	\noindent \textbf{Case 1.} For $t\leq d^{\frac{1}{10}}.$
 	
 	Suppose that 
 	\begin{equation*}
 	B(F) \leq d^{-\frac{3}{4}(m-1) +\frac{1}{2}}, \quad
 	\textnormal{and}
 	\quad
 	S(F)+\sum_{i=1}^D S(T_i) \leq \frac{2d}{\ep}\left(1+\frac{\ep}{3} \right).
 	\end{equation*}
 	Then, by (\ref{eq:Brecursion on F typical}), $B(\mathcal{F}^{m+1,l}) \leq d^{-\frac{3m}{4}}$. Therefore, for any $t\geq 1$,
 	\begin{equation}\label{eq:Brecursion on F split small t}
 	\begin{split}
 	\Pgw \left(B(\mathcal{F}^{m+1,l}) \geq d^{-\frac{3m}{4}}t \right)
 	&\leq
 	\P\left(D \geq \left(1+\frac{\ep}{6} \right)d \right)
 	+
 	\P\left(B(F) \geq d^{-\frac{3}{4}(m-1)+\frac{1}{2} } \right)
 	\\
 	&+
 	\Pgw \left(S(F) + \sum_{i=1}^D S(T_i) \geq \frac{2d}{\ep} \left(1+\frac{\ep}{3} \right) \right).
 	\end{split} 	
 	\end{equation}
 	By following the logics from Section \ref{subsubsection1} and relying on the induction hypothesis that tells us the tail probability of $B(F)$, we see that the r.h.s. of (\ref{eq:Brecursion on F split small t}) is at most $t^{-\sqrt{d}} (\log t)^{-2}$ for all $2 \leq t\leq d^{\frac{1}{10}}.$
 	
 	\vspace{2mm}
 	\noindent \textbf{Case 2.} For $t\geq d^{\frac{1}{10}}$.
 	
 	For convenience, set
 	$$\widehat{S} = \prod_{i=1}^D (1+\lambda S(T_i)),
 	\quad \textnormal{and}\quad
 	\widetilde{B} = d^{\frac{3}{4}(m-1)} B(F).
 	$$
 	We use (\ref{eq:Brecursion on F atypical}) and hence prove the following. 
 	
 	\begin{claim}\label{claim:B tail large t}
 		Under the above setting, we have
 		\begin{equation}\label{eq:Brecursion:in claim}
 		\Pgw \left( B(\mathcal{F}^{m+1,l}) \geq d^{-\frac{3m}{4}} t \right)
 		\leq
 		\Pgw \left(\widetilde{B}\cdot \widehat{S} \geq
 		\frac{1}{2}d^{\frac{1}{4}}t
 		\right)
 		\leq 
 		t^{-\sqrt{d}} (\log t)^{-2}, \quad \textnormal{for all } t\geq d^{\frac{1}{10}}.
 		\end{equation}
 		Note that the first inequality is just a rewriting of (\ref{eq:Brecursion on F atypical}).
 	\end{claim}
 	
 	The idea is to split the probability as follows: let $\alpha = d^{\frac{1}{10}}$, and we see that
 	\begin{equation}\label{eq:BtimesS recursion}
 	\begin{split}
 	\Pgw \left(\widetilde{B}\cdot \widehat{S} \geq
 	t
 	\right)
 	&\leq 
 	\Pgw \left(\widehat{S} \geq t\alpha \right) +\Pgw \left(\widetilde{B} \geq \frac{t}{2\alpha^2} d^{\frac{1}{4}}  \right)\\
 	&+
 	\sum_{k=0}^{k_0} \Pgw\left(\widehat{S} \geq t\alpha^{-k} \right) \cdot
 	\Pgw \left(\widetilde{B} \geq \frac{\alpha^{k-1} }{2}d^{\frac{1}{4}} \right),
 	\end{split}
 	\end{equation}
 	where we set $k_0 = \lfloor\frac{\log t}{\log \alpha}\rfloor$-1, the point where $\alpha^{k+1}\leq t$ is the closest to $t$. Similar ideas as Lemma \ref{lem:recursion principle atypical} yields the conclusion (\ref{eq:Brecursion:in claim}), but the computation is much simpler. We postpone the remaining details of Claim \ref{claim:B tail large t} to Section \ref{subsec:pfof claim} in the Appendix.
 \end{proof}

 \subsection{Back to unicyclic graphs and the  proof of Proposition \ref{prop:SMunicyclic} for small $t$}\label{subsec:unic prop small t2}
 
 In the previous subsection, the graph $\mathcal{F}^{m,l}$ is introduced to describe a way of decomposing  $\gwc^1$-processes. Now we combine the pieces together and derive a recursive tail probability estimates for $\gwc^1$-processes.
 
 For integers $m_1,m_2\geq 1$ and $l\geq 0$, let $\underbar{m} = (m_1,m_2)$ and define the random graph $\mathcal{A}^{\underbar{m},l}$ as 
 \begin{itemize}
 	\item $\mathcal{A}^{\underbar{m},l}$ consists of a length-($m_1+m_2$) line $v_{-m_1}\ldots v_{-1}v_0 v_1\ldots v_{m_2}$ and i.i.d.$\;$trees $\mathcal{T}_j \sim \gw(\xi)^l$  rooted at $v_j$ for $j\in \{-m_1+1,\ldots, m_2-1\}$. We designate $\rho=v_0$ as the root of $\mathcal{A}^{\underbar{m},l}$.
 \end{itemize}
 
 We define the \textit{root-added contact process} on $\mathcal{A}^{\underbar{m},l}$ by adding a permanently infected parent $\rho^+$ whose only connection is to $\rho=v_0$. Further, we define $S(\mathcal{A}^{\underbar{m},l})$ and $M^l(\mathcal{A}^{\underbar{m},l})$ as Definition \ref{def:SMunicyclic}, and $B(\mathcal{A}^{\underbar{m},l})$ as  (\ref{eq:def:BofF}) but considering the infections at both endpoints $v_{-m_1}$, $v_{m_2}$ (not just a single endpoint as in (\ref{eq:def:BofF})).  Under this setting, we first establish the tail probability estimates for $\mathcal{A}^{\underbar{m},l}$.
 
 \begin{lemma}\label{lem:Drecursion}
 	Let $m_1, m_2,l\geq 1$ be integers, 	$\ep\in(0,1)$ and $\mathfrak{c}=\{c_\delta \}_{\delta\in (0,1]}$ be a collection of positive constants. Then there exists $d_0(\ep, \mathfrak{c}) >0$ such that the following holds true. For any $\xi$ that satisfies $d:=\E \xi \geq d_0$ and (\ref{eq:concentration condition}) with $\mathfrak{c}$, we have for $\lambda=(1-\ep) d^{-1}$ and $\mathcal{A}^{\underbar{m},l}$ as above that
 	\begin{eqnarray}
 	\P_{\textsc{gw}} \left(S(\mathcal{A}^{\underbar{m},l}) \geq  t \right) &\leq 
 	t^{-\sqrt{d}} (\log t)^{-2} ~~~\textnormal{for all }t\geq \frac{2}{\ep}; \label{eq:unicycD tail S}
 	\\
 	\P_{\textsc{gw}} \left(M^l(\mathcal{A}^{\underbar{m},l}) \geq \left(1- \frac{\ep}{10} \right)^l t \right) &\leq 
 	t^{-\sqrt{d}} (\log t)^{-2} ~~~\textnormal{for all }t\geq 2;\label{eq:unicycD tail M}
 	\\
 	\P_{\textsc{gw}} \left(B(\mathcal{A}^{\underbar{m},l}) \geq d^{-\frac{3}{4}(m_1 \wedge m_2) } t \right) &\leq 
 	t^{-\sqrt{d}} (\log t)^{-2} ~~~\textnormal{for all }t\geq 2,\label{eq:unicycD tail B}
 	\end{eqnarray}
 	where $S(\mathcal{A}^{\underbar{m},l})$, $M^l(\mathcal{A}^{\underbar{m},l})$ and $B(\mathcal{A}^{\underbar{m},l}) $  are given as above.	
 \end{lemma}
 
 We can prove this lemma based on Lemma \ref{lem:recursion for F}, by a straight-forward applications of the recursive argument from Theorems \ref{thm:tail bound of S} and \ref{thm:M tail bd}. To prevent repeating the same argument many times, we postpone the proof until Section \ref{subsec:pfof Drecursion} in the Appendix. 
 
 We conclude this subsection by settling the following lemma, which establishes Proposition \ref{prop:SMunicyclic} for $t\leq d^{\frac{1}{10}}$.
 
 \begin{lemma}\label{lem:unicyclic recursion small t}
 	Let $m\geq 2$ and $l\geq 1$ be integers, $\ep \in (0,1)$ and $\mathfrak{c}=\{c_\delta \}_{\delta \in (0,1]}$ be positive constants. Then, under the setting of Proposition \ref{prop:SMunicyclic}, we have
 	\begin{equation}\label{eq:unicyc tail:small t}
 	\begin{split}
 	\P_{\textsc{gw}} \left(S(\mathcal{H}^{m,l}) \geq  t \right) &\leq 
 	3t^{-\sqrt{d}} (\log t)^{-2} ~~~\textnormal{for all } \frac{2}{\ep} \leq t \leq d^{\frac{1}{10}};
 	\\
 	\P_{\textsc{gw}} \left(M^l(\mathcal{H}^{m,l}) \geq \left(1- \frac{\ep}{10} \right)^l t \right) &\leq 
 	3t^{-\sqrt{d}} (\log t)^{-2} ~~~\textnormal{for all } 2\leq t \leq d^{\frac{1}{10}}.
 	\end{split}
 	\end{equation}
 \end{lemma}
 
 \begin{proof}
 	Let $\mathcal{H}^{m,l}\sim \gwc^1 (\xi,m)^l$, set $m' = \lceil \frac{m+1}{2}\rceil$ and  consider the cover $A_1$ and $A_2$ of $\mathcal{H}^{m,l}$ defined in Definition \ref{def:cover of unic}. Set $m_1 = m+1-m'$ and $m_2 = m'-1$. Then, we can see that $m_1, m_2 \geq 1 $ for $m\geq 2$ and $A_1, A_2 \sim \mathcal{A}^{\underbar{m},l}$ for $\underbar{m}=(m_1,m_2)$, with $\mathcal{A}^{\underbar{m},l}$ defined as in the beginning of Section \ref{subsec:unic prop small t2}. Thus, we have (\ref{eq:unicycD tail S}), (\ref{eq:unicycD tail M}) and (\ref{eq:unicycD tail B}), where the last one in particular implies that
 	\begin{equation*}
 	\Pgw \left(B(A_1) \geq d^{-\frac{3}{4}}t \right) \leq t^{-\sqrt{d}} (\log t)^{-2}, \quad \textnormal{for all }t\geq 2,
 	\end{equation*}
 	and the same thing for $A_2$ as well.
 	
 	Recall the decomposition of the contact process on $\mathcal{H}^{m,l}$ by its cover (Definition \ref{def:decomposition}). According to its formulation, the quantities $R(\mathcal{H}^{m,l})$ and $\bar{M}^l(\mathcal{H}^{m,l})$ (see Definition \ref{def:SMunicyclic} for their definitions) can be estimated as follows.
 	\begin{equation*}
 	\begin{split}
 	&R(\mathcal{H}^{m,l}) 
 	\leq 
 	R(A_1) + B(A_1)R(A_2) + B(A_1)B(A_2)R(\mathcal{H}^{m,l}),
 	\;\; \textnormal{and hence}\\
 	&R(\mathcal{H}^{m,l}) \leq \frac{R(A_1)+ B(A_1)R(A_2)}{1-B(A_1)B(A_2)}.
 	\end{split}
 	\end{equation*}
 	We have the same thing for $\bar{M}^l(\mathcal{H}^{m,l})$, namely,
 	\begin{equation*}
 	\bar{M}^l(\mathcal{H}^{m,l}) \leq \frac{\bar{M}^l(A_1)+ B(A_1)\bar{M}^l(A_2)}{1-B(A_1)B(A_2)}.
 	\end{equation*}
 	Observe that  Proposition \ref{prop:S vs R} and Corollary \ref{cor:MvsMprime} can be applied to $\mathcal{H}^{m,l}$ so that
 	\begin{equation*}
 	\frac{S(\mathcal{H}^{m,l})}{1+\lambda S(\mathcal{H}^{m,l})} \leq R(\mathcal{H}^{m,l}), \quad \textnormal{and} \quad 
 	M^l(\mathcal{H}^{m,l})\leq (1+\lambda S(\mathcal{H}^{m,l})) \bar{M}^l(\mathcal{H}^{m,l}),
 	\end{equation*}
 	and hence we have
 	\begin{equation*}
 	\begin{split}
 	S(\mathcal{H}^{m,l}) 
 	&\leq 
 	\frac{S(A_1)+ B(A_1)S(A_2)}{1-B(A_1)B(A_2) -\lambda S(A_1)- \lambda B(A_1)S(A_2) }\\
 	M^l(\mathcal{H}^{m,l})
 	&\leq 	 
 	\frac{M^l(A_1)+ B(A_1)M^l(A_2)}{1-B(A_1)B(A_2) -\lambda S(A_1)- \lambda B(A_1)S(A_2) },
 	\end{split}
 	\end{equation*}
 	To establish the first inequality of (\ref{eq:unicyc tail:small t}), we first observe that if
 	\begin{equation*}
 	t\leq d^{\frac{1}{10}},\quad S(A_1) , \; S(A_2) \leq \left(1-d^{-\frac{5}{9}}  \right)t, \quad
 	\textnormal{and} \quad B(A_1) ,\; B(A_2) \leq d^{-\frac{5}{8}}, 
 	\end{equation*}
 	then $S(\mathcal{H}^{m,l}) \leq t$. Therefore, for $t\leq d^{\frac{1}{2}}$, we have
 	\begin{equation*}
 	\begin{split}
 	\P_{\textsc{gw}} \left(S(\mathcal{H}^{m,l}) \geq  t \right)
 	&\leq
 	\Pgw\left(S(A_1) \geq \left(1-d^{-\frac{5}{9}}  \right)t \right) +
 	\Pgw \left(S(A_2) \geq \left(1-d^{-\frac{5}{9}}  \right)t \right)\\
 	&+
 	\Pgw\left(B(A_1) \geq d^{-\frac{5}{8}} \right)
 	+
 	\Pgw \left(B(A_2) \geq d^{-\frac{5}{8}} \right).
 	\end{split}
 	\end{equation*}
 	Since $(1-d^{-\frac{5}{9}})^{-\sqrt{d}}\leq 1+2d^{-\frac{1}{18}}$, applying Lemma \ref{lem:Drecursion} gives us that the r.h.s.$\;$of the above is bounded by $3t^{-\sqrt{d}}(\log t)^{-2}$ for $2/\ep \leq t\leq d^{\frac{1}{10}}$. We can also do the same thing for $M^l(\mathcal{H}^{m,l})$ to deduce the conclusion. 
 \end{proof}

 \subsection{Proof of Proposition \ref{prop:SMunicyclic} for large $t$}\label{subsec:unic prop large t}
 
 In this section, we conclude the proof of Proposition  \ref{prop:SMunicyclic}. Consider $\mathcal{H}^{m,l} \sim \gwc^1(\xi,m)^l$.
 
 \begin{figure}
 	\centering
 	\begin{tikzpicture}[scale=0.45]
 	\draw \boundellipse{0,0}{4}{1};
 	\filldraw[black] (-3,-2.64575131106/4) circle (2pt) node[black, anchor=south]{$v_1 $};
 	\filldraw[black] (3,-2.64575131106/4) circle (2pt) node[black, anchor=south]{$v_4 $};
 	\filldraw[black] (1,-3.87298334621/4) circle (2pt) node[black, anchor=south]{$v_3 $};
 	\filldraw[black] (-1,-3.87298334621/4) circle (2pt) node[black, anchor=south]{$v_2 $};

 	\draw[black] (-1,-3.87298334621/4)--(-1.8,-3.87298334621/4-1.5);
 	\draw[black] (-1,-3.87298334621/4)--(-1+.8,-3.87298334621/4-1.5);
 	\draw[black] (-1.8,-3.87298334621/4-1.5)--(-1+.8,-3.87298334621/4-1.5);
 	\filldraw[black] (-1,-3.87298334621/4-1.5) circle (0pt) node[black, anchor=north]{$\mathcal{T}_2 $};
 	
 	\draw[black] (1,-3.87298334621/4)--(1.4,-3.87298334621/4-1.5);
 	\draw[black] (1,-3.87298334621/4)--(1-.4,-3.87298334621/4-1.5);
 	\draw[black] (1.4,-3.87298334621/4-1.5)--(1-.4,-3.87298334621/4-1.5);
 	\filldraw[black] (1,-3.87298334621/4-1.5) circle (0pt) node[black, anchor=north]{$\mathcal{T}_3 $};
 	
 	\draw[black] (3,-2.64575131106/4)--(3.5,-2.64575131106/4-2.5);
 	\draw[black] (3,-2.64575131106/4)--(2.5,-2.64575131106/4-2.5);
 	\draw[black] (3.5,-2.64575131106/4-2.5)--(2.5,-2.64575131106/4-2.5);
 	\filldraw[black] (3,-2.64575131106/4-2.5) circle (0pt) node[black, anchor=north]{$\mathcal{T}_4 $};
 	
 	\draw[black] (-3,-2.64575131106/4)--(-6,-2.64575131106/4-1);
 	\draw[black] (-3,-2.64575131106/4)--(-4.5,-2.64575131106/4-0.8);
 	\draw[black] (-3,-2.64575131106/4)--(-3,-2.64575131106/4-0.6);
 	
 	\filldraw[black] (-6,-2.64575131106/4-1) circle (2pt);
 	\filldraw[black] (-6,-2.64575131106/4-1.4) circle (0pt) node[black,anchor=west]{$u_1$};
 	\filldraw[black] (-4.5,-2.64575131106/4-1+.2) circle (2pt);
 	\filldraw[black] (-4.5,-2.64575131106/4-1-.25) circle (0pt) node[black,anchor=west]{$u_2$};		
 	\filldraw[black] (-3,-2.64575131106/4-1+.4) circle (2pt);
 	\filldraw[black] (-3,-2.64575131106/4-1-.1) circle (0pt) node[black,anchor=west]{$u_3$};
 	
 	\draw[black]
 	(-6,-2.64575131106/4-1) -- (-6.2,-2.64575131106/4-3) --
 	(-5.8,-2.64575131106/4-3) -- (-6,-2.64575131106/4-1);
 	\filldraw[black] (-6,-2.64575131106/4-3) circle (0pt) node[black,anchor=north]{$T_1$};
 	
 	\draw[black] (-4.5,-2.64575131106/4-1+.2) -- (-4.8,-2.64575131106/4-3.5+.2) -- (-4.2,-2.64575131106/4-3.5+.2) -- (-4.5,-2.64575131106/4-1+.2);
 	\filldraw[black] (-4.5,-2.64575131106/4-3.5+.2) circle (0pt) node[black,anchor=north]{$T_2$};
 	
 	\draw[black] (-3,-2.64575131106/4-1+.4) -- (-3.3,-2.64575131106/4-4+.4) -- (-2.7,-2.64575131106/4-4+.4) -- (-3,-2.64575131106/4-1+.4);
 	\filldraw[black] (-3,-2.64575131106/4+.4-4) circle (0pt) node[black,anchor=north]{$T_3$};

 	\filldraw[black] (0,1.2) circle (0pt) node[black, anchor=south]{\large{$\mathcal{H}^{m,l}$} };
 	\draw[black] (6,-1) -- (8,-1);
 	\draw[black] (8,-1) -- (7.7,-.8);
 	\draw[black] (8,-1) -- (7.7,-1.2);
 	
 	\filldraw[black] (10,-.5) circle (2pt) node[black, anchor=south]{$u_1$};
 	\filldraw[black] (12,-.5) circle (2pt) node[black, anchor=south]{$u_2$};
 	\filldraw[black] (14,-.5) circle (2pt) node[black, anchor=south]{$u_3$};
 	
 	\draw[black] (10,-.5) -- (10.2, -.5-2) -- (9.8,-.5-2) -- (10,-.5);
 	\filldraw[black] (10,-2.5) circle (0pt) node[black,anchor=north]{$T_1$};
 	
 	\draw[black] (12,-.5) -- (12.3,-.5-2.5) -- (11.7,-.5-2.5) -- (12,-.5);
 	\filldraw[black] (12,-3) circle (0pt) node[black,anchor=north]{$T_2$};
 	
 	\draw[black] (14,-.5) -- (14.3,-.5-3) -- (13.7,-.5-3) -- (14,-.5);
 	\filldraw[black] (14,-3.5) circle (0pt) node[black,anchor=north]{$T_3$};
 	
 	
 	\draw \boundellipse{20,0}{4}{1};
 	\filldraw[black] (20-3,-2.64575131106/4) circle (2pt) node[black, anchor=south]{$v_1 $};
 	\filldraw[black] (23,-2.64575131106/4) circle (2pt) node[black, anchor=south]{$v_4 $};
 	\filldraw[black] (21,-3.87298334621/4) circle (2pt) node[black, anchor=south]{$v_3 $};
 	\filldraw[black] (20-1,-3.87298334621/4) circle (2pt) node[black, anchor=south]{$v_2 $};

 	\draw[black] (20-1,-3.87298334621/4)--(20-1.8,-3.87298334621/4-1.5);
 	\draw[black] (20-1,-3.87298334621/4)--(20-1+.8,-3.87298334621/4-1.5);
 	\draw[black] (20-1.8,-3.87298334621/4-1.5)--(20-1+.8,-3.87298334621/4-1.5);
 	\filldraw[black] (20-1,-3.87298334621/4-1.5) circle (0pt) node[black, anchor=north]{$\mathcal{T}_2 $};
 	
 	\draw[black] (21,-3.87298334621/4)--(21.4,-3.87298334621/4-1.5);
 	\draw[black] (21,-3.87298334621/4)--(21-.4,-3.87298334621/4-1.5);
 	\draw[black] (21.4,-3.87298334621/4-1.5)--(21-.4,-3.87298334621/4-1.5);
 	\filldraw[black] (21,-3.87298334621/4-1.5) circle (0pt) node[black, anchor=north]{$\mathcal{T}_3 $};
 	
 	\draw[black] (23,-2.64575131106/4)--(23.5,-2.64575131106/4-2.5);
 	\draw[black] (23,-2.64575131106/4)--(22.5,-2.64575131106/4-2.5);
 	\draw[black] (23.5,-2.64575131106/4-2.5)--(22.5,-2.64575131106/4-2.5);
 	\filldraw[black] (23,-2.64575131106/4-2.5) circle (0pt) node[black, anchor=north]{$\mathcal{T}_4 $};
 	
 	\filldraw[black]  (20,1.2) circle (0pt) node[black,anchor=south]{\large{$\dot{\mathcal{H}}$}};

 	\end{tikzpicture}
 	\caption{An illustration of dividing  $\mathcal{H}^{m,l}\sim \gwc^1(\xi,m)^l$ into   $T_1,T_2,T_3 \sim \gw(\xi)^{l-1}$ and $\dot{ \mathcal{H}}\sim \gwc^2(\xi,m)^l$, given that $m=4$ and $D=3$.}\label{fig:gwc}
 \end{figure}
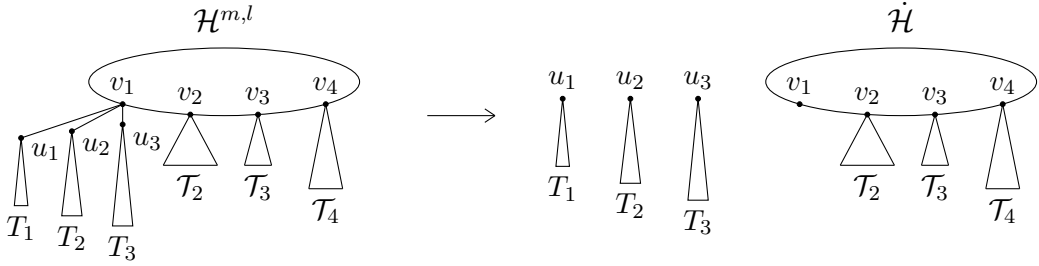
 
 For a $\gwc^1$-process $\mathcal{H}^{m,l}\sim \gwc^1 (\xi,m)^l$, we can derive a recursive inequality on $\mathcal{H}^{m,l}$ similar to (\ref{eq:recursion eq atypical}) using $\gwc^2$-process. As described in Figure \ref{fig:gwc} let $v_1\ldots v_m v_1$ be its cycle part, and let $\mathcal{T}_j, \; j\in[m]$ be i.i.d. $\gw(\xi)^l$ rooted at each $v_i$. Let $D$ be the degree of $v_1$ in $\mathcal{T}_1$, and let $T_1,\ldots,T_D$ be the subtrees rooted at $u_1,\ldots u_D$, the children of $v_1$ in $\mathcal{T}_1$. Further, let $\dot{\mathcal{H}}$ be the graph that consists of the cycle $v_1\ldots v_m v_1$ and the trees $\mathcal{T}_j$ rooted at $v_j$ for $j\in\{2,\ldots,m\}$, which is $\dot{\mathcal{H}} \sim \gwc^2(\xi,m)^l$. Also, note that $T_1,\ldots , T_D \sim $ i.i.d.$\;\gw(\xi)^{l-1}$. Then, the following lemma shows the connection between $\gwc^1$- and $\gwc^2$-processes. Its proof is based on the same ideas as Proposition \ref{prop:Srecursion atypical}, and it will be presented in Section \ref{subsubsec:unic2}.
 
 \begin{lemma}\label{lem:gwc1 recursion atypical}
 	Under the above setting, we have
 	\begin{equation}\label{eq:gwc1 Srecursion atypical}
 	S(\mathcal{H}^{m,l}) \leq (1+2\lambda S(\dot{\mathcal{H}})) \prod_{i=1}^D (1+\lambda S(T_i)).
 	\end{equation}
 \end{lemma}
 
 This shows that understanding the tail probabilities on $\gwc^2$-processes is helpful in estimating those on $\gwc^1$-processes. The following proposition gives necessary estimates for $\gwc^2$-processes.

 We prove Proposition \ref{prop:SMunicyclic} and Corollary \ref{cor:SMgwc2} in tandem. This is done by an inductive argument which we detail below. Let $l\geq 1$ be a fixed integer. We  establish both (\ref{eq:unicyc tail:in prop}) and (\ref{eq:gwc2 tail:in prop}) for all $m$ via the following steps.
 \begin{itemize}[leftmargin=20mm]
 	\item [Step 1.]  Show that (\ref{eq:unicyc tail:in prop}) (resp., (\ref{eq:gwc2 tail:in prop})) are true for $m=1$ (resp., $m=2$).
 	
 	\item [Step 2.] For each integer $m\geq 2$, prove that if (\ref{eq:gwc2 tail:in prop}) is true for $m$, then we have  (\ref{eq:unicyc tail:in prop}) for $m$. 
 	
 	\item [Step 3.] For each integer $m\geq 2$, prove that if (\ref{eq:unicyc tail:in prop}) and (\ref{eq:gwc2 tail:in prop}) are true for $m$, then we have (\ref{eq:gwc2 tail:in prop}) for $m+1$. 
 	
 	\item [3-1.] Show that (\ref{eq:gwc2 tail:in prop}) holds when $t\leq d^{\frac{1}{10}}$ (for $m+1$).
 	
 	\item [3-2.] Show that (\ref{eq:gwc2 tail:in prop}) holds when $t\geq  d^{\frac{1}{10}}$ (for $m+1$).
 \end{itemize}

 \subsubsection{Proof of Proposition \ref{prop:SMunicyclic} and Corollary \ref{cor:SMgwc2}, Step 1}\label{subsubsec:unic1}

 Proving (\ref{eq:unicyc tail:in prop}) for $m=1$ is straight-forward, since $\mathcal{H}^{1,l} \sim \gw(\xi)^l$ and the result follows from Theorems \ref{thm:tail bound of S} and \ref{thm:M tail bd}. To establish (\ref{eq:gwc2 tail:in prop}) for $m=2$, we divide (\ref{eq:gwc2 tail:in prop}) into two cases, $t\leq d^{\frac{1}{10}}$ and $t\geq d^{\frac{1}{10}}$. 	Observe that $\dot{\mathcal{H}}\sim \gwc^2(\xi,2)^l$ is the same as $(\mathcal{T}_2,v_2) \sim \gw(\xi)^l$ with a parent $\rho = v_1$ connected to $v_2$ via a pair of edges.
 
 \vspace{2mm}
 \noindent \textbf{Case 1.} $t\leq d^{\frac{1}{10}}$.
 \vspace{1mm}
 
 Letting $S(\dot{\mathcal{H}})= S_2(\dot{\mathcal{H}})$, $M^l(\dot{\mathcal{H}})=M^l_2(\dot{\mathcal{H}})$ as Definition \ref{def:SMdef gwc2}, an analog of Proposition \ref{prop:S vs R} tells us that
 \begin{equation}\label{eq:Sunic tail initialcase}
 \frac{S(\dot{\mathcal{H}})}{1+2\lambda S(\dot{\mathcal{H}})} \leq R(\mathcal{T}_2) \leq S(\mathcal{T}_2), \quad 
 \textnormal{and} \quad 
 \frac{M^l(\dot{\mathcal{H}})}{1+2\lambda S(\dot{\mathcal{H}})}
 \leq \bar{M}^l(\mathcal{T}_2) \leq M^l(\mathcal{T}_2).
 \end{equation}
 The left inequality can be rewritten as
 \begin{equation*}
 S(\dot{ \mathcal{H}}) \leq \frac{S(\mathcal{T}_2)}{1-2\lambda S(\mathcal{T}_2)}.
 \end{equation*}
 Based on this information and on Theorems \ref{thm:tail bound of S} and \ref{thm:M tail bd}, we deduce that for $3/\ep \leq  t\leq d^{\frac{1}{10}}$,
 \begin{equation}\label{eq:Munic tail initialcase}
 \Pgw \left(S(\dot{\mathcal{H}}) \geq t \right)
 \leq 
 \Pgw \left(S(\mathcal{T}_2) \geq \left(1-d^{-\frac{3}{4}} \right)t \right)
 \leq 2 t^{-\sqrt{d}} (\log t)^{-2},
 \end{equation}
 since $(1-d^{-\frac{3}{4}})^{-\sqrt{d}}<\frac{3}{2}$ and $|\log(1-d^{-\frac{3}{4}})| \leq \frac{1}{100} \log t$ for large $d$. For the same reason, we obtain that
 \begin{equation*}
 \begin{split}
 \Pgw \left(M^l(\dot{\mathcal{H}}) \geq \left(1-\frac{\ep}{10}\right)^l t \right)
 &\leq\Pgw\left(M^l(\mathcal{T}_2) \geq \left(1-d^{-\frac{3}{4}} \right) \left(1-\frac{\ep}{10} \right)^l t \right)\\
 &\;\;+\Pgw \left( S(\mathcal{T}_2) \geq d^{\frac{1}{5}} \right),
 \end{split}
 \end{equation*}
 and the r.h.s.$\;$is bounded by $2t^{-\sqrt{d}} (\log t)^{-2}$.
 
 \vspace{2mm}
 \noindent \textbf{Case 2.} $t\geq d^{\frac{1}{10}}.$	
 \vspace{1mm}
 
 We can deduce the conclusion by exactly the same argument as Theorem \ref{thm:tail bound of S} (Sections \ref{subsubsection2}, \ref{subsubsection3}) and Theorem \ref{thm:M tail bd} (Sections \ref{subsubsec:M2}, \ref{subsubsec:M3}). Indeed, the recursive inequalities (\ref{eq:recursion eq atypical}), (\ref{eq:Mrecursion atypical}) do not change for $\dot{\mathcal{H}}$ since they are independent of the number of connections between $v_1$ and $v_2$.

 \subsubsection{Proof of Proposition \ref{prop:SMunicyclic} and Corollary \ref{cor:SMgwc2}, Step 2}\label{subsubsec:unic2}
 
 We have already established (\ref{eq:unicyc tail:in prop}) for $t\leq d^{\frac{1}{10}}$ in Lemma \ref{lem:unicyclic recursion small t}. Hence, we focus on settling the other part, $t\geq d^{\frac{1}{10}}$.		We begin with presenting the proof of Lemma \ref{lem:gwc1 recursion atypical}, which shows the relation between $\gwc^1$- and $\gwc^2$-processes.
 
 \begin{proof}[Proof of Lemma \ref{lem:gwc1 recursion atypical}]
 	Let $\mathcal{H}^{m,l}$, $\dot{\mathcal{H}}$, $D$,  $\{T_i\}_{i=1}^D$ and $\{u_i\}_{i=1}^D$ be as in the statement of the lemma. As we run $\cp^\lambda_{v_1^+}(\mathcal{H}^{m,l};\one_{v_1})$, we adopt the analogous idea as Proposition \ref{prop:Srecursion atypical} by ignoring the recoveries at $v_1$ except when all the other vertices are healthy. This translates the original process $\cp^\lambda_{v_1^+}(\mathcal{H}^{m,l};\one_{v_1})$ to the product chain of $\cp^\lambda_{v_1}(T_i ), \; i\in[D]$ and $\cp^\lambda_{v_2}(\dot{\mathcal{H}})$ in the perspectives of Proposition \ref{prop:Srecursion atypical}. In $\cp^\lambda_{v_2} (\dot{\mathcal{H}};\zero)$, the first infection from $v_2$ happens with rate $2\lambda$ which is passed to either $v_3$ or $v_{m+1}$, each with probability $\frac{1}{2}$. Therefore, the expected excursion time for this chain is $S(\dot{\mathcal{H}})=\frac{1}{2}(S_3(\dot{\mathcal{H}})+S_{m+1}(\dot{\mathcal{H}}))$. Based on these observation, we follow the logic (\ref{eq:Sbound geometric trials}), (\ref{eq:stationary dist1}) and (\ref{eq:stationary dist prop}) from Proposition \ref{prop:Srecursion atypical} to obtain (\ref{eq:gwc1 Srecursion atypical}).
 \end{proof}
 
 Having Lemma \ref{lem:gwc1 recursion atypical} and (\ref{eq:gwc2 tail:in prop}) for $m$, we can establish the first inequality of (\ref{eq:unicyc tail:in prop}) by the same reasoning as Sections \ref{subsubsection2} and \ref{subsubsection3} (Lemma \ref{lem:recursion principle atypical} and Corollary \ref{cor:recursion principle atypical}). 
 
 For the second line of (\ref{eq:unicyc tail:in prop}), we have the following recursive inequality for $M^l(\mathcal{H}^{m,l})$ which can be derived analogously as Proposition \ref{prop:Mrecursion atypcial}.
 \begin{equation}\label{eq:Mrecursion H atypical}
 M^l(\mathcal{H}^{m,l}) \leq 
 2\lambda M^l(\dot{\mathcal{H}}) \prod_{i=1}^D(1+\lambda S(T_i)) + \lambda(1+2\lambda S(\dot{\mathcal{H}})) \sum_{i=1}^D M^{l-1}(T_i) \prod_{\substack{j=1\\j\neq i} }^{D} (1+\lambda S(T_j)).
 \end{equation}
 
 As done in Section \ref{subsubsec:M2}, we set 
 \begin{equation*}
 T_0:= \dot{\mathcal{H}}, \quad \widetilde{M}_i := \left(1-\frac{\ep}{10} \right)^{-(l-1)} M^{l-1}(T_i),\quad \textnormal{and} \quad 
 W_i := \widetilde{M}_i \vee S(T_i).
 \end{equation*}
 Moreover, let $\widetilde{M} = (1-\frac{\ep}{10})^{-l} M^l(\mathcal{H}^{m,l})$. Then, we can  derive the following from the above inequality.
 \begin{equation*}
 \left(1-\frac{\ep}{10} \right)\widetilde{M} \leq \prod_{i=0}^{D} (1+ 4\lambda W_i). 
 \end{equation*}
 Also, the assumption (\ref{eq:gwc2 tail:in prop}) and Theorem \ref{thm:M tail bd} tell us the tail bound on $W_i$, namely,
 $$\Pgw\left(W_i \geq t \right)\leq 4t^{-\sqrt{d}}(\log t)^{-2} + t^{-\sqrt{d}} (\log t)^{-2} \leq 5t^{-\sqrt{d}}(\log t)^{-2}, $$
 which falls into the assumption of Lemma \ref{lem:recursion principle atypical} and Corollary \ref{cor:recursion principle atypical}.	 Therefore, we conclude the proof of (\ref{eq:unicyc tail:in prop}) by obtaining that	 for all $t\geq d^{\frac{1}{10}}$,
 \begin{equation*}
 \Pgw \left(M^l(\dot{\mathcal{H}}) \geq \left(1-\frac{\ep}{10} \right)^l t \right) 
 \leq
 \Pgw\left(2\prod_{i=0}^{D} (1+ 4\lambda W_i) \geq t \right) \leq 4t^{-\sqrt{d}} (\log t)^{-2}.
 \end{equation*}

 \subsubsection{Proof of Proposition \ref{prop:SMunicyclic} and Corollary \ref{cor:SMgwc2}, Step 3-1}\label{subsubsec:unic3}
 
 Let $\dot{\mathcal{H}} \sim \gwc^2 (\xi, m+1)^l$, $v_1v_2 \ldots v_{m+1}v_1$ be its cycle part with root $\rho=v_1$ and $\mathcal{T}_2, \ldots ,\mathcal{T}_{m+1}$ be the i.i.d. $\gw(\xi)^l$ trees rooted at $v_2 ,\ldots,v_{m+1}$, respectively. As described in Figure \ref{fig:HandHbar}, let $H$ and $\bar{H}$ be the graph defined as follows:
 \begin{itemize}
 	\item $H$ consists of a cycle $v_2 \ldots v_{m+1}v_2$ and the trees $\mathcal{T}_j$  rooted at $v_j$, $j\in\{2,\ldots,m+1\}$. 
 	
 	\item $\bar{H}$ is obtained by removing the edge $(v_2v_{m+1})$ from $H$. In other words, it consists of a length-$(m-1)$ path $v_2 \ldots v_{m+1}$ and the trees $\mathcal{T}_j$  rooted at $v_j$, $j\in\{2,\ldots,m+1\}$. 
 \end{itemize}

 \begin{figure}
 	\centering
 	\begin{tikzpicture}[scale=0.46]
 	\draw \boundellipse{0,0}{4}{1};
 	\filldraw[black] (-3,-2.64575131106/4) circle (2pt) node[black, anchor=south]{$v_2 $};
 	\filldraw[black] (3,-2.64575131106/4) circle (2pt) node[black, anchor=south]{$v_5 $};
 	\filldraw[black] (1,-3.87298334621/4) circle (2pt) node[black, anchor=south]{$v_4 $};
 	\filldraw[black] (-1,-3.87298334621/4) circle (2pt) node[black, anchor=south]{$v_3 $};
 	\filldraw[black] (0,1) circle (2pt) node[black, anchor=north]{$v_1$};
 	
 	\draw[black] (-3,-2.64575131106/4)--(-3.4,-2.64575131106/4-4);
 	\draw[black] (-3,-2.64575131106/4)--(-3+.4,-2.64575131106/4-4);
 	\draw[black] (-3.4,-2.64575131106/4-4)--(-3+.4,-2.64575131106/4-4);
 	\filldraw[black] (-3,-2.64575131106/4-4) circle (0pt) node[black, anchor=north]{$\T_2 $};
 	
 	\draw[black] (-1,-3.87298334621/4)--(-1.8,-3.87298334621/4-1.5);
 	\draw[black] (-1,-3.87298334621/4)--(-1+.8,-3.87298334621/4-1.5);
 	\draw[black] (-1.8,-3.87298334621/4-1.5)--(-1+.8,-3.87298334621/4-1.5);
 	\filldraw[black] (-1,-3.87298334621/4-1.5) circle (0pt) node[black, anchor=north]{$\T_3 $};
 	
 	\draw[black] (1,-3.87298334621/4)--(1.4,-3.87298334621/4-1.5);
 	\draw[black] (1,-3.87298334621/4)--(1-.4,-3.87298334621/4-1.5);
 	\draw[black] (1.4,-3.87298334621/4-1.5)--(1-.4,-3.87298334621/4-1.5);
 	\filldraw[black] (1,-3.87298334621/4-1.5) circle (0pt) node[black, anchor=north]{$\T_4 $};
 	
 	\draw[black] (3,-2.64575131106/4)--(3.5,-2.64575131106/4-2.5);
 	\draw[black] (3,-2.64575131106/4)--(2.5,-2.64575131106/4-2.5);
 	\draw[black] (3.5,-2.64575131106/4-2.5)--(2.5,-2.64575131106/4-2.5);
 	\filldraw[black] (3,-2.64575131106/4-2.5) circle (0pt) node[black, anchor=north]{$\T_5 $};
 	
 	\filldraw[black] (0,1.2) circle (0pt) node[black, anchor=south]{\large{$\dot{\mathcal{H}}$} };
 	\draw[black] (5,-1) -- (7,-1);
 	\draw[black] (7,-1) -- (6.7,-.8);
 	\draw[black] (7,-1) -- (6.7,-1.2);

 	\draw \boundellipse{12,0}{4}{1};
 	\filldraw[black] (9,-2.64575131106/4) circle (2pt) node[black, anchor=south]{$v_2 $};
 	\filldraw[black] (15,-2.64575131106/4) circle (2pt) node[black, anchor=south]{$v_5 $};
 	\filldraw[black] (13,-3.87298334621/4) circle (2pt) node[black, anchor=south]{$v_4 $};
 	\filldraw[black] (11,-3.87298334621/4) circle (2pt) node[black, anchor=south]{$v_3 $};
 	
 	\draw[black] (9,-2.64575131106/4)--(9.4,-2.64575131106/4-4);
 	\draw[black] (9,-2.64575131106/4)--(9-.4,-2.64575131106/4-4);
 	\draw[black] (9-.4,-2.64575131106/4-4)--(9+.4,-2.64575131106/4-4);
 	\filldraw[black] (9,-2.64575131106/4-4) circle (0pt) node[black, anchor=north]{$\T_2 $};
 	
 	\draw[black] (11,-3.87298334621/4)--(11-.8,-3.87298334621/4-1.5);
 	\draw[black] (11,-3.87298334621/4)--(11+.8,-3.87298334621/4-1.5);
 	\draw[black] (11-.8,-3.87298334621/4-1.5)--(11+.8,-3.87298334621/4-1.5);
 	\filldraw[black] (11,-3.87298334621/4-1.5) circle (0pt) node[black, anchor=north]{$\T_3 $};
 	
 	\draw[black] (13,-3.87298334621/4)--(13.4,-3.87298334621/4-1.5);
 	\draw[black] (13,-3.87298334621/4)--(13-.4,-3.87298334621/4-1.5);
 	\draw[black] (13.4,-3.87298334621/4-1.5)--(13-.4,-3.87298334621/4-1.5);
 	\filldraw[black] (13,-3.87298334621/4-1.5) circle (0pt) node[black, anchor=north]{$\T_4 $};
 	
 	\draw[black] (15,-2.64575131106/4)--(15.5,-2.64575131106/4-2.5);
 	\draw[black] (15,-2.64575131106/4)--(14.5,-2.64575131106/4-2.5);
 	\draw[black] (15.5,-2.64575131106/4-2.5)--(14.5,-2.64575131106/4-2.5);
 	\filldraw[black] (15,-2.64575131106/4-2.5) circle (0pt) node[black, anchor=north]{$\T_5 $};
 	
 	\filldraw[black] (12,1.2) circle (0pt) node[black, anchor=south]{\large{$H$} };

 	
 	\draw[black] (18,-0.5) -- (22.5,-0.5);
 	
 	\filldraw[black] (20.25,0.7) circle (0pt) node[black,anchor=south]{\large{$\bar{H}$}};
 	\filldraw[black] (18,-0.5) circle (2pt) node[black,anchor=south]{$v_2$};
 	\filldraw[black] (19.5,-0.5) circle (2pt) node[black,anchor=south]{$v_3$};
 	\filldraw[black] (21,-0.5) circle (2pt) node[black,anchor=south]{$v_4$};
 	\filldraw[black] (22.5,-0.5) circle (2pt) node[black,anchor=south]{$v_5$};
 	
 	\draw[black] (22.5,-0.5)--(23,-0.5-2.5);
 	\draw[black] (22.5,-0.5)--(22,-0.5-2.5);
 	\draw[black] (23,-0.5-2.5)--(22,-0.5-2.5);
 	\filldraw[black] (22.5,-0.5-2.5) circle (0pt) node[black, anchor=north]{$\T_5 $};


 	\draw[black] (19.5,-0.5)--(19.5-.8,-0.5-1.5);
 	\draw[black] (19.5,-0.5)--(19.5+.8,-0.5-1.5);
 	\draw[black] (19.5-.8,-0.5-1.5)--(19.5+.8,-0.5-1.5);
 	\filldraw[black] (19.5,-0.5-1.5) circle (0pt) node[black, anchor=north]{$\T_3 $};
 	
 	\draw[black] (21,-0.5)--(21-.4,-0.5-1.5);
 	\draw[black] (21,-0.5)--(21+.4,-0.5-1.5);
 	\draw[black] (21-.4,-0.5-1.5)--(21+.4,-0.5-1.5);
 	\filldraw[black] (21,-0.5-1.5) circle (0pt) node[black, anchor=north]{$T_4 $};
 	
 	\draw[black] (18,-.5)--(17.6,-4.5)--(18.4,-4.5)--(18,-.5);
 	\filldraw[black] (18,-4.5) circle (0pt) node[black,anchor=north]{$\T_2$};
 	\end{tikzpicture}
 	\caption{A description of $\dot{\mathcal{H}}\sim \gwc^2(\xi,m+1)^l$, $H$ and $\bar{H}$ in Section \ref{subsubsec:unic3} with $m=4$.}\label{fig:HandHbar}
 \end{figure}
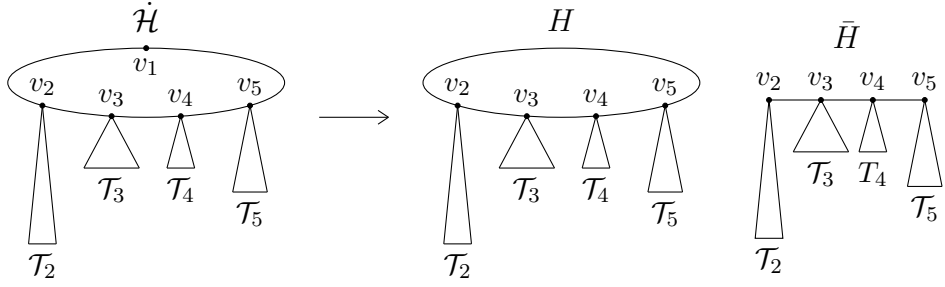

 Also, let $R_j(\bar{H})$ (resp., $\bar{M}^l_j(\bar{H})$)  denote the expected survival time (resp., expected total infections at leaves) of the (usual) contact process $\cp^\lambda(\bar{H};\one_{v_j})$. 
 
 In the process $\cp^\lambda_{v_1} (\dot{\mathcal{H}};\zero)$, the permanently infected root $v_1$ infects $v_2$ and $v_{m+1}$ with rate $\lambda$ each. Hence, the process can be interpreted similarly as in Definition \ref{def:decomposition}:
 \begin{itemize}
 	\item At rate $2\lambda$, initiate a copy of either $\cp^\lambda(\bar{H};\one_{v_2})$ or $\cp^\lambda(\bar{H};\one_{v_{m+1}})$, each chosen with probability $\frac{1}{2}$ and whose event times are coupled with $\cp^\lambda_{v_1}(\dot{ \mathcal{H}},\zero)$.
 \end{itemize}
 Therefore, for $S(\dot{\mathcal{H}}), M^l(\dot{\mathcal{H}})$ defined as Definition \ref{def:SMdef gwc2}, the argument from  Proposition \ref{prop:S vs R} and Corollary \ref{cor:MvsMprime} implies
 \begin{equation}\label{eq:gwc2 SvsR}
 \frac{S(\dot{\mathcal{H}})}{1+2\lambda S(\dot{\mathcal{H}})} \leq \frac{R_2(\bar{H})+R_{m+1}(\bar{H})}{2}, \quad
 \textnormal{and} \quad
 \frac{M^l(\dot{\mathcal{H}})}{1+2\lambda S(\dot{\mathcal{H}})} 
 \leq \frac{\bar{M}^l_2(\bar{H} )+ \bar{M}^l_{m+1}(\bar{H})}{2}.
 \end{equation}
 
 Note that $H\sim \gwc^1(\xi,m)^l$, and let $S_i(H)$ and $M^l_i(H)$ be defined as Definition \ref{def:SMunicyclic}, with respect to the root-added process $\cp^\lambda_{v_i^+}(H;\one_{v_i})$. Also, we set 
 $$S(H)= \frac{S_2(H)+ S_{m+1}(H)}{2}, \quad \textnormal{and} \quad M^l(H) = \frac{M^l_2(H) +M^l_{m+1}(H)}{2}. $$ 
 Then, since $H'\subset H$, we use $R_i(H')\leq S_i(H)$ and $M'_i(H') \leq M_i(H)$ in (\ref{eq:gwc2 SvsR}) to see that
 \begin{equation*}
 S(\dot{\mathcal{H}})\leq \frac{ S(H)}{1-2\lambda S(H)}, \quad
 \textnormal{and} \quad
 \frac{M^l(\dot{\mathcal{H}})}{1+2\lambda S(\dot{\mathcal{H}})} 
 \leq M^l(H).
 \end{equation*}
 Now, we go through the same computations as (\ref{eq:Sunic tail initialcase}) and (\ref{eq:Munic tail initialcase}). Namely, we deduce that for $3/\ep \leq  t\leq d^{\frac{1}{10}}$,
 \begin{equation*}
 \Pgw \left(S(\dot{\mathcal{H}}) \geq t \right)
 \leq 
 \Pgw \left(S(H) \geq \left(1-d^{-\frac{3}{4}} \right)t \right)
 \leq 4 t^{-\sqrt{d}} (\log t)^{-2},
 \end{equation*}
 since $(1-d^{-\frac{3}{4}})^{-\sqrt{d}}<\frac{5}{4}$ and $|\log(1-d^{-\frac{3}{4}})| \leq \frac{1}{100} \log t$. For the same reason, we obtain that
 \begin{equation*}
 \begin{split}
 \Pgw \left(M^l(\dot{\mathcal{H}}) \geq \left(1-\frac{\ep}{10}\right)^l t \right)
 &\leq\Pgw\left(M^l(H) \geq \left(1-d^{-\frac{3}{4}} \right) \left(1-\frac{\ep}{10} \right)^l t \right)\\
 &\;\;+\Pgw \left( S(H) \geq d^{\frac{1}{5}} \right),
 \end{split}
 \end{equation*}
 and the r.h.s. is bounded by $4t^{-\sqrt{d}} (\log t)^{-2}$. Hence, we finish the proof for Step 3-1.

 \subsubsection{Proof of Proposition \ref{prop:SMunicyclic} and Corollary \ref{cor:SMgwc2}, Step 3-2}\label{subsubsec:unic4}
 
 Let $ \dot{\mathcal{H}}\sim \gwc^2(\xi,m+1)^l$ and let $\mathcal{T}_2,\ldots , \mathcal{T}_{m+1}$ as above. Consider the graphs $H_2$ and $H_{m+1}$ defined as follows.
 \begin{itemize}
 	\item $H_2$ consists of a cycle $v_2\ldots v_{m+1}$ and the trees $\mathcal{T}_3,\ldots,\mathcal{T}_{m+1}$ rooted at $v_3,\ldots,v_{m+1}$.  The vertex $\rho=v_2$ is designated as the root.
 	
 	\item $H_{m+1}$ consists of a cycle $v_2 \ldots v_{m+1}$ and the trees $\mathcal{T}_2,\ldots,\mathcal{T}_m$ rooted at $v_2,\ldots,v_m$. The vertex $\rho=v_{m+1}$ is designated as the root.
 \end{itemize}		
 
 Note that $H_2 \sim \gwc^2(\xi,m)^l$, and we let $S_3(H_2), S_{m+1}(H_2)$ as Definition \ref{def:SMdef gwc2} with $S(H_2) := \frac{1}{2}(S_3(H_2)+ S_{m+1}(H_2))$. Similarly, define $M^l(H_2)$, $S(H_{m+1})$ and $M^l(H_{m+1})$. Further, let $D_2$ (resp. $D_{m+1}$) be the degree of $v_2$ in the tree $\mathcal{T}_2$, and denote the subtrees at each child of $v_2$ in $\mathcal{T}_2$ by $T^2_i,$ $i\in[D_2]$. Define $D_{m+1}$ and $T^{m+1}_i$, $i\in[D_{m+1}]$ analogously with respect to $\mathcal{T}_{m+1}$.
 
 Then, an analogous idea as Lemma \ref{lem:gwc1 recursion atypical} shows that
 \begin{equation}\label{eq:Srecursion gwc2 atypical}
 \begin{split}
 S_2(\dot{\mathcal{H}}) 
 &\leq
 (1+2\lambda S(H_2)) \prod_{i=1}^{D_2} (1+\lambda S(T^2_i));\\
 S_{m+1} (\dot{\mathcal{H}}) 
 &\leq
 (1+2\lambda S(H_{m+1})) \prod_{i=1}^{D_{m+1}} (1+\lambda S(T^{m+1}_i)).
 \end{split}
 \end{equation}	
 (Indeed, this was also proven in \cite{BNNASurvival}, 	Lemma 4.11, and its verification is based on the same idea as Proposition \ref{prop:Srecursion atypical}.) Note that we are assuming the tail probabilities for $S(H_2)$ and $S(H_{m+1}) $ satisfy (\ref{eq:gwc2 tail:in prop}), while those for $S(T_i)$ are given by Theorem \ref{thm:tail bound of S}. Based on this observation, we can follow the same analysis done in Lemma \ref{lem:recursion principle atypical}, Corollary \ref{cor:recursion principle atypical} and Sections \ref{subsubsection2}, \ref{subsubsection3} to see that for $t\geq d^{\frac{1}{10}},$
 \begin{equation*}
 \begin{split}
 \Pgw \left(S(\dot{\mathcal{H}}) \geq t \right)
 &\leq 
 \Pgw \left((1+2\lambda S(H_2)) \prod_{i=1}^{D_2} (1+\lambda S(T^2_i)) \geq t\right)\\
 &\;\;+
 \Pgw \left(	(1+2\lambda S(H_{m+1})) \prod_{i=1}^{D_{m+1}} (1+\lambda S(T^{m+1}_i))  \geq t \right)\\
 &\leq 
 2t^{-\sqrt{d}} (\log t)^{-2} + 2 t^{-\sqrt{d}} (\log t)^{-2} = 4t^{-\sqrt{d}} (\log t)^{-2}.
 \end{split}
 \end{equation*}
 
 The logic to derive (\ref{eq:gwc2 tail:in prop}) for $M^l(\dot{\mathcal{H}})$ is similar. We first derive the recursive inequalities, based on (\ref{eq:Srecursion gwc2 atypical}) and the ideas from Proposition \ref{prop:Mrecursion atypcial}, We have 
 \begin{equation*}
 \begin{split}
 M^l_a(\dot{\mathcal{H}}) 
 \leq 
 2\lambda M^l(H_a) \prod_{i=1}^{D_a} (1+\lambda S(T^a_i))
 + \lambda(1+2\lambda S(H_a))\sum_{i=1}^{D_a} M^{l-1}(T^a_i) \prod_{\substack{j=1 \\ j\neq i}}^{D_a} (1+\lambda S(T^a_j)),
 \end{split}
 \end{equation*}
 for $a\in \{2, m+1\}$. As done in Section \ref{subsubsec:M2}, we set 
 \begin{equation*}
 T^a_0:= H_a, \quad \widetilde{M}^a_i := \left(1-\frac{\ep}{10} \right)^{-(l-1)} M^{l-1}(T^a_i),\quad \textnormal{and} \quad 
 W^a_i := \widetilde{M}^a_i \vee S(T^a_i),
 \end{equation*}
 and derive from the above inequality that with $\widetilde{M}_a = (1-\frac{\ep}{10})^{-l} M_a(\dot{\mathcal{H}})$,
 \begin{equation*}
 \left(1-\frac{\ep}{10} \right)\widetilde{M}_a \leq \prod_{i=0}^{D_a} (1+ 4\lambda W^a_i), \quad a\in \{2,\; m+1\}. 
 \end{equation*}
 Then, the rest of the proof goes the same as Section \ref{subsubsec:unic2}, yielding that
 \begin{equation*}
 \Pgw \left(M^l(\dot{\mathcal{H}}) \geq \left(1-\frac{\ep}{10} \right)^l t \right) 
 \leq
 \sum_{a\in\{2,m+1\} } \Pgw\left(2\prod_{i=0}^{D_a} (1+ 4\lambda W^a_i) \geq t \right) \leq 4t^{-\sqrt{d}} (\log t)^{-2},
 \end{equation*}
 for all $t\geq d^{\frac{1}{10}}$, hence establishing Step 3-2. Combining the argument in Sections \ref{subsubsec:unic1}---\ref{subsubsec:unic4}, we conclude the proof of Proposition \ref{prop:SMunicyclic} and Corollary \ref{cor:SMgwc2}. \qed

 \section{Proof of technical lemmas}

 \subsection{Remaining proof of Lemma \ref{lem:recursion for F}}\label{subsec:pfof recursion for F}
 
 We finish the proof of Lemma \ref{lem:recursion for F} by establishing (\ref{eq:unicycF tail S}) and (\ref{eq:unicycF tail M}). The ideas are very similar as in the proofs of Theorems \ref{thm:tail bound of S} and \ref{thm:M tail bd}.

 \begin{proof}[Remaining proof of Lemma \ref{lem:recursion for F}]
 	  Let $\mathcal{F}^{m,l}$, $\{v_j \}_{j=1}^m$ and $\{\mathcal{T}_j \}_{j=1}^m$  be as in the definition of $\mathcal{F}^{m,l}$. Further, recall the definitions of $D$, $\{u_i\}_{i=1}^D$, $\{T_i\}_{i=1}^D$ and $F$ from the proof of Lemma \ref{lem:recursion for F}. 
 	  
 	  For any fixed $l$, assume that (\ref{eq:unicycF tail S}) and (\ref{eq:unicycF tail M}) are true for $m$. We show that under this assumption, they are true for $m+1$. (Note that the case $m=1$ is trivial, since $\mathcal{F}^{1,l}$ is a single vertex.)
 	  We begin with presenting the recursion inequalities which follow from straight-forward generalizations of previous propositions.
 	  
 	  We first have 
 	  \begin{equation*}
 	  \begin{split}
 	  S(\mathcal{F}^{m,l}) 
 	  &\leq \frac{1+\lambda \left(S(F) +\sum_{i=1}^D S(T_i \right)}{1-\lambda -2\lambda^2 \left( S(F)+\sum_{i=1}^DS(T_i) \right)};\\
 	  M(\mathcal{F}^{m,l})
 	  &\leq
 	  \frac{\lambda \left(M(F)+ \sum_{i=1}^D M(T_i) \right)}{1-\lambda-2\lambda^2 \left(S(F) +\sum_{i=1}^D S(T_i) \right)},
 	  \end{split}
 	  \end{equation*}
 	  which are basically rephrasings of (\ref{eq:recursion eq typical}) and (\ref{eq:Mrecursion typical}), respectively. Based on these inequalities, we can repeat the argument from Sections \ref{subsubsection1} and \ref{subsubsec:M1} to deduce (\ref{eq:unicycF tail S}) and (\ref{eq:unicycF tail M}) for $t\leq d^{\frac{1}{10}}$ and $m+1$.
 	  
 	  For the other case, $t\geq d^{\frac{1}{10}}$, we observe that
 	  \begin{equation*}
 	  \begin{split}
 	  S(\mathcal{F}^{m,l}) 
 	  &\leq (1+\lambda S(F))\prod_{i=1}^D (1+\lambda S(T_i));\\
 	  M(\mathcal{F}^{m,l})
 	  &\leq
 	  \lambda M(F) \prod_{i=1}^D (1+\lambda S(T_i)) + \lambda(1+\lambda S(F)) \sum_{i=1}^D M(T_i) \prod_{\substack{j=1\\ j\neq i}}^{D}(1+\lambda S(T_j)),
 	  \end{split}
 	  \end{equation*}
 	  which come from (\ref{eq:recursion eq atypical}) and (\ref{eq:Mrecursion atypical}). Note that the inequality for $M(\mathcal{F}^{m,l})$ is also reminiscent of (\ref{eq:Mrecursion H atypical}). Then, we follow the logic from Sections \ref{subsubsection2}, \ref{subsubsection3}, \ref{subsubsec:M2} and \ref{subsubsec:M3} to obtain (\ref{eq:unicycF tail S}) and (\ref{eq:unicycF tail M}) for $t\geq d^{\frac{1}{10}}$ and $m+1$.
 \end{proof}

 \subsection{Proof of Claim \ref{claim:B tail large t}}\label{subsec:pfof claim}
 
 In this subsection, we finish the remaining proof of Claim \ref{claim:B tail large t}
 
 \begin{proof}[Remaining proof of Claim \ref{claim:B tail large t}]
 	Recall the definitions of $\widetilde{B}$ and $\widehat{S}$ in (\ref{eq:Brecursion:in claim}). In the setting of Claim \ref{claim:B tail large t}, we are assuming (\ref{eq:unicycF tail B}) for $m$, which is
 	\begin{equation}\label{eq:claimpf:Bassumption}
 	\Pgw \left(\widetilde{B} \geq t \right) \leq  t^{-\sqrt{d}} (\log t)^{-2}, \quad \textnormal{for all }t\geq 2.
 	\end{equation}
 	
 	The bound on $\widehat{S}$ was derived in Sections \ref{subsubsection2} and \ref{subsubsection3} which yielded that
 	\begin{equation}\label{eq:claimpf:Sassumption}
 	\Pgw \left(\widehat{S} \geq t \right) \leq t^{-\sqrt{d}} (\log t)^{-2}, \quad \textnormal{for all }t\geq d^{\frac{1}{10}}.
 	\end{equation}
 	
 	Set $\alpha = d^{\frac{1}{10}}$ and $k_0 = \lfloor \frac{\log t}{\log \alpha} \rfloor -1$. The rest of the proof is analyzing (\ref{eq:BtimesS recursion}). The assumptions (\ref{eq:claimpf:Bassumption}) and (\ref{eq:claimpf:Sassumption}) give that each of the first two terms in its r.h.s. is smaller than $\frac{1}{4} t^{-\sqrt{d}} (\log t)^{-2}$. Then, the third term can also be controlled based on observing that for $k\in [k_0]$ and $t\geq d^{\frac{1}{10}},$
 	\begin{equation}\label{eq:claimpf:beforesum}
 	\begin{split}
 	&\Pgw\left(\widehat{S} \geq t\alpha^{-k} \right) \cdot
 	\Pgw \left(\widetilde{B} \geq \frac{\alpha^{k-1} }{2}d^{\frac{1}{4}} \right)\\
 	&\leq 
 	\left(\frac{\alpha^k}{t} \cdot \frac{2}{d^{\frac{1}{4}}\alpha^{k-1}} \right)^{\sqrt{d}} \left(\frac{1}{\log\left(t\alpha^{-k} \right) } \right)^{2} \left(\frac{1}{\log \left(\frac{1}{2}\alpha^{k-1} d^{\frac{1}{4}} \right)} \right)^2\\
 	&\leq
 	\left(d^{\frac{1}{8}} t \right)^{-\sqrt{d}} 
 	\frac{1}{
 		(\log t - k\log \alpha)^2
 	}
 	\frac{1}{
 		\left(\frac{1}{8}\log d + k\log \alpha \right)^2
 	}.
 	\end{split}
 	\end{equation}
 	Observing that $\log t- k_0 \log \alpha \geq \log \alpha = \frac{1}{10} \log d$, we can apply Lemma \ref{lem:sum for induction log term} in order to control the sum of (\ref{eq:claimpf:beforesum}) over $k$. All in all, we deduce that (\ref{eq:Brecursion:in claim}) holds true for $t\geq d^{\frac{1}{10}}.$ 
 \end{proof}

 \subsection{Proof of Lemma \ref{lem:Drecursion}}\label{subsec:pfof Drecursion}
 
 The goal of this subsection is to present the proof of Lemma \ref{lem:Drecursion}. 
 
 \begin{proof}[Proof of Lemma \ref{lem:Drecursion}]
 	 We begin with establishing the recursive inequalities for $S(\mathcal{A}^{\underbar{m},l})$, $M(\mathcal{A}^{\underbar{m},l})$ and $B(\mathcal{A}^{\underbar{m},l})$. To this end, we first define some subgraphs of $\mathcal{A}^{\underbar{m},l}$ to generate the recursions. Recall the definitions of $\{v_j \}_{j=-m_1}^{m_2}$ and $\{\mathcal{T}_j \}_{j=-m_1+1}^{m_2-1}$ from the beginning of Section \ref{subsec:unic prop small t2}.
 	 \begin{itemize}
 	 	\item Let $D$ be the degree of $v_0$ inside $\mathcal{T}_0$, and let $T_1,\ldots,T_D$ be the subtrees of $\mathcal{T}_0$ rooted at each child of $v_0$ in $\mathcal{T}_0$.
 	 	
 	 	\item Let $F_1$ be the graph consists of the length-$(m_1-1)$ line $v_{-1} \ldots v_{-m_1}$ and the trees $\mathcal{T}_{-1},\ldots , \mathcal{T}_{-m_1+1}$ rooted at $v_{-1}, \ldots, v_{-m_1+1}$, respectively.
 	 	
 	 	\item Let $F_2$ be defined similarly in the other branch of $\mathcal{A}^{\underbar{m},l}$. That is, it consists of the length-$(m_2-1)$ line $v_{1} \ldots v_{m_2}$ and the trees $\mathcal{T}_{1},\ldots , \mathcal{T}_{m_2-1}$ rooted at $v_{1}, \ldots, v_{m_2-1}$, respectively.
 	 	
 	 	\item For convenience, we set $T_{D+1} := F_1$ and $T_{D+2} := F_2.$ Note that the graphs $T_1,\ldots, T_{D+2}$ are all trees.
 	 \end{itemize}
 	 
 	 Based on the same ideas as the derivation of (\ref{eq:recursion eq typical}), (\ref{eq:Mrecursion typical}) and (\ref{eq:Brecursion on F typical}), we have the first recursive inequalities as follows.
 	 \begin{equation}\label{eq:Drecursion typical}
 	 \begin{split}
 	 S(\mathcal{A}^{\underbar{m},l}) 
 	 &\leq
 	 \frac{1+ \sum_{i=1}^{D+2} S(T_i)}{
 	 	1-\lambda -2\lambda^2 \sum_{i=1}^{D+2} S(T_i)
 	 	};\quad 
 	 M(\mathcal{A}^{\underbar{m},l})
 	 \leq
 	 \frac{\lambda \sum_{i=1}^{D+2} M(T_i)}{
 	 	1-\lambda -2\lambda^2 \sum_{i=1}^{D+2} S(T_i)
 	 	}	;\\
 	 B(\mathcal{A}^{\underbar{m},l})
 	 &\leq
 	 \frac{\lambda(B(T_{D+1} )+ B(T_{D+2}) )   }{ 
 		1-\lambda -2\lambda^2 \sum_{i=1}^{D+2} S(T_i) 	}	.
 	 \end{split}
 	 \end{equation}
 	 Moreover, 
 	 the second recursive inequalities are obtained from Propositions \ref{prop:Srecursion atypical}, \ref{prop:Mrecursion atypcial} and from (\ref{eq:Brecursion on F atypical}), which read
 	 \begin{equation}\label{eq:Drecursion atyipical}
 	 \begin{split}
 	 S(\mathcal{A}^{\underbar{m},l})
 	 &\leq
 	 \prod_{i=1}^{D+2} (1+\lambda S(T_i)); \quad 
 	 M(\mathcal{A}^{\underbar{m},l})
 	 \leq
 	 \lambda \sum_{i=1}^{D+2} M(T_i) \prod_{\substack{j=1\\j\neq i}}^{D+2} (1+\lambda S(T_i));\\
 	 B(\mathcal{A}^{\underbar{m},l})
 	 &\leq
 	\lambda  \sum_{i=D+1}^{D+2} B(T_i) \prod_{\substack{j=1\\j\neq i}}^{D+2} (1+\lambda S(T_i)).
 	 \end{split}
 	 \end{equation}
 	 
 	Note that we know the tail probabilities for $S(T_i),\; i\in[D+2]$ by Theorems \ref{thm:tail bound of S}, \ref{thm:M tail bd} and Lemma \ref{lem:recursion for F}. Therefore, we can deduce (\ref{eq:unicycD tail S}) and (\ref{eq:unicycD tail M}) by following  Sections \ref{subsubsection1}---\ref{subsubsection3} and \ref{subsubsec:M1}---\ref{subsubsec:M3}. 
 	
 	To see (\ref{eq:unicycD tail B}), we rely on Lemma \ref{lem:recursion for F}. For $t\leq d^{\frac{1}{10}}$, we have the following inequality reminiscent of (\ref{eq:Brecursion on F split small t}).
 	\begin{equation*}
 	\begin{split}
 		\Pgw \left(B(\mathcal{A}^{\underbar{m},l}) \geq d^{\frac{3m_1 \wedge m_2}{4}}t \right)
 		&\leq
 		\P\left(D \geq \left(1+\frac{\ep}{6} \right)d \right)
 		+
 		\Pgw \left( \sum_{i=1}^{D+2} S(T_i) \geq \frac{2d}{\ep} \left(1+\frac{\ep}{3} \right) \right)\\
 		&\;\;+
 		P\left(B(F_1) \geq d^{-\frac{3}{4}(m_1-1)+\frac{1}{2} } \right)
		+
		P\left(B(F_2) \geq d^{-\frac{3}{4}(m_2-1)+\frac{1}{2} }. \right)
 	\end{split}
 	\end{equation*}
 	By estimating the terms in the r.h.s. as in Case 1 of the proof of Lemma \ref{lem:recursion for F} and Section \ref{subsubsection1}, we see that the above is smaller than $t^{-\sqrt{d}}(\log t)^{-2}$ for $t\leq d^{\frac{1}{10}}.$
 	
 	For $t\geq d^{\frac{1}{10}}$, we set 
 \begin{equation*}
 \widetilde{B}(\mathcal{A}^{\underbar{m},l}) := d^{\frac{3}{4}(m_1 \wedge m_2)} B(\mathcal{A}^{\underbar{m},l}), 
 \quad 
 \textnormal{and} \quad
 \widetilde{B}_a := d^{\frac{3}{4}(m_a-1)} B(F_a), \;\; a\in\{1,2\},
 \end{equation*}
and   observe that by (\ref{eq:Drecursion atyipical}),
 	\begin{equation*}
 \Pgw \left( \widetilde{B}(\mathcal{A}^{\underbar{m},l}) \geq t \right)
 	\leq
 	\sum_{a\in \{1,2\}}	\Pgw\left( \widetilde{B}_a \prod_{\substack{j=1\\ j\neq D+a}}^{D+2} (1+\lambda S(T_i)) 
 	\geq 
 	\frac{d^{\frac{1}{4}}t}{3} \right).
 	\end{equation*}
 	Then, the summand in the r.h.s. can be bounded by applying Claim \ref{claim:B tail large t} and implies (\ref{eq:unicycD tail B}) for $t\geq d^{\frac{1}{10}}.$
 \end{proof}

\subsection{Proof of Lemma \ref{lem:aug properties}}\label{subsec:pfof aug}

Here, we establish Lemma \ref{lem:aug properties}. Indeed, the third statement in the lemma was proven in Lemma 4.3 of \cite{BNNASurvival}. Therefore, we focus on the first two statements of the lemma.

\begin{proof}[Proof of Lemma \ref{lem:aug properties}]
	We begin with establishing the first statement. We only show
	$$d^\sharp \leq \left(1+\frac{\ep}{9} \right)d, $$
	since $\widetilde{d}^\sharp \leq (1+\frac{\ep}{9})\widetilde{d}$ follows analogously.
	 Let $Z$ be as in Definition \ref{def:aug}. We first note that $$ Z \geq \sum_k \left(1-\frac{\ep}{10} \right) p_k =1-\frac{\ep}{10}.$$ 
	Moreover, when $k_0<k_{\textnormal{max}}$, 
	$$\sum_{k\leq k_0} k \left(1-\frac{\ep}{10} \right)p_k + \sum_{k>k_0} k \sqrt{p_k} 
	\leq \sum_{k} k \left(1-\frac{\ep}{10} \right)p_k + \frac{\ep}{10} \leq \left(1-\frac{\ep}{10} \right)d + \frac{\ep}{10} 
	$$
	This gives us  $d^\sharp\leq \left(1+\frac{\ep}{9} \right)d$, for all $d$ larger than some absolute constant. In the other case when $k_0=k_{\textnormal{max}}$, we are done if $k_0 \leq (1+\frac{\ep}{10})d$. If not, the concentration condition gives that
	$$p_{k_0} \leq \exp(- c_{\ep'} d)$$
	with $\ep' = \frac{\ep}{10}$, and hence for $d$ larger than some constant depending on $\ep$ and $\mathfrak{c}$, we obtain
	$$\sum_{k<k_0} kp_k + k_0\sqrt{p_{k_0}} \leq \left(1-\frac{\ep}{10} \right) d + \ep, $$
	which again implies $d^\sharp\leq \left(1+\frac{\ep}{9} \right)d.$
	
	For the second, we only establish the first of  (\ref{eq:concentrationcond2}), since the second line follows analogously. Let $\ep' \in[\frac{\ep}{10},1]$. For any integer $k\in[(1+\ep')d , 2d]$, we have
	\begin{equation*}
	\mu^\sharp(k) \leq \frac{\sqrt{p_k}}{Z} \leq 2\exp\left(-\frac{c_{\ep'}}{2}d\right).
	\end{equation*} 
	Therefore, we obtain that
	\begin{equation*}
	\begin{split}
	\P_{D^\sharp \sim\mu^\sharp} \left( D^\sharp \geq (1+ \ep')d \right)
	&\leq \sum_{k=\lceil (1+\ep')d \rceil }^{\lfloor 2d\rfloor} 2\exp\left(-\frac{c_{\ep'}}{2}d\right)
	+
	\sum_{k\geq \lceil 2d \rceil } 2\exp\left(-\frac{c_1k}{2} \right)\\
	&\leq 2d \exp\left({-\frac{c_{\ep'}}{2}d}\right)  + \frac{2e^{-c_1 d}}{1-e^{-\frac{c_1}{2}}} 
	\leq
	\exp\left(-\frac{c_{\ep'} \wedge c_1}{3}d \right),
	\end{split}
	\end{equation*}
	where the last inequality holds for any $d$ larger than some constant $d_0(\mathfrak{c})$. Therefore, we can set $c'_{\ep'} = \frac{1}{3} (c_{\ep'} \wedge c_1)$ and obtain the first line of (\ref{eq:concentrationcond2}).
\end{proof}

 \subsection{Proof of Lemma \ref{lem:1cyc}}\label{subsec:pfof 1cyc}
 
 Lemma 4.5 of \cite{BNNASurvival} had a very similar statement as Lemma \ref{lem:1cyc}, but it was weaker in terms of determining the scope of $\gamma$. It turns out that in order to prove the stronger version, Lemma \ref{lem:1cyc},  the following bound on the size of $N(v,l)$ is needed. 
 
 \begin{lemma}\label{lem:gw size}
 	Under the setting of Lemma \ref{lem:1cyc}, let $\mu^\sharp := \mu^\sharp_\ep$ be the augmented distribution of $\mu$ and $\widetilde{\mu}^\sharp$ be its size-biased distribution. Further, let  $L_n:= \frac{1}{10} \log_d n$ and $\mathcal{T} \sim  \gw(\mu^\sharp,\widetilde{\mu}^\sharp)^{L_n}$, where $d$ is the mean of $\widetilde{\mu}$, the size-biased distribution of $\nu$. Then, there exists $d_0(\ep, \mathfrak{c})>0$ such that if $d\geq d_0$, then we have
 	\begin{equation*}
 	\P\left(|\mathcal{T}| \geq n^{\frac{1}{5}} \right) \leq \exp\left(-n^{\frac{1}{20}} \right)
 	\end{equation*}	
 \end{lemma}
 
 \begin{proof}
 	Let $c'_1 \in \mathfrak{c}'$ be the constant in Lemma \ref{lem:aug properties}. Then, we first note that the second statement of Lemma \ref{lem:aug properties} tells us the following: 
 	\begin{equation}\label{eq:Dsharp exp moment}
 	\E_{D\sim \widetilde{\mu}^\sharp} \left[\exp\left(\frac{c'_1}{2} D \right) \right]
 	\leq
 	\exp\left(2d \right) + \sum_{k\geq 2d} e^{\frac{c'_1k}{2}} e^{-c'_1k} \leq e^{3d}, 
 	\end{equation}
 	where the last inequality holds if $d$ is larger than some constant depending on $c_1'.$ Moreover, it is straight-forward to see that the same bound holds for ${\mu}^\sharp$ as well. We also note that 
 	\begin{equation}\label{eq:dpower}
 	\left(\frac{6d}{c'_0}\right)^{L_n} \leq n^{\frac{1}{9}} ,
 	\end{equation}
 	 if $d$ is larger than some constant depending on $c'_0$.
 	
 	Let $Y_h$ be the number of vertices at depth $h$ of $\mathcal{T}$. Further, let $D \sim \widetilde{\mu}^\sharp$ and set $c'' = \frac{c'_0}{2}$. Note that H\"older's inequality implies $\E e^{aX} \leq (\E e^{X})^a$ for $0\leq a\leq 1$. This gives that
 	\begin{equation*}
 	\begin{split}
 	\E\left[ \exp\left( \left(\frac{c''}{3d}\right)^hY_h  \right) \right]
 	&= \E\left[ \E_{D}\left[ \exp\left(\left(\frac{c''}{3d}\right)^hD \right) \right]^{Y_{h-1}} \right]\\
 	&\leq
 	\E \left[ \exp\left\{ 3d \cdot \frac{(c'')^{h-1}}{(3d)^h} Y_{h-1} \right\}  \right]
 	=
 	\E\left[\exp\left(\left(\frac{c''}{3d} \right)^{h-1} Y_h \right) \right].
 	\end{split}
 	\end{equation*}
 	Hence, we obtain that
 	\begin{equation}\label{eq:Yhexpmoment}
 	\E \left[\exp \left( \left(\frac{c''}{3d}\right)^{L_n} \sum_{h=0}^{L_n} Y_h \right) \right]
 	\leq
 	\sum_{h=0}^{L_n} \left(\frac{c''}{3d}\right)^h \leq 2,
 	\end{equation}
 	if $d$ is larger than some constant depending on $\mathfrak{c}$. (Note that (\ref{eq:Dsharp exp moment}) holds for $\mu^\sharp$ as well.) Therefore, combining (\ref{eq:Yhexpmoment}) with (\ref{eq:dpower}), we deduce that
 	\begin{equation*}
 	\begin{split}
 	\P \left(\sum_{h=0}^{L_n} Y_h \geq n^{\frac{1}{5}} \right)
 	&\leq
 	\exp\left(-\left(\frac{c''}{3d}\right)^{L_n}n^{\frac{1}{5}}  \right)
 	\times
 		\E \left[\exp \left( \left(\frac{c''}{3d}\right)^{L_n} \sum_{h=0}^{L_n} Y_h \right) \right]\\
 		&\leq
 		\exp\left(-n^{\frac{1}{20}} \right).
 	\end{split}
 	\end{equation*}
 \end{proof}
 
 \begin{proof}[Proof of Lemma \ref{lem:couplinglocalnbd}]
 In Lemma 4.5, \cite{BNNASurvival},	its proof reduces to obtaining tail probability bounds of $\sum_{h=0}^{L_n} Y_h$ for $Y_h$ as above.  Having Lemma \ref{lem:gw size} in hand, following the proof of Lemma 4.5, \cite{BNNASurvival} deduces Lemma \ref{lem:couplinglocalnbd}.
 \end{proof}

 \end{document}